\documentclass{ws-m3as}
\usepackage[super]{cite}
\usepackage{url}
\usepackage{cancel}
\usepackage{lmodern}

\usepackage{hyperref}
\usepackage{algorithm}
\usepackage{algpseudocode}
\DeclareSymbolFont{rsfs}{U}{rsfs}{m}{n}
\DeclareSymbolFontAlphabet{\mathscrsfs}{rsfs}
\newcommand{\mscr}[1]{\mathscrsfs{#1}}

\usepackage[caption=false]{subfig}
\usepackage{pgfplots}
\allowdisplaybreaks

\newcommand{\tnorm}[1]{{\left\vert\kern-0.25ex\left\vert\kern-0.25ex\left\vert #1 
		\right\vert\kern-0.25ex\right\vert\kern-0.25ex\right\vert}}

\usepackage{graphicx,epstopdf}
\hypersetup{hypertexnames=false}
\usepackage{amsopn}

\usepackage[showonlyrefs]{mathtools}
\mathtoolsset{showonlyrefs}
\usepackage{xspace}
\usepackage{bold-extra}
\usepackage[most]{tcolorbox}
\usepackage{xcolor}
\colorlet{texcscolor}{blue!50!black}
\colorlet{texemcolor}{red!70!black}
\colorlet{texpreamble}{red!70!black}

\colorlet{codebackground}{black!25!white!25}

\allowdisplaybreaks

\newcommand{\R}{{\mathbb R}}
\def\epsilon{\varepsilon}
\def\hat{\widehat}

\def\curl{\mathrm{\textbf{curl}}}
\renewcommand{\tau}{\phi}

\newcommand\dE{{\partial \E}}

\def\vb{{\bf v}}

\newcommand\bn{\boldsymbol{n}}

\newcommand{\blue}[1]{\textcolor{black}{#1}}

\definecolor{myMagenta}{rgb}{1,0,1}

\newcounter{ind}

\def\epsilon{\varepsilon}
\def\hat{\widehat}

\def\w{\boldsymbol{w}}
\def\f{\boldsymbol{f}}
\def\x{\boldsymbol{x}}
\def\u{\boldsymbol{u}}
\def\v{\boldsymbol{v}}

\def\y{\boldsymbol{y}}
\def\n{\boldsymbol{n}}

\def\V{\boldsymbol{V}} 
\def\H{\boldsymbol{H}}
\def\bbeta{\boldsymbol{\beta}}
\def\L{\boldsymbol{L}}
\def\E{K}

\def\wh{\boldsymbol{w_h}}
\def\uh{\boldsymbol{u_h}}
\def\uone{\boldsymbol{u_1}}
\def\utwo{\boldsymbol{u_2}}

\def\yh{\boldsymbol{y_h}}
\def\n{\boldsymbol{n}}

\begin{document}
\markboth{A. Khan, D. Mora, A. Silgado and H. Singh}{$C^1$-Conforming VEM for the Optimal Control of Oseen Equations}
%
\catchline{}{}{}{}{}
%
\title{$\boldsymbol{C^1}$ Virtual Element Methods for the Optimal Control of Oseen Equations with Stream-Function Formulation}

\author{Arbaz Khan\footnote{Corresponding author. A. Khan was supported by ANRF ARG-Matrics ANRF/ARGM/2025/001949/MTR.} \ \and Harpal Singh\footnote{``H. Singh was supported by CSIR, India with Award no. 09/0143(16048)/2022-EMR-I."}}

\address{Department of Mathematics, Indian Institute of Technology Roorkee\\
	 Roorkee 247667, India,
\textsuperscript{\textup{*}}arbaz@ma.iitr.ac.in, \textsuperscript{\textup{\dag}}harpal\_s@ma.iitr.ac.in}

\author{David Mora\footnote{``D. Mora was partially supported by the National Agency
	for Research and Development, ANID-Chile through FONDECYT project 1261217,
	and by project Centro de Modelamiento Matem\'atico (CMM), FB210005,
	BASAL funds for centers of excellence."}}

\address{
GIMNAP, Departamento de Matem\'atica, Universidad del B\'io-B\'io,
	  Concepci\'on, Chile \\and CI$^2$MA, Universidad de Concepci\'on, 	 Concepci\'on, Chile,
\textsuperscript{\textup{\ddag}}dmora@ubiobio.cl}

\author{Alberth Silgado\footnote{``A. Silgado was partially supported by the PRIN 2022 grant No. 2022MBY5JM, \emph{FREYA – Fault REactivation: a hYbrid numerical Approach}, funded by the Italian Ministry of Universities and Research (MUR) and the European Union through NextGenerationEU (M4C2). This research was carried out primarily during A. Silgado’s postdoctoral fellowship at the Università degli Studi di Bari.''}}
\address{
Centro de Ciencias Exactas, GIMNAP, Departamento de Ciencias B\'asicas, Facultad de Ciencias,\\  Universidad
	 del B\'io-B\'io, Chill\'an, Casilla 447, Chile,
	  asilgado@ubiobio.cl}
\maketitle


\begin{abstract}
This work analyzes a $C^1$-conforming virtual element method for the optimal control of generalized Oseen equations, formulated using the stream-function on simply connected polygonal domains. The method is based on a conforming discretization in $H^2$ of arbitrary order $k \ge 2$. The discrete scheme considers an approximation of the state and co-state equations which involves a specifically designed operator for the convective term and a \textit{variational discretization} of the control variable. We derive {quasi}-optimal a priori error estimates in the $H^1$ and $H^2$-norms for the stream-function, and in the $L^2$-norm for the control, under standard assumptions on the family of polygonal meshes. The proposed method recovers the fluid velocity, vorticity, and pressure using suitable projections. Finally, we report several numerical experiments to support the theoretical findings using different polygonal meshes including the cavity problem. We highlight that this analysis also covers OCPs governed by Brinkman and Stokes equation.
\end{abstract}



\section{Introduction} \label{Introduction.} \setcounter{equation}{0}
Partial differential equations (PDEs) are important to model various phenomena in physics, science, and engineering. Over the years, many numerical methods have been developed to solve PDEs, such as the finite difference method (FDM), the finite element method (FEM), and the finite volume method (FVM). All these methods are well-established with efficient implementations and strong theoretical results. In 2013, Beir\~ao da Veiga et al. introduced the Virtual Element Method (VEM),\cite{VEM1} to solve PDEs for many practical applications. The VEM is built on the ideas of the Mimetic Finite Differences and integrates them into the Galerkin framework, which leads to a generalized version of FEM. VEM stands out for its flexibility in handling polytopal meshes, including cases where traditional FEM faces challenges. {The polytopal meshes provide a better meshing of domain, automatically includes hanging nodes, give more efficient adaptive refinement, accommodate both concave and convex elements and provides more local rotational symmetries. Some key features include no explicit basis construction (which are critical in standard FEM), supports higher-order approximations and handles complex geometries.} Due to its versatility, VEM has inspired a lot of new developments in the field to enhance its applicability (cf Refs.~\refcite{VEMBOOK}, \refcite{MR4586821}).

Optimal control problems (OCPs) for fluid flow governed by the Stokes, Brinkman, Oseen, and Navier–Stokes equations have many practical applications. Reducing drag, controlling turbulence, managing blood flow, vehicle aerodynamics, among others, are some of the well-known real-life applications governed by these equations. In particular, the Oseen model plays a crucial role in fluid dynamics as a linearized approximation of the Navier-Stokes equations, making it essential for analyzing slow viscous flows with convection effects. It serves as a bridge between the Stokes and Navier-Stokes models, incorporating both diffusive and convective contributions. Additionally, it relates to the Brinkman model, which extends Stokes flow to porous media, highlighting its broad applicability in various physical and engineering contexts.  Numerical approximation of OCPs is challenging because of the coupling of control, state, and co-state equations. There is a wealth of research on solving such problems using conforming and non-conforming FEMs given in Refs.~\refcite{AAER1}, \refcite{MR2302057}, \refcite{CNS7}, \refcite{MR555382}, \refcite{MR3284269},  \refcite{LEADG}, \refcite{CMEYER}, \refcite{CNS2}, \refcite{HSAK1}, \refcite{HSAK3} and the references cited therein. In recent years, VEM has gained significant interest and has been used to solve various problems in fluid and solid mechanics, as seen in Refs. \refcite{4NLP}, \refcite{MR5111615}, \refcite{VEMREV}, \refcite{VEMGSOC}, \refcite{VSOSEEN}, \refcite{STABVEM}, \refcite{VEMSE}, \refcite{VEMFOUR}, \refcite{MFNS} and \refcite{MR4974425}. VEMs are commonly used for OCPs, with notable contributions coming from Refs.~\refcite{C1VEMCON}, \refcite{gudi2026virtual}, \refcite{CVEM2}, \refcite{SINGH2026521}, \refcite{JTVEM4}, \refcite{CVEM1} and \refcite{OCPBH} for different problems.
\subsection{Stream-Function Formulations}
A variety of distinct formulations have been proposed and analyzed for the Oseen equations; see Refs. \refcite{HHOO}, \refcite{VABAM}, \refcite{VANAYA}, \refcite{MR3577937}, \refcite{VSOSEEN}, \refcite{MLUBE}, \refcite{BURE}, \refcite{CCHEN}, \refcite{AGO} and \refcite{mishra2026}. $C^1$-VEM has been discussed in Refs. \refcite{NCVEMOSEEN}, \refcite{C1VEMNS}, \refcite{STREAMSTOKES}, \refcite{MORABRINKMAN}, \refcite{MORAOSEEN} and \refcite{MR4982456} for many problems. Recently, Mora et al. developed a $C^1$-VEM of high order for the generalized Oseen problem written in terms of the stream-function\cite{MORAOSEEN}. This formulation offers several benefits when solving fluid flow problems, such as \cite{MORAOSEEN,MSNSC1}:
\begin{itemize}
	\item The conforming discretization of a fourth-order PDE demands $C^{1}$-continuity, which is usually difficult due to the high number of degrees of freedom required
	(cf. \cite[Section~6.1]{PGC}). However, VEM provides an attractive
	and competitive approach to achieve this.
	\item {The formulation leads to a reduced system, since it involves only the scalar stream-function instead of the velocity–pressure pair. This results in a smaller global linear system}.
	\item Incompressibility condition on the velocity field is satisfied automatically and velocity, vorticity, and pressure can be recovered easily from the discrete stream-function. 
	\item Computations demonstrate the accuracy and robustness of this approach in capturing complex flow patterns and ensuring reliable results and better flow visualization.
\end{itemize} 
\subsection{Fundamental Contributions}
Brenner et al. introduced a $C^{1}$-VEM for an elliptic distributed OCP with pointwise state constraints in Ref. \refcite{C1VEMCON}. We remark that OCPs for fluid flows using the stream-function formulation and $C^{1}$-VEM remain unexplored in the existing literature. To the best of our knowledge, this paper represents the first design and analysis of a numerical scheme for distributed optimal control of generalized Oseen equations in stream-function form. Next, we highlight the main contributions of this work:
\begin{itemize}
	\item \textit{\textbf{Well-posedness:}} 
	We present a non-symmetric VEM discretization of general order $k \geq 2$ for solving OCPs for the Oseen equations, with the formulation written in terms of the stream-function, which involves a fourth-order PDE. We use $\textit{``optimize then discretize"}$ strategy; see Remark~4.1\cite{AAER1} with a \textit{variational discretization} of the control variable. The reduced scalar problem is analyzed using the Lax-Milgram theorem.
	\item \textit{\textbf{Advantages:}}
	Building on the VEM framework, we develop a $C^1$-conforming scheme for the stream-function formulation of the Oseen control problem. The use of polygonal meshes ensures geometric flexibility and efficiency, while velocity, vorticity, and pressure are seamlessly recovered via projection-based post-processing. A key contribution of this work is the introduction of a carefully designed operator for handling the convective term (cf. Lemma~\ref{Lemma 5.3.1.1.1}). Unlike standard approaches, this operator is tailored to the specific structure of our problem, and its formulation is critical in ensuring improved theoretical and numerical performance. 
	This operator is instrumental in establishing optimal $H^1$-norm error estimates for the stream-function. These estimates play a central role in guaranteeing optimal convergence rates for the control variable, which is a primary quantity of interest in the optimal control setting.
	The improvement achieved here stems from a novel aspect of our formulation: the strategic use of the $L^2$-orthogonality property. This property is built directly into the variational formulation. By doing so, we are able to better align the discrete and continuous structures of the problem, leading to more accurate and stable numerical approximations.
	\item \textit{\textbf{A Priori Error Analysis:}}
	{
		We develop an error analysis for the proposed $C^1$-conforming VEM of arbitrary order $k \geq 2$. Under suitable mesh and regularity assumptions, we establish quasi-optimal error estimates for the stream-function in the $H^2$ and $H^1$-norms in Theorem~\ref{Theorem 5.4.2.3}. In particular, for the lowest-order case $k=2$, we obtain an optimal convergence rate of $O(h^2)$ in the $L^2$-norm for the control variable if \(\psi, \vartheta \in H^3(\Omega)\). For \(k=3\), the convergence rate improves to \( O(h^3) \), assuming that the solution of state and co-state stream-function problems \(\psi, \vartheta \in H^4(\Omega)\) with suitable assumptions (see Remark~\ref{remark:aditional:regu2}f).
	} This additional regularity inherently satisfies the incompressibility condition and enhances the analysis as compared to previous results that assume velocity to be in \( H^3(\Omega)\). We also highlight that the present analysis also covers OCPs governed by Brinkman equation $(\bbeta=\boldsymbol{0})$ and Stokes equation $(\bbeta=\boldsymbol{0}, \sigma=0)$.
	\item \textit{\textbf{Velocity, Vorticity, and Pressure Recovery:}} Following the algorithm recently introduced in Refs.~\refcite{MORABRINKMAN} and \refcite{MORAOSEEN}, it is possible to recover the fluid velocity, vorticity, and pressure fields within the OCP framework. 
\end{itemize}
\subsection{Structure of the Paper}
The rest of the paper is structured as follows: in the remainder of this section, we introduce the function spaces required for the analysis. In Sec.~\ref{Model problem}, we state the OCP and discuss the existence and uniqueness of the optimal control with the stream-function formulation. Sec.~\ref{Virtual Element Methods} is devoted to the construction of a VEM discrete problem and its well-posedness. We provide a comprehensive a priori error analysis for all the variables in Sec.~\ref{Error Analysis}.  Sec.~\ref{Numerical Experiments.} covers the numerical tests with different polygonal meshes to validate the theoretical results. Moreover, additional examples of cases not covered by our theory are also presented. Lastly, in Section~\ref{SEC:CONCLUSIONS}, we provide concluding remarks and briefly discuss possible future developments of this work.
\subsection{Function Spaces and Notations} \label{Preliminaries and continuous formulation.} \setcounter{equation}{0}
Let $\Omega \subset \R^2$ be a simply connected polygonal domain with boundary $\Gamma := \partial \Omega$. For any open and bounded set $\mathfrak{D} \subset \R^2$ or $\mathfrak{D} \subset \R$, let $W^{s,p}(\mathfrak{D})$ denote the usual Sobolev spaces with associated norms $\|\cdot\|_{W^{s,p}(\mathfrak{D})}$ and semi-norms $|\cdot|_{W^{s,p}(\mathfrak{D})}$, for $s \ge0$ and  $1 \le p \le \infty$. For the case $s=0$, we define $W^{0,p}(\mathfrak{D}):=L^p(\mathfrak{D})$ with norm $\|\cdot\|_{p,\mathfrak{D}}$, and when $p=2$, we define $W^{s,2}(\mathfrak{D}):=H^{s}(\mathfrak{D})$ with norm $\|\cdot\|_{s,\mathfrak{D}}$. $(\cdot,\cdot)_{\mathfrak{D}}$ denotes standard $L^2$ inner-product. For the case $\mathfrak{D} = \Omega$, we avoid writing the domain in the subscript. Bold-face notations denote the vector-valued spaces. We define the following spaces:
\begin{align}
\label{vspace} \V &:= \H_{0}^{1}(\Omega) = \big\{\v \in \H^{1}(\Omega) :  \v|_{\Gamma} = \mathbf{0} \big\}, &&
	\boldsymbol{V_{0}} := \big\{\v \in \V: \nabla \cdot \v =0 \quad \text{in} \quad  \Omega \big\},\\
	Q &:= L_{0}^{2}(\Omega) = \big\{v \in L^2(\Omega) : (v,1) = 0 \big\},&& 
	\label{Hspace} \mscr{H} := \big\{v \in H^{1}(\Omega) : (v,1) = 0\big\}.
\end{align} 

\blue{Throughout this paper, we sometimes use $a \lesssim b$ (resp., $a \gtrsim b$) to denote $a \le C b$ (resp., $a \ge C b$), where $C$ is a positive constant independent of the mesh. We write $a \simeq b$ if both inequalities hold simultaneously.}
\section{Optimal Control Problem} \label{Model problem}
Consider the optimization problem that consists of finding the distributed optimal control $\u$, fluid velocity field $\y: \Omega \rightarrow \R^{2}$ and the pressure $p:\Omega \rightarrow \R$ such that:
\begin{align}\label{5.2.0.1}
	\mathcal{J}(\y,\u) = \underset{\tilde{\u} \in \mscr{U}_{\boldsymbol{ad}}}{\min} \mathcal{J}(\y,\tilde{\u}) = \underset{\tilde{\u} \in \mscr{U}_{\boldsymbol{ad}}}{\min} \ \bigg( \frac{1}{2} \|\y-\boldsymbol{y_d}\|_{0}^{2}
	+\frac{\lambda }{2}  \|\tilde{\u}\|_{0}^{2}\bigg),
\end{align}
subject to the generalized Oseen equations:
\begin{equation}\label{5.2.0.2}
	\begin{aligned}
		- \nu \Delta \y + (\boldsymbol{\beta} \cdot \nabla) \y + \sigma \y  + \nabla p &= \f + \u && \text{in} \  \Omega, \qquad &&
		\nabla \cdot \y = 0 &&  \text{in} \  \Omega, \\
		\hspace{4.4cm} \y &= \boldsymbol{0} && \text{on} \ \Gamma, \qquad&&
		 (p,1) = 0.
	\end{aligned}
\end{equation}
Here, {$\boldsymbol{y_d}\in \L^{2}(\Omega)$} is the desired velocity field, $\lambda>0$ is the regularization parameter, $\nu>0$ denotes the viscosity of the fluid, $\boldsymbol{\beta} \in \boldsymbol{W}^{1,\infty}(\Omega)$ is the convective velocity field, $\sigma \in L^{\infty}(\Omega)$ is a scalar function,
$\mathbf{\f} \in \boldsymbol{L}^{2}(\Omega)$ is given source term and $\u$ belongs to the \textbf{admissible set} defined as:
\begin{align*}
	\mscr{U}_{\boldsymbol{ad}}:= \left\{\u \in \boldsymbol{L}^{2}(\Omega): \boldsymbol{u_a} \le \u(\x) \le \boldsymbol{u_b} \ a.e. \ \text{in} \ \Omega; \boldsymbol{u_a}, \boldsymbol{u_b} \in \R^{2}\right\}.
\end{align*}
{Here, $\boldsymbol{u_a}$ and $\boldsymbol{u_b}$ represent the lower and upper bounds, respectively. The inequality $\boldsymbol{u_a} \le \u(\x) \le \boldsymbol{u_b}$ is understood component-wise, that is, each component of $\u(\x)$ is constrained between the corresponding components of $\boldsymbol{u_a}$ and $\boldsymbol{u_b}$.}
We suppose that there exists a constant {$\kappa\ge 0$ such that:
\begin{align*}
  \sigma(\x) - \frac{1}{2}(\nabla \cdot \boldsymbol{\beta}(\x)) \ge \kappa \quad \forall \x \in \Omega.
\end{align*}}
\subsection{Well-posedness} \label{Continuous formulation.}
The velocity-pressure weak formulation for the system \eqref{5.2.0.1} is: find $(\y,p) \in \V \times Q$, such that
\begin{equation}\label{5.2.1.1}
		\begin{aligned}
			(\nu \nabla \y,\nabla \v) \hspace{-0.3mm}+\hspace{-0.3mm} \frac{1}{2}((\boldsymbol{\beta} \cdot \nabla) \y, \v) \hspace{-0.3mm}-\hspace{-0.3mm} \frac{1}{2}((\boldsymbol{\beta} \cdot \nabla) \v, \y)
			+ \Big(\hspace{-0.3mm}\Big(\sigma\hspace{-0.3mm}-\hspace{-0.3mm}\frac{1}{2}(\nabla \cdot \boldsymbol{\beta})\Big) \y, \v\Big) \hspace{-0.3mm}-\hspace{-0.3mm} (p, \nabla \cdot \v) & = (\f+\u,\v) && \forall \v \in \V,\\
			\hspace{4.55cm} (q, \nabla \cdot \y) &= 0 && \forall q \in Q.
		\end{aligned}	
\end{equation}
\begin{definition}
	{Let $\Omega \subset \R^{2}$ be a simply connected domain}. Then, a vector field $\v \in \boldsymbol{V_{0}} $ if and only if there exists a scalar function $\phi \in H^2(\Omega)$ (known as \textbf{stream-function}), which is defined up to a constant (cf. Ref.~\refcite{VP}), such that
	\begin{align*}
		\v = \curl (\phi) \in \boldsymbol{V_{0}}.
	\end{align*}
\end{definition}
Let us consider the Hilbert space $\big(Z,\|\cdot\|_{Z}\big)$ with
\begin{align*}
	Z &:= \big\{\phi \in H^2(\Omega): \phi = \partial_{\n} \phi = 0 \ \text{on} \ \Gamma \big\} \quad \text{and} \quad &&\|\phi\|_{Z} := \big\|D^2 \phi \big\|_{0} \quad \forall \phi \in Z,
\end{align*}
where $\partial_{\n}$ denotes the normal derivative and $D^{2} \phi := (\partial_{ij} \phi)_{1\le i,j \le 2}$ denotes the Hessian matrix of $\phi$. By choosing $\psi \in Z$ as a stream-function of the velocity $\y \in \boldsymbol{V_{0}}$, we express the equivalent weak form of system \eqref{5.2.1.1} as: find $\psi \in Z$ such that
	\begin{align}\label{5.2.1.2}
	\mscr{O}_{s}(\psi,\phi) &:= \mscr{A}(\psi,\phi) + \mscr{B}(\psi,\phi) + \mscr{C}(\psi,\phi) = \mscr{L}_{\f,\u}(\phi) &&\forall \phi \in Z,
	\end{align}
	where the bilinear forms $\mscr{A}, \mscr{B}, \mscr{C}:Z \times Z \rightarrow \R$ and the linear functional $\mscr{L}_{\f,\u}:Z \rightarrow \R$ are defined as:
	\begin{subequations}\label{5.2.1.3}
		\begin{align}
			\label{5.2.1.3a} \mscr{A}(\psi,\phi) &:= \left(\nu D^{2} \psi, D^{2} \phi\right), \\
			\label{5.2.1.3b} \mscr{B}(\psi,\phi) &:= \frac{1}{2}[\left((\boldsymbol{\beta} \cdot \nabla) \curl(\psi), \curl(\phi)\right) - \left((\boldsymbol{\beta} \cdot \nabla) \curl(\phi), \curl(\psi)\right)],\\
			\label{5.2.1.3c} \mscr{C}(\psi,\phi) &:=  \Big(\Big(\sigma-\frac{1}{2}(\nabla \cdot \boldsymbol{\beta})\Big) \curl(\psi), \curl(\phi)\Big),&&
			{\mscr{L}_{\f,\u}(\phi)} &:= \left(\f+\u, \curl(\phi)\right),
		\end{align}
	\end{subequations}
for all $\psi,\phi \in Z$.
	The following lemma establishes well-posedness
	of the continuous problem \eqref{5.2.1.2}.
	\begin{lemma}\label{Lemma 2.1}
		There exist positive constants $C_{\bbeta}$ and $C_{\sigma}$, such that the bilinear and linear forms defined in \eqref{5.2.1.3} satisfy the following conditions for all $\psi,\phi \in Z$:
		\begin{align*}
			|\mscr{A}(\psi,\phi)| &\le \|\psi\|_{Z} \|\phi\|_{Z}, \qquad&&|\mscr{A}(\phi,\phi)| \ge \|\phi\|_{Z}^{2},\\
			|\mscr{C}(\psi,\phi)| &\le C_{\sigma}\|\psi\|_{Z} \|\phi\|_{Z}, \qquad &&|\mscr{C}(\phi,\phi)| \ge \kappa|\phi|_{1}^{2}, \\
			|\mscr{B}(\psi,\phi)| &\le C_{\bbeta}\|\psi\|_{Z} \|\phi\|_{Z},\qquad&& {|\mscr{L}_{\f,\u}(\phi)| \le \|\mscr{L}_{\f,\u}\|_{-1} \|\phi\|_{Z}}.
		\end{align*}
	\end{lemma}
	Now, we are in a position to state that problem \eqref{5.2.1.2} is well-posed. It is a consequence of  Lemma~\ref{Lemma 2.1} and the Lax-Milgram theorem.
	\begin{theorem}\label{Theorem 5.2.1.3}
There exists a unique solution $\psi \in Z$ to the continuous problem \eqref{5.2.1.2}, which satisfies the following estimate, depending on the given data, for some constant $C > 0$
		\begin{align*}
			\|\psi\|_{Z} \le C {\|\mscr{L}_{\f,\u}\|_{-1}}.
		\end{align*}
	\end{theorem}
	\subsection{Optimality System}\label{Optimality system}
	By following Ref.~\refcite{FTO}, we define the \textit{control-to-state} map $\boldsymbol{\mathcal{S}}: \boldsymbol{L}^{2}(\Omega) \rightarrow \boldsymbol{L}^{2}(\Omega)$ that associates $\curl(\psi)$ of the state stream-function $\psi$ to a given control $\u \in \mscr{U}_{\boldsymbol{ad}}$. The optimal control problem (\ref{5.2.0.1}-\ref{5.2.0.2}) reduces to the following quadratic optimization problem in the Hilbert space $\L^{2}(\Omega)$:
	\begin{align}\label{5.2.1.4}
		\underset{\u \in \mscr{U}_{\boldsymbol{ad}}}{\min}\mathcal{J}_{r}(\u) := \frac{1}{2} \|\boldsymbol{\mathcal{S}}(\u)- \curl(\psi_d)\|_{0}^{2} +\frac{\lambda }{2}  \|\u\|_{0}^{2},
	\end{align}
	where $\curl(\psi_d)=\boldsymbol{y_d}$ {with $\nabla \cdot \boldsymbol{y_d} = 0$}. The functional $\mathcal{J}_{r} :\boldsymbol{L}^{2}(\Omega) \rightarrow \R$ just defined is known as the \textit{reduced cost functional} as it is dependent only on the optimal control $\u$.
	Since $\lambda >0$, $\mathcal{J}_r$ is strictly convex.	The existence and uniqueness of a optimal solution $(\psi,\u)$ to the optimization problem 
	\begin{align*}
		\underset{\u \in \mscr{U}_{\boldsymbol{ad}}}{\min} \mathcal{J}_r(\u) \qquad \text{subject to} \qquad \eqref{5.2.1.2}
	\end{align*}
	follows by an application of Theorem~2.14\cite{FTO}.
	
	Since, the \textit{control-to-state} operator $\boldsymbol{\mathcal{S}}$ is linear and continuous from the control space into the state space and the norm functional being a continuous linear operator is Fr\'echet differentiable. Therefore, the \textit{reduced cost functional} $\mathcal{J}_r$ is also Fr\'echet differentiable. Now, from the first derivative of the \textit{reduced cost functional}, we derive the \textit{necessary} and also \textit{sufficient} (because of the convexity), so called \textit{first order optimality condition}. The optimal solution $\u$ satisfies the following \textit{variational inequality} which is also called as \textit{first order optimality condition} (cf. Lemma~2.21\cite{FTO}):
	\begin{align*}
		\mathcal{J}_{r}'(\u)(\tilde{\u}-\u) \ge 0 \qquad \forall \tilde{\u} \in \mscr{U}_{\boldsymbol{ad}},
	\end{align*}
	Furthermore, by substituting gradient of the \textit{reduced cost functional} in the above inequality (cf. Theorem~2.20\cite{FTO}), we have
	\begin{align*}
		 \big(\curl(\vartheta) + \lambda \u,\tilde{\u}-\u\big) \ge 0  \qquad \forall \tilde{\u} \in \mscr{U}_{\boldsymbol{ad}},
	\end{align*}
	where $\vartheta$ denotes a unique stream-function solution to the following co-state problem
	\begin{align}
		\label{5.2.1.6}	\mscr{O}_{a}(\tau,\vartheta) &:= \mscr{A}(\vartheta,\tau) - \mscr{B}(\vartheta,\tau) + \mscr{C}(\vartheta,\tau) = {\mscr{L}_{\psi_d,\psi}(\tau)} &&\forall \tau \in Z,\\
		\nonumber \mscr{L}_{\psi_d,\psi}(\tau) &:= \left(\curl(\psi)- \curl(\psi_d),\curl(\tau)\right),
	\end{align}
	where we have used the skew-symmetry of $\mscr{B}(\cdot,\cdot)$ and the symmetry of bilinear forms $\mscr{A}(\cdot,\cdot)$ and $\mscr{C}(\cdot,\cdot)$. 
	
	By coupling the state and co-state system with the \textit{variational inequality}, the continuous optimality system is written as: find $(\psi, \vartheta, \u) \in Z \times Z \times \mscr{U}_{\boldsymbol{ad}}$, such that 
	\begin{subequations}\label{5.2.1.7}
		\begin{align}{}
			\label{5.2.1.7a} \mscr{O}_{s}(\psi,\phi) &= \mscr{A}(\psi,\phi) + \mscr{B}(\psi,\phi) + \mscr{C}(\psi,\phi) = {\mscr{L}_{\f,\u}(\phi)} &&\forall \phi \in Z,\\
			\label{5.2.1.7b} \mscr{O}_{a}(\tau,\vartheta)&= \mscr{A}(\vartheta,\tau) - \mscr{B}(\vartheta,\tau)+ \mscr{C}(\vartheta,\tau)\ = {\mscr{L}_{\psi_d,\psi}(\tau)} &&\forall \tau \in Z,\\
			\label{5.2.1.7c} &\hspace{1.48cm} \big(\curl(\vartheta) + \lambda \u,\tilde{\u}-\u\big) \ge \boldsymbol{0} &&\forall \tilde{\u} \in \mscr{U}_{\boldsymbol{ad}}.
		\end{align} 
	\end{subequations}
	\indent The \textit{variational inequality} \eqref{5.2.1.7c} is equivalently expressed as:
	\begin{align}\label{5.2.1.8}
		\u = \mscr{P}_{[\boldsymbol{u_a},\boldsymbol{u_b}]} \left(-\frac{1}{\lambda} \curl(\vartheta(\x))\right) \qquad \text{for} \ a.e. \ \x \ \text{in} \ \Omega,
	\end{align}
	(cf. Theorem~2.28,\cite{FTO}), where $\mscr{P}_{[\boldsymbol{u_a},\boldsymbol{u_b}]}$ is a projection operator defined as:
	\begin{align*}
		{\mscr{P}_{[\boldsymbol{u_a},\boldsymbol{u_b}]}(\v)(\x) := \min \bigl\{\boldsymbol{u_b}, \max \{\boldsymbol{u_a},\v(\x)\} \bigr\}}.
	\end{align*}
	\indent The optimal solution $\u \in \mscr{U}_{\boldsymbol{ad}}$ satisfies the following coercivity condition\cite{CCNRNEW}:
	\begin{align}\label{5.2.1.9}
		\mathcal{J}_r''(\u)[\tilde{\u},\tilde{\u}] \ge \lambda \|\tilde{\u}\|_{0}^{2} \qquad \qquad \forall \tilde{\u} \in \mscr{U}_{\boldsymbol{ad}}.
	\end{align}
	{In order to derive a priori error estimates, we need to analyze
	the regularity of the solutions for the optimal control problem.
	\begin{lemma}\label{addReg}
	There exists $s\in (\frac{1}{2},1]$ such that for sufficiently smooth $\boldsymbol{\beta}$ and $\sigma$
	and for all $\f\in\L^{2}(\Omega)$ and $\psi_d\in H^{1}(\Omega)$, we have $(\psi, \vartheta, \u) \in H^{2+s}(\Omega) \times H^{2+s}(\Omega) \times \H^{1}(\Omega)$.
	\end{lemma}
	\begin{proof}
	Since $\f\in\L^{2}(\Omega)$, $\u\in\mscr{U}_{\boldsymbol{ad}}\subset\L^{2}(\Omega)$
	and $\psi_d\in H^{1}(\Omega)$,
	by following the Refs.~\cite{MR595625, MR3051409}, we have that for a simply connected polygonal domain, not necessarily convex, with homogeneous Dirichlet boundary conditions, there exists $s\in (\frac{1}{2},1]$ such that the solution of problems \eqref{5.2.1.7a} and \eqref{5.2.1.7b} belong to $H^{2+s}(\Omega)$. Thus, by using \eqref{5.2.1.8}, we obtain that $\u\in\H^{1}(\Omega)$.
	\end{proof}
	}
\begin{remark}\label{remark:aditional:regu}
{The constant $s$ in the lemma above is the Sobolev regularity for the
biharmonic equation with the right-hand side in $H^{-1}(\Omega)$
and homogeneous Dirichlet boundary conditions.
This constant only depends on the domain $\Omega$. If $\Omega$ is convex,
then $s=1$. Otherwise, the lemma holds for all $s<s_0$, where $s_0\in (\frac{1}{2},1)$ depends
on the largest re-entrant angle of $\Omega$ (see Ref.~\cite{MR775683}
for the precise equation determining $s_0$).}
\end{remark}
\begin{remark}\label{remark:aditional:regu2}
{If $\Omega$ is convex with all the interior angles of $\Omega$ are less than $126.283$ degrees, $\f\in\H^{1}(\Omega)$ and  $\psi_d\in H^{2}(\Omega)$
then $(\psi, \vartheta, \u) \in H^{4}(\Omega) \times H^{4}(\Omega) \times \H^{1}(\Omega)$ (see Refs.~\cite{MR595625, MR4374004} for more details).
In this case, we note that the right-hand sides in \eqref{5.2.1.7a} and \eqref{5.2.1.7b} belong to $\L^{2}(\Omega)$.}
\end{remark}
	\section{VEM Discretization and Well-posedness} \label{Virtual Element Methods} \setcounter{equation}{0}
This section is devoted to the construction of a $C^1$-VEM for the numerical approximation of the state and adjoint problems \eqref{5.2.1.7a} and~\eqref{5.2.1.7b}, respectively. The method is formulated for general polynomial degrees $k \geq 2$.
	
	Let $\{\mscr{T}_{h}\}_{h>0}$ be a sequence of subdivisions of $\Omega$ into general polygon-shaped elements $K$. We define $h_K$ as the largest side length of an element $K$ and set $h = \max_{K \in \mscr{T}_{h}} h_K$. The number of vertices of $K$ is denoted by $N_{K}$ and $\v_{i}$ represents a generic vertex of $K$, where $i \in \{1, \dots, N_K\}$. Let $\mathcal{E}(\mscr{T}_{h})$ be the set of all edges with $h_{E}$ and $\n_{K}^{E}$ denoting the length and the unit normal vector pointing outward from $K$ for an edge $E$. Additionally, we define $x_E$ as the midpoint of an edge $E$ and $\x_{K}$ as the centroid of $K$. For the theoretical analysis, we assume there exists a constant $C_{\mscr{T}_h} > 0$, independent of $h_K$, such that for all $K \in \mscr{T}_h$:\\
	$\boldsymbol{\mathcal{A}1}:$ Ratio of the diameter \( h_K \) to shortest edge of element \( K \) exceeds \( C_{\mscr{T}_{h}} \).\\
	$\boldsymbol{\mathcal{A}2}:$ \( K \) is star-shaped around each point inside a ball with radius \( C_{\mscr{T}_{h}} h_K \).

	\subsection{Virtual Spaces and Polynomial Projections Operator}\label{Virtual Spaces and Polynomial Projections Operator}
	Let \(\mathbb{P}_{\ell}(\mathfrak{D})\) represent the space of polynomial functions that can be defined locally within \(\mathfrak{D}\) with a total degree of at most \(\ell\). Meanwhile, \(\mscr{M}_{\ell}^{*}(K)\) denotes the collection of scaled monomials that are defined for each polygon \(K\) as:
	\begin{align*}
		\mscr{M}_{\ell}^{*}(K):= \left\{ \left(\frac{\x-\x_{K}}{h_K}\right)^{\boldsymbol{\alpha}}:|\boldsymbol{\alpha}| = \ell\right\},
	\end{align*}
	where $\boldsymbol{\alpha} = (\alpha_1, \alpha_2)$ is a non-negative multi-index with $|\boldsymbol{\alpha}| = \alpha_1 + \alpha_2$ and $\x^{\boldsymbol{\alpha}} = x_1^{\alpha_1}x_2^{\alpha_2}$. Next, we define 
	\begin{align*}
		\mscr{M}_{k}(K) := \underset{\ell \le k}{\bigcup} \mscr{M}_{\ell}^{*}(K) = \big\{m_{j}\big\}_{j=1}^{d_k},
	\end{align*}
	as a basis for the space $\mathbb{P}_{k}(K)$, whose dimension is $d_{k}$.
	
	We now define the set of scaled monomials defined on each edge $E$ as:
	\begin{align*}
		\mscr{M}_{\ell}(E) := \left\{1,\frac{x-x_E}{h_E},\bigg(\frac{x-x_E}{h_E}\bigg)^{2},\ldots,\bigg(\frac{x-x_E}{h_E}\bigg)^{\ell}\right\}.
	\end{align*}
	\indent For all $K \in \mscr{T}_h$ and an integer $k \geq 2$, we define the following basic local virtual space:
	\begin{align*}
		\tilde{Z}_{h}^{k}(K):= \Big\{\phi_h \in H^{2}(K)&: \ \Delta^{2} \phi_h \in \mathbb{P}_{k-2}(K), \ \phi_h|_{\partial K} \in C^{0}(\partial K), \phi_h|_{E} \in \mathbb{P}_{r}(E),\\
		& \ \ \ \ \nabla \phi_h|_{\partial K} \in \big[C^{0}(\partial K)\big]^{2}, \ \partial_{\n_{K}^{E}} \phi_h|_{E} \in \mathbb{P}_{t}(E) \ {\forall E \subset \partial K}\Big\},
	\end{align*}
	where $r := \max \{3,k\}$ and $t:= k-1$. For all $\phi_h \in \tilde{Z}_{h}^{k}(K)$, we define the following linear operators that map $\tilde{Z}_{h}^{k}(K)$ to $\mathbb{R}$:
	\begin{itemize}
		\item $\boldsymbol{D_1}$: holds linear operators that compute $\phi_h$ at all the vertices of $K$;\vspace{1mm}
		\item $\boldsymbol{D_2}$: holds linear operators that compute $h_{\v_i} \nabla \phi_h$ at all the vertices of $K$; \vspace{1mm}
		\item $\boldsymbol{D_3}$: moments $\dfrac{1}{h_E} \displaystyle\int_{E} m(s) \phi_h(s) ds$ \ $\forall m \in \mscr{M}_{k-4} (E)$, $\forall E \in \mathcal{E}(\mscr{T}_{h})$ and $k\ge 4$; \vspace{1mm}
		\item $\boldsymbol{D_4}$: moments $\displaystyle \int_{E} m(s) \partial_{\n_{K}^{E}} \phi_h(s) ds \ \forall m \in \mscr{M}_{k-3} (E)$, $\forall E \in \mathcal{E}(\mscr{T}_{h})$ and $k\ge 3$; \vspace{1mm}
		\item $\boldsymbol{D_5}$: moments $\dfrac{1}{h_K^{2}} \displaystyle \int_{K} m(\x) \phi_h(\x) d \x \ \forall m \in \mscr{M}_{k-4} (K)$, $\forall K \in \mscr{T}_{h}$ and $k\ge 4$,
	\end{itemize}
	where $h_{\v_i}$ represent the average diameter of the elements sharing $\v_i$ as a common vertex. We break down the bilinear forms $\mscr{A}(\cdot,\cdot)$, $\mscr{B}(\cdot,\cdot)$, and $\mscr{C}(\cdot,\cdot)$ into their local contributions as follows:
		\begin{align*}
            \mscr{A}(\psi,\phi) = \sum_{K \in \mscr{T}_{h}} \mscr{A}^{K}(\psi,\phi) &:= \sum_{K \in \mscr{T}_{h}} \big(\nu D^{2} \psi, D^{2} \phi\big)_{K},
			\\
			 \mscr{B}(\psi,\phi)= \sum_{K \in \mscr{T}_{h}} \mscr{B}^{K}(\psi,\phi) &:=\sum_{K \in \mscr{T}_{h}}\frac{1}{2} \Big(\widetilde{\mscr{B}}^{K}(\psi,\phi)-\widetilde{\mscr{B}}^{K}(\phi,\psi)\Big),\\
        \mscr{C}(\psi,\phi) = \sum_{K \in \mscr{T}_{h}} \mscr{C}^{K}(\psi,\phi) &:= \sum_{K \in \mscr{T}_{h}} \Big(\Big(\sigma-\frac{1}{2}(\nabla \cdot \boldsymbol{\beta})\Big) \curl(\psi), \curl(\phi)\Big)_{K},
		\end{align*}
	with $\widetilde{\mscr{B}}^{K}(\psi,\phi):= \big((\boldsymbol{\beta} \cdot \nabla) \curl(\psi), \curl(\phi) \big)_K$ for all $\psi,\phi \in Z$. Now, our aim is to build discrete version of local bilinear forms by defining the following projector, for $k\ge2$:
	\begin{align*}
		\Pi_{K}^{k,D}&: \tilde{Z}_{h}^{k}(K) \rightarrow \mathbb{P}_{k}(K) \subseteq \tilde{Z}_{h}^{k}(K), \qquad &&\phi_h \rightarrow \Pi_{K}^{k,D} \phi_h,
	\end{align*}
\blue{where $\Pi_{K}^{k,D} \phi_h$ is the solution of the local problems:
	\begin{align}
		\mscr{A}^{K}\big(\Pi_{K}^{k,D} \phi_h, q\big) &= \mscr{A}^{K}(\phi_h,q) \quad \forall q \in \mathbb{P}_{k}(K),\label{proj:PiD} \\ 
	\Pi^{0}_{\partial K}(\Pi_{K}^{k,D} \phi_h) &= \Pi^{0}_{\partial K}(\phi_h), \label{kernel:proj:PiD1}\\  
		\Pi^{0}_{\partial K}(\partial_i\Pi_{K}^{k,D} \phi_h) &=  \Pi^{0}_{\partial K}(\partial_i\phi_h), \quad \text{for $i=1,2$.} \label{kernel:proj:PiD}
	\end{align}
	where 
	\begin{align}\label{5.3.1.2.0}
		\Pi^{0}_{\partial K}(\varphi_h) := \frac{1}{|\partial K|} \int_{ \partial K}\varphi_h  \qquad \forall \varphi_h \in H^2(K).
	\end{align}}
	
	The following result establishes that the projector $\Pi_{K}^{k,D}$ is fully computable using the sets $\boldsymbol{D_1}$-$\boldsymbol{D_5}$\cite{VEMPB}.

	\begin{lemma}\label{Lemma 5.3.1.1}
		The operator $\Pi_{K}^{k,D}: \tilde{Z}_{h}^{k}(K) \rightarrow \mathbb{P}_{k}(K)$ can be explicitly calculated for any $\phi_h \in \tilde{Z}_{h}^{k}(K)$, relying totally on the data from the linear operators $\boldsymbol{D_1}$-$\boldsymbol{D_5}$.
	\end{lemma}
\begin{proof}
Let $\phi_h \in \tilde{Z}_h^k(K)$ and $q_k \in \mathbb{P}_k(K)$, then we perform two successive integrations by parts on the right-hand side of~\eqref{proj:PiD}, which yields
\begin{equation*}
	\begin{aligned}
		\mscr{A}^{K}(\phi_h,q_k)
		&= \nu\int_K D^2 \phi_h : D^2 q_k \\
		&= \nu\int_K \Delta^2 q_k\, \phi_h
		- \nu\int_{\partial K} \operatorname{div}(D^2 q_k)\cdot \boldsymbol{n}_{K}\, \phi_h
		+ \nu\int_{\partial K} D^2 q_k\, \boldsymbol{n}_{K} \cdot \nabla \phi_h .
	\end{aligned}
\end{equation*}

Since $\Delta^2 q_k \in \mathbb{P}_{k-4}(K)$, the volume integral on the right-hand side of the above identity can be explicitly evaluated using the degrees of freedom in the set $\boldsymbol{D_5}$. Moreover, the boundary integrals depend solely on the traces of $\phi_h$ and $\nabla \phi_h$ along $\partial K$, and are therefore computable from the degrees of freedom in the sets $\boldsymbol{D_1}$-$\boldsymbol{D_4}$. Finally, the kernel component of the operator $\Pi^{k,D}_K$ (see~\eqref{kernel:proj:PiD}) is determined using only the degrees of freedom in the sets $\boldsymbol{D_1}$-$\boldsymbol{D_2}$.
	\end{proof}
\	
	For each polygon $K \in \mscr{T}_h$, the local virtual space is defined as:
	\begin{align*}
		Z_{h}^{k}(K):= \Big\{\phi_h \in \tilde{Z}_{h}^{k}(K) : (q^{*}, \Pi_{K}^{k,D}\phi_h - \phi_h)_{K} = 0 \ \forall q^{*} \in \mscr{M}_{k-3}^{*}(K) \cup \mscr{M}_{k-2}^{*}(K) \Big\},
	\end{align*}
where $\mscr{M}_{k-3}^{*}(K)$ and $\mscr{M}_{k-2}^{*}(K)$ are scaled monomials of degree $k-3$ and $k-2$, respectively, with the convention that $\mscr{M}_{-1}^{*}(K)=\emptyset$.
  
\indent We can observe that $\mathbb{P}_{k}(K) \subseteq Z_{h}^{k}(K) \subseteq \tilde{Z}_{h}^{k}(K)$. The sets $\boldsymbol{D_1}$-$\boldsymbol{D_5}$ form the degrees of freedom for $Z_{h}^{k}(K)$\cite{VEMFOUR}. The condition in defining the local space $Z_{h}^{k}(K)$ is crucial for building an $L^2$-projection, which helps to define the discrete bilinear forms. Specifically, we use the $L^2(K)$-projection onto $\mathbb{P}_{k-2}(K)$, defined as: for any function $\phi \in L^2(K)$, the projection $\Pi_{K}^{k-2}\phi$ belongs to $\mathbb{P}_{k-2}(K)$ and satisfies
	\begin{align*}
		\big(q,\phi-\Pi_{K}^{k-2}\phi \big)_{K} = 0 \qquad \forall q \in \mathbb{P}_{k-2}(K).
	\end{align*}
	\indent The next lemma shows that $\Pi_{K}^{k-2}$ can be calculated within $Z_{h}^{k}(K)$, which follows from the definition of local virtual space and the corresponding degrees of freedom.
	\begin{lemma}\label{Lemma 5.3.1.1.0}
		The operator $\Pi_{K}^{k-2}: Z_{h}^{k}(K) \rightarrow \mathbb{P}_{k-2}(K)$ can be directly computed for any $\phi_h \in Z_{h}^{k}(K)$, using just the degrees of freedom $\boldsymbol{D_1}$-$\boldsymbol{D_5}$.
	\end{lemma}
\indent The next result establishes the construction
of a novel projection operator to be used in the discretization of the convective term. The proposed operator enables the derivation of optimal error estimates for the stream-function.
	\begin{lemma}\label{Lemma 5.3.1.1.1}
	For all $\phi_h \in Z_h^{k}(K)$, the operator $\boldsymbol{\Pi}_{K,\text{mat}}^{k-2}$ defined as:
	\begin{align*}
		\big(\boldsymbol{\Pi}_{K,\text{mat}}^{k-2} \nabla \mathrm{\bf{curl}}(\phi_h),\boldsymbol{q}_{k-2}\big)_{K} &= \big(\nabla \mathrm{\bf{curl}}(\phi_h),\boldsymbol{q}_{k-2}\big)_{K} \qquad \forall \boldsymbol{q}_{k-2} \in [\mathbb{P}_{k-2}(K)]^{2\times 2},
	\end{align*}
	is fully computable using the degrees of freedom $\boldsymbol{D_1}$-$\boldsymbol{D_5}$.
\end{lemma}
\begin{proof}
		Let $\boldsymbol{q}_{k-2} \in [\mathbb{P}_{k-2}(K)]^{2\times 2}$, then by using integration by parts and the definition of projection operator $\boldsymbol{\Pi}_{K,\text{mat}}^{k-2}$, we have
		\begin{align*}
				\big(\boldsymbol{\Pi}_{K,\text{mat}}^{k-2} \nabla \curl(\phi_h),\boldsymbol{q}_{k-2}\big)_{K} &= \big(\nabla \curl(\phi_h),\boldsymbol{q}_{k-2}\big)_{K}\\
				&= - \big(\curl(\phi_h),\nabla  \cdot \boldsymbol{q}_{k-2}\big)_{K} + \big(\curl(\phi_h), \boldsymbol{q}_{k-2} \boldsymbol{n}_{K}\big)_{\partial K}\\
				&= \big(\phi_h,\curl(\nabla  \cdot \boldsymbol{q}_{k-2})\big)_{K} -  \big(\phi_h, (\nabla  \cdot \boldsymbol{q}_{k-2})\cdot \boldsymbol{t}_{K}\big)_{\partial K} \\
				& \quad + \big(\curl(\phi_h), \boldsymbol{q}_{k-2} \boldsymbol{n}_{K}\big)_{\partial K} := I_{K} + I_{E}^{1} + I_{E}^{2}.
			\end{align*}
			
	The term $I_{K}$ is computable using degrees of freedom $\boldsymbol{D_5}$ for $k \ge 4$, while $I_{E}^{1}$ is computable using degrees of freedom $\boldsymbol{D_1}$ to $\boldsymbol{D_4}$ for $k \ge 3$ and $I_{E}^{2}$ is computable using degrees of freedom $\boldsymbol{D_1}$ to $\boldsymbol{D_4}$ for $k \ge 2$.
\end{proof}

	For \( k \ge 2 \), we define a projection onto the polynomial space \( [\mathbb{P}_{k-1}(K)]^{2} \). For any \( \v \in \boldsymbol{L}^{2}(K) \), the operator \( \boldsymbol{\Pi}_{K}^{k-1} \) mapping \( \boldsymbol{L}^{2}(K) \) to \( [\mathbb{P}_{k-1}(K)]^{2} \) is defined as:
	\begin{align*}
		\big(\boldsymbol{q}_{k-1}, \v - \boldsymbol{\Pi}_{K}^{k-1} \v\big)_{K} = 0 \qquad \forall \boldsymbol{q}_{k-1} \in [\mathbb{P}_{k-1}(K)]^{2}.
	\end{align*}
	\indent We find that the vector function $\boldsymbol{\Pi}_{K}^{k-1} \curl (\phi_h) \in [\mathbb{P}_{k-1}(K)]^{2}$ for any $\phi_h \in Z_{h}^{k}(K)$ can be calculated directly from the degrees of freedom $\boldsymbol{D_1}$-$\boldsymbol{D_5}$ using integration by parts. Next, we will show a stability property of the projector $\boldsymbol{\Pi}_{K}^{k-1}$. To do this, we will use the inverse inequality for polynomials on polygonal elements, which holds true under assumption $\boldsymbol{\mathcal{A}2}$ (cf. Remark~6.1,\cite{STABVEM}).
	\begin{lemma}\label{Lemma 5.3.1.3}
		For a positive constant $\hat{C}$ (independent of $h$), if assumption $\boldsymbol{\mathcal{A}2}$ holds, then we have
		\begin{align*}
			|q|_{1,K} \le \hat{C} h_{K}^{-1} \|q\|_{0,K} \qquad \forall q \in \mathbb{P}_{\ell}(K) \quad \text{and} \quad \ell \ge 1.
		\end{align*}
	\end{lemma}
	
	By applying Lemma~\ref{Lemma 5.3.1.3}, the projector \(\boldsymbol{\Pi}_{K}^{k-1}\) has a stability property. Specifically, for any element \(K\) in the mesh \(\mscr{T}_{h}\) and for all functions \(\phi_h\) in \(Z_{h}^{k}(K)\), there exists a constant \(C_N\) (which does not depend on \(K\) or \(h\)) such that:
	\begin{align}\label{3.5.1.5}
		\big|\boldsymbol{\Pi}_{K}^{k-1} \curl (\phi_h)\big|_{1,K} \le C_{N} |\phi_h|_{2,K}.
	\end{align}
	
	\blue{We now state the following well-known Poincar\'e-Friedrichs inequalities on Lipschitz domains, which will be useful for the forthcoming analysis (see below Theorems~\ref{thrm:SD:stab} and~\ref{thrm:SNabla:stab}).
		\begin{align}
		h^{-1}_{K} \|v\|_{0,K} &\lesssim |v|_{1,K} + h^{-1}_{K}\left| \int_{K} v \, {\rm d}s\right| && \forall v \in H^1(K), \label{Poincare1:H1}\\ 	
			h^{-1}_{K} \|v\|_{0,K} &\lesssim |v|_{1,K} + h^{-1}_{K}\left| \int_{\partial K} v \, {\rm d}s\right| && \forall v \in H^1(K), \label{Poincare2:H1}\\ 
			h^{-2}_{K} \|v\|_{0,K}& \lesssim |v|_{2,K} + h^{-2}_{K}\left| \int_{\partial K} v \, {\rm d}s\right| + h^{-1}_{K}\left| \int_{\partial K} \nabla v \, {\rm d}s\right| && \forall v \in H^2(K).\label{Poincare:H2}
	\end{align}}
	
	Next, for \(k \ge 2\), we define an additional projector, \(\Pi_{K}^{k,\nabla^{\perp}}\), which maps \(Z_{h}^{k}(K)\) to \(\mathbb{P}_{k}(K)\). This is determined by solving a specific local problem for all \(\phi_h \in Z_{h}^{k}(K)\):
	\begin{align}
		\label{revised1} \Big(\curl \Big(\Pi_{K}^{k,\nabla^{\perp}} \phi_h \Big), \curl (q)\Big)_{K}&= \big(\curl (\phi_h), \curl (q)\big)_{K} \quad \forall q \in \mathbb{P}_{k}(K),\\
		\label{revised2} \blue{\Pi^{0}_{\partial K}(\Pi_{K}^{k,\nabla^{\perp}} \phi_h)} &= \blue{\Pi^{0}_{\partial K}(\phi_h),}
	\end{align}
	where $\blue{\Pi^{0}_{\partial K}(\cdot)}$ is same as \eqref{5.3.1.2.0}. 
	
	
	The next result indicates that operator $\Pi_{K}^{k,\nabla^{\perp}}$ can be completely calculated using the sets $\boldsymbol{D_1}$ to $\boldsymbol{D_5}$, as mentioned in Lemma~3.4\cite{MORABRINKMAN}.
	\begin{lemma}\label{Lemma 5.3.1.4}
		The operator $\Pi_{K}^{k,\nabla^{\perp}}: Z_{h}^{k}(K) \rightarrow \mathbb{P}_{k}(K) \subseteq Z_{h}^{k}(K)$ is determined for all $\phi_h \in Z_h^{k}(K)$, by using the data provided from the degrees of freedom $\boldsymbol{D_1}$ to $\boldsymbol{D_5}$.
	\end{lemma}
	\indent We now introduce the global virtual space to find the solution to problem \eqref{5.2.1.2}. Given a decomposition $\mscr{T}_h$ of the domain $\Omega$ into polygonal elements $K$, we define
	\begin{align*}
		Z_h^{k} := \left\{\phi_h \in Z: \phi_h|_K \in Z_h^{k}(K) \quad \forall K \in \mscr{T}_h \right\}.
	\end{align*}
	\subsection{Discretization of Control}
	The most basic option for control discretization is the piecewise constant approximation, where the control is taken to be cellwise constant over the mesh. Another widely used choice is the piecewise linear approximation, which represents the control via continuous, piecewise affine functions. Regarding convergence, it is known that these approximations yield rates of $O(h)$ (see Refs.~\refcite{MR686788} and \refcite{MR555382}) and $O(h^{3/2})$ (see Ref.~\refcite{MR1993702}) for piecewise constant and linear discretizations, respectively. In contrast to the above explicit schemes, we use a \textit{variational discretization} approach introduced in Ref.~\refcite{MHINZE}. The key idea is that the control is not discretized directly. Instead, its discrete counterpart is implicitly determined through the discretized adjoint equation and the first-order optimality system, in particular via the projection formula \eqref{5.2.1.8}. The discrete admissible space $\mscr{U}_{\boldsymbol{ad}}^{\boldsymbol{h}}$ coincides with $\mscr{U}_{\boldsymbol{ad}}$. This implicit treatment avoids restrictions on the control space and often leads to improved accuracy. This scheme has been used for OCPs in Refs.~\refcite{CVEM2}, \refcite{CVEM3}, \refcite{CVEM1} and \refcite{OCPBH} with a virtual element discretization for the state and adjoint problems.
	\subsection{Local and Global Discrete Forms}
	In this subsection, we will build the discrete counterparts of the continuous local bilinear forms and the right hand side, based on the projection operators introduced earlier in Sec.~\ref{Virtual Spaces and Polynomial Projections Operator}.
	
	Let $\mscr{S}_{D}^{K}(\cdot,\cdot)$ and $\mscr{S}_{\curl}^{K}(\cdot,\cdot)$ be any symmetric and positive definite bilinear forms which satisfy the following conditions: 
	\begin{subequations}\label{5.3.2.1}
		\begin{align}
			\label{5.3.2.1a}\mscr{A}^{K}(\phi_h,\phi_h) &\simeq \mscr{S}_{D}^{K}(\phi_h,\phi_h) \qquad \forall\phi_h \in Z_h^{k}(K)  \cap {\rm Ker}(\Pi_{K}^{k,D}),\\
			\label{5.3.2.1b} \mscr{C}^{K}(\phi_h,\phi_h) &\simeq \mscr{S}_{\curl}^{K}(\phi_h,\phi_h)  \quad \; \forall\phi_h \in Z_h^{k}(K)  \cap {\rm Ker}(\Pi_{K}^{k,\nabla^{\perp}}).
		\end{align}
	\end{subequations}
    
Now, for all $\psi_h, \phi_h \in Z_h^{k}(K)$, we define the following local discrete bilinear forms:
	\begin{subequations}\label{5.3.2.2}
		\begin{align}
			\label{5.3.2.2a} \mscr{A}^{h,K}(\psi_h,\phi_h) &:= \mscr{A}^{K}\big(\Pi_{K}^{k,D} \psi_h,\Pi_{K}^{k,D} \phi_h\big) \hspace{-0.5mm}+\hspace{-0.5mm} \mscr{S}_{D}^{K}\big(\psi_h \hspace{-0.5mm}-\hspace{-0.5mm}\Pi_{K}^{k,D} \psi_h,\phi_h \hspace{-0.5mm}-\hspace{-0.5mm}\Pi_{K}^{k,D} \phi_h\big)\\
			\label{5.3.2.2b} 
			\widetilde{\mscr{B}}^{h,K}(\psi_h,\phi_h) &:= \big(\big( \boldsymbol{\Pi}_{K,\text{mat}}^{k-2}  \nabla \curl(\psi_h) \big) \bbeta , \boldsymbol{\Pi}_{K}^{k-1} \curl(\phi_h)\big)_{K} \\
			\nonumber	\mscr{C}^{h,K}(\psi_h,\phi_h) &:= \Big(\Big(\sigma-\frac{1}{2}(\nabla \cdot \boldsymbol{\beta})\Big)  \boldsymbol{\Pi}_{K}^{k-1} \curl(\psi_h), \boldsymbol{\Pi}_{K}^{k-1} \curl(\phi_h)\Big)_{K} \\
			\label{5.3.2.2c} & \ \ \ \ \ + \mscr{S}_{\curl}^{K}\Big(\psi_h -\Pi_{K}^{k,\nabla^\perp} \psi_h,\phi_h -\Pi_{K}^{k,\nabla^\perp} \phi_h\Big).
		\end{align}
	\end{subequations}
	
\blue{To complete the definition of local bilinear forms \eqref{5.3.2.2a} and \eqref{5.3.2.2c}, we define the stabilizer forms $\mscr{S}_{D}^{K}(\cdot,\cdot)$ and $\mscr{S}_{\curl}^{K}(\cdot,\cdot)$ from Refs.~\refcite{VEMCN}, \refcite{VEMPB} and \refcite{MORABRINKMAN} satisfying \eqref{5.3.2.1} as: 
\begin{align}
		\mscr{S}_{D}^{K}(\psi_h,\phi_h) &:= \nu h_{K}^{-2} \sum_{j=1}^{N_{K}^{\text{dof}}} \boldsymbol{D}_i(\psi_h) \boldsymbol{D}_{i}(\phi_h) && \forall \psi_h, \phi_h \in Z_{h}^{k}(K), \label{stab:Delta}\\
		\mscr{S}_{\curl}^{K}(\psi_h,\phi_h) &:= \rho_{\bbeta,\sigma}^{K} \sum_{j=1}^{N_{K}^{\text{dof}}} \boldsymbol{D}_{i}(\psi_h) \boldsymbol{D}_{i}(\phi_h) && \forall \psi_h, \phi_h \in Z_{h}^{k}(K), \label{stab:nabla}
	\end{align}
where for each polygon $K \in \mscr{T}_{h}$:
	\begin{itemize}
		\item $N_{K}^{\text{dof}}$ denote the number of degrees freedom of $Z_{h}^{k}(K)$;
		\item The local degrees of freedom are collected in the vector
		$\boldsymbol{D}=\{\boldsymbol{D}_i \}_{i=1}^{N_{K}^{\text{dof}} }$,
		where $\boldsymbol{D}_i$ associates to each sufficiently smooth function
		$\phi$ its $i$-th local degree of freedom $\boldsymbol{D}_i(\phi)$;
		\item the parameter $\rho^K_{\boldsymbol{\beta},\sigma}$ is a multiplicative factor introduced to account for the physical magnitude.
	\end{itemize}}
	

\blue{In the following subsection, we establish stability bounds for the stabilizing
	forms $\mscr{S}_{D}^{K}(\cdot,\cdot)$ and
	$\mscr{S}_{\curl}^{K}(\cdot,\cdot)$. These bounds will be instrumental in
	establishing the well-posedness of the scheme \eqref{5.3.3.1}.}
\subsection{Stability bounds}
\blue{Although the stability properties of the stabilization form
	$\mscr{S}_{D}^{K}(\cdot,\cdot)$ were established in
	\cite{MR4521659} (see also \cite{HuangYu2021}), we recall them here in
	the form needed for the subsequent analysis. These estimates will be
	instrumental in deriving the stability properties of
	$\mscr{S}_{\curl}^{K}(\cdot,\cdot)$.}
	
\blue{We first establish the stability estimates for
	$\mscr{S}_{D}^{K}(\cdot,\cdot)$ and then use them to derive the
	corresponding bounds for $\mscr{S}_{\curl}^{K}(\cdot,\cdot)$.}	
	\blue{\begin{proposition}\label{Lemma:SD:stab1}
			Let $K \in  \mscr{T}_{h}$ and $\mscr{S}_{D}^{K}(\cdot,\cdot)$ be the bilinear form defined in~\eqref{stab:Delta}. Then, under assumptions~$\boldsymbol{\mathcal{A}1}$-$\boldsymbol{\mathcal{A}2}$, for all $\phi_h \in Z_h^k(K)$, the following property holds:
		$$\mscr{S}_{D}^{K}(\phi_h,\phi_h) \lesssim h^{-4}_{K}\|\phi_h\|^2_{0,K}  +h^{-2}_{K}|\phi_h|^2_{1,K}+ |\phi_h|^2_{2,K}.$$
	\end{proposition}}
\begin{proof}
\blue{From definition~\eqref{stab:Delta}, we have
\begin{equation}\label{split:I}
	\begin{aligned}
		\mscr{S}_{D}^{K}(\phi_h,\phi_h)
		&=\mscr{S}^{\boldsymbol{\partial}}(\phi_h,\phi_h)
		+\mscr{S}^{\boldsymbol{o}}(\phi_h,\phi_h)
		=
		h_{\E}^{-2}
		\sum_{i=1}^{N_{\partial}}
		\left(\boldsymbol{D}^{\boldsymbol{\partial}}_{i}(\phi_h)\right)^2
		+
		h_{\E}^{-2}
		\sum_{j=1}^{N_{o}}
		\left(\boldsymbol{D}^{\boldsymbol{o}}_{j}(\phi_h)\right)^2,
	\end{aligned}
\end{equation}
where $N_{\partial}=(3+k_E+k_{\bn})N_E, \: 
N_{o}=\dim(\mscr{M}_{k-4}(\E)),$
with $k_e=\dim(\mscr{M}_{t-3}(E))$ and
$k_{\bn}=\dim(\mscr{M}_{r-4}(E)).$ Here,
$\{\boldsymbol{D}^{\boldsymbol{\partial}}_{i}\}_{i=1}^{N_{\partial}}$
denotes the collection of boundary degrees of freedom
$\boldsymbol{D_1}$--$\boldsymbol{D_4}$, whereas
$\{\boldsymbol{D}^{\boldsymbol{o}}_{j}\}_{j=1}^{N_{o}}$
denotes the collection of interior degrees of freedom
$\boldsymbol{D_5}
$.
First, we will analyze the term $\mscr{S}^{\boldsymbol{o}}(\phi_h,\phi_h)$. Indeed,  by recalling that the polynomial basis  $\{m_j\}_{j=1}^{N_0}$ in $\boldsymbol{D^o}$ satisfies $\|m_j\|_{L^{\infty}(E)} \leq 1$, using the 
Cauchy-Schwarz inequality, and the fact $|\E| \simeq h^2_{\E}$, we obtain
\begin{equation}\label{eq:S^o}
	\begin{split}
		\mscr{S}^{\boldsymbol{o}}(\phi_h,\phi_h)
		&= h^{-2}_K\sum_{j=1}^{N_{o}}  \left( h^{-2}_K \int_K \phi_h \, m_j \, {\rm d}K  \right)^2 	\\
		&\leq h^{-2}_K h^{-4}_K \left(\int_K \phi_h \,\, {\rm d}K  \right)^2 \, \sum_{j=1}^{N_{o}} \|m_j\|_{L^{\infty}(K)}^2 \\
		& \lesssim h^{-2}_K h^{-4}_K |K| \|\phi_h\|_{0,K}^2 
		\lesssim  h^{-4}_K\|\phi_h\|_{0,K}^2.
	\end{split}
\end{equation}
Now, we will relabel the index for the term $\mscr{S}^{\boldsymbol{\partial}}(\phi_h,\phi_h)$, as follows
\begin{equation}\label{eq:S^partial}
	\begin{split}
		\mscr{S}^{\boldsymbol{\partial}}(\phi_h,\phi_h)
		&=h^{-2}_K\sum_{i_1=1}^{N_E} \big( \boldsymbol{D}^{\boldsymbol{\partial},\vb}_{i_1}(\phi_h) \big)^2 + h^{-2}_K\sum_{i_2=1}^{2N_E} \big(\boldsymbol{D}^{\boldsymbol{\partial}, \nabla}_{i_2}(\phi_h) \big)^2\\
		& = h^{-2}_K\sum_{i_1=1}^{N_E} \big( \boldsymbol{D}^{\boldsymbol{\partial},\vb}_{i_1}(\phi_h) \big)^2 + h^{-2}_K\sum_{i_2=1}^{2N_E} \big(\boldsymbol{D}^{\boldsymbol{\partial}, \nabla}_{i_2}(\phi_h) \big)^2\\
		&\quad + h^{-2}_K\sum_{i_3=1}^{k_{E}N_E} \big( \boldsymbol{D}^{\boldsymbol{\partial},I}_{i_3}(\phi_h) \big)^2
		+h^{-2}_K \sum_{i_4=1}^{k_{\bn}N_E} \big( \boldsymbol{D}^{\boldsymbol{\partial},II}_{i_4}(\phi_h) \big)^2\\
		&=:\mscr{S}^{\boldsymbol{\partial},\vb}(\phi_h,\phi_h) +\mscr{S}^{\boldsymbol{\partial},\nabla}(\phi_h,\phi_h) + \mscr{S}^{\boldsymbol{\partial},I}(\phi_h,\phi_h)+ \mscr{S}^{\boldsymbol{\partial},II}(\phi_h,\phi_h).
	\end{split}
\end{equation}
The next step is to bound each term from above. We begin with the first term. By the Sobolev inequality \cite[Equation (2.4)]{VEMSE}, we have
\begin{equation*}
	\begin{split}
	\mscr{S}^{\boldsymbol{\partial},\vb}(\phi_h,\phi_h)
		&\lesssim  h^{-2}_{K}\|\phi_h\|^2_{L^{\infty}(\dE)} \lesssim h^{-2}_{K} \|\phi_h\|^2_{L^{\infty}(\E)} \\
		&	 \lesssim  h^{-2}_{\E}(h^{-2}_{\E} \|\phi_h\|^2_{0,\E} + |\phi_h|^2_{1,\E} + h^2_{\E}|\phi_h|^2_{2,\E}) \lesssim  h^{-4}_{\E} \|\phi_h\|^2_{0,\E} + h^{-2}_{\E} |\phi_h|^2_{1,\E} + |\phi_h|^2_{2,\E}.
	\end{split}
\end{equation*} 
For the term $\mscr{S}^{\boldsymbol{\partial},\nabla}(\phi_h,\phi_h)$, using the decomposition of $\nabla\phi_h$ into its normal and tangential components, followed by the application of inverse and trace inequalities, we obtain
\begin{equation*}
	\begin{split}
		\mscr{S}^{\boldsymbol{\partial},\nabla}(\phi_h,\phi_h)&=h^{-2}_{K}\sum_{i_2=1}^{2N_E} \boldsymbol{D}^{\boldsymbol{\partial}, \nabla}_{i_2}(\phi_h)^2 \lesssim \|\nabla \phi_h\|^2_{L^{\infty}(\partial \E)}  \\
		&\lesssim  \|\nabla \phi_h \cdot \boldsymbol{n}_{K} \|^2_{L^{\infty}(\partial \E)} + \|\nabla \phi_h \cdot \boldsymbol{t}_{K} \|^2_{L^{\infty}(\partial \E)}\\
		& \lesssim h^{-1}_{E} (\|\nabla \phi_h \cdot \boldsymbol{n}_{K} \|^2_{L^{2}(\partial \E)} + \|\nabla \phi_h \cdot \boldsymbol{t}_{K} \|^2_{L^{2}(\partial \E)}) \\
			& \lesssim  h^{-1}_{E}(h^{-1}_{K}\|\nabla \phi_h\|^2_{0,K}+ h_{K}|\nabla \phi_h|^2_{1,K} )  \lesssim  h^{-2}_{K}|\phi_h|^2_{1,K}+ |\phi_h|^2_{2,K}.
	\end{split}
\end{equation*}
Next, for the third term, we use Cauchy-Schwarz inequality, the fact $\|m_j\|_{0,E}^2 \lesssim h_{E}$ for all $m_j \in \mscr{M}_{k-2}(E)$, to derive
\begin{equation*}
\begin{split}	
\mscr{S}^{\boldsymbol{\partial},I}(\phi_h,\phi_h)
&=  h^{-2}_K\sum_{i_3=1}^{k_{E}N_E} \big( \boldsymbol{D}^{\boldsymbol{\partial},I}_{i_3}(\phi_h) \big)^2
\leq h^{-2}_K\sum_{i_3=1}^{k_{E}N_E}\|m_j\|_{0,E}^2	\|\partial_{\bn_K} \phi_h\|_{0,E}^2	\\
	&\lesssim	  h^{-2}_K h_E \|\partial_{\bn_K} \phi_h\|_{0,E}^2 
\lesssim	h_\E^{-2}|\phi_h|^2_{1,\E}	+|\phi_h|^2_{2,\E},
\end{split}	
\end{equation*}
where we have used scaled trace inequality.}

\blue{Following similar arguments, we can prove that
\begin{equation*}
	\mscr{S}^{\boldsymbol{\partial},II}(\phi_h,\phi_h) \lesssim	h_\E^{-2}|\phi_h|^2_{1,\E}	+|\phi_h|^2_{2,\E}.
	\end{equation*}}
\blue{Therefore, combining \eqref{split:I}, \eqref{eq:S^o}, and \eqref{eq:S^partial} with the estimates derived above yields the desired result.}
%
%
\end{proof}
	
{A careful inspection of the arguments in Lemma 3.6 of Ref.~\refcite{HuangYu2021}
	shows that the following inequality holds for all functions in the full
	virtual space $Z_h^k(\E)$. Since the proof follows from the same arguments
	as those employed in Ref.~\refcite{HuangYu2021}, we omit the details. We also note that stability analysis for related spaces and
	discretizations have been developed in the more general framework
	considered in Ref.~\refcite{MR4521659}, which includes the case $m=n=2$ considered
	here.}
\blue{\begin{proposition}\label{Lemma:SD:stab2}
		Let $K \in \mscr{T}_{h}$ and $\mscr{S}_{D}^{K}(\cdot,\cdot)$ be the bilinear form defined in~\eqref{stab:Delta}. Then, under assumptions~$\boldsymbol{\mathcal{A}1}$-$\boldsymbol{\mathcal{A}2}$, for all $\phi_h \in Z_h^k(\E)$, the following property holds:
$$\mscr{S}_{D}^{K}(\phi_h,\phi_h) \gtrsim 	|\phi_h|^2_{2,K}.$$
		\end{proposition}}

\blue{The next step is to use Propositions~\ref{Lemma:SD:stab1} and~\ref{Lemma:SD:stab2} to show that the stabilization bilinear forms satisfy the corresponding stability properties.}

\blue{\begin{theorem}\label{thrm:SD:stab}
The stabilizing form $\mscr{S}_{D}^{K}(\cdot,\cdot)$ defined in~\eqref{stab:Delta} satisfies property \eqref{5.3.2.1a}.
\end{theorem}
\begin{proof}
To prove boundedness, we first add and subtract $\Pi_{K}^1\phi_h\in\mathbb{P}_1(K)$. Then, using the approximation and stability properties of the projector $\Pi_{K}^1$, together with the inverse inequality for polynomials, we obtain
\begin{equation*}
	\begin{split}
		|\phi_h|^2_{1,K} &\leq |\phi_h-\Pi_{K}^1 \phi_h |^2_{1,K}+ |\Pi_{K}^1 \phi_h|_{1,K}\\
		&\lesssim h^2_{K}|\phi_h |^2_{2,K}+ h^{-2}_{K}\|\Pi_{K}^1 \phi_h\|^2_{0,K} \lesssim h^2_{K}|\phi_h |^2_{2,K}+ h^{-2}_{K}\|\phi_h\|^2_{0,K}.
	\end{split}	
\end{equation*}		
Then, from  Proposition~\ref{Lemma:SD:stab1}, above bound and the Poincar\'e inequality~\eqref{Poincare:H2}, for all $\phi_h \in Z_h^{k}(K)  \cap {\rm Ker}(\Pi_{K}^{k,D})$, we have
	\begin{equation*}
	\begin{split}
\mscr{S}_{D}^{K}(\phi_h,\phi_h) &\lesssim h^{-4}_{K}\|\phi_h\|^2_{0,K} + |\phi_h|^2_{2,K} +h^{-2}_{K}|\phi_h|^2_{1,K}\\
&\lesssim h^{-4}_{K}\|\phi_h\|^2_{0,K} + |\phi_h|^2_{2,K}\\
	&\lesssim h^{-4}_{K}(h^{4}_{K}|\phi_h|^2_{2,K}) + |\phi_h|^2_{2,K} \lesssim |\phi_h|^2_{2,K}.
	\end{split}
\end{equation*}
The proof is complete by combining the above estimate and Proposition~\ref{Lemma:SD:stab2}. 
\end{proof}
}
	
	\blue{Now, we will prove the stability properties for the form $\mscr{S}_{\curl}^{K}(\cdot,\cdot)$ established in the following theorem.
	\begin{theorem}\label{thrm:SNabla:stab}
	The stabilizing form $\mscr{S}_{\curl}^{K}(\cdot,\cdot)$ defined in~\eqref{stab:nabla} satisfies property \eqref{5.3.2.1b}.
		\end{theorem}
	\begin{proof}
We will first prove the inequality $\mscr{S}_{\curl}^{K}(\phi_h,\phi_h) \lesssim |\phi_h|^2_{1,K}$. Indeed, for all $\phi_h \in Z^k_h(K)$, from Proposition~\ref{Lemma:SD:stab1} and inverse inequality for virtual functions (cf.~\cite[Lemma 4.3]{MR4521659}), we have that 	
\begin{equation}\label{S1}
	\begin{split}
		\mscr{S}_{\curl}^{K}(\phi_h,\phi_h) &\simeq h^{2}_{K}\mscr{S}_{D}^{K}(\phi_h,\phi_h)  \\
			&\lesssim h^{2}_{K}(h^{-4}_{K}\|\phi_h\|^2_{0,K} + |\phi_h|^2_{2,K} +h^{-2}_{K}|\phi_h|^2_{1,K})\\
			&\lesssim h^{-2}_{K}\|\phi_h\|^2_{0,K} + h^{2}_{K}(h^{-2}_{K}|\phi_h|^2_{1,K}) + |\phi_h|^2_{1,K}  \lesssim h^{-2}_{K}\|\phi_h\|^2_{0,K}+|\phi_h |^2_{1,K}.
	\end{split}
\end{equation}
On the other hand, if  $\phi_h \in Z_h^{k}(K)  \cap {\rm Ker}(\Pi_{K}^{k,\nabla^{\perp}})$, then by using the Poincaré inequality~\eqref{Poincare2:H1} and the second condition in the definition of projector $\Pi_{K}^{k,\nabla^{\perp}} $ (cf.~\eqref{revised2}), we have 
$\Pi^0_{\partial K}\phi_h=\Pi^0_{\partial K} (\Pi_{K}^{k,\nabla^{\perp}} \phi_h) = 0$, then 
\begin{equation}\label{S2}
h^{-2}_{K} \|\phi_h\|^2_{0,K} \lesssim  |\phi_h|^2_{1,K}.	
\end{equation}
Thus, from bounds~\eqref{S1} and~\eqref{S2}, we have 
\begin{equation*}
	\mscr{S}_{\curl}^{K}(\phi_h,\phi_h) \lesssim |\phi_h|^2_{1,K}.  
\end{equation*}
Next, we prove the converse inequality. Theorem~\ref{thrm:SD:stab} directly implies that 
\begin{equation}\label{Snabla1}
	\begin{split}
		\mscr{S}_{\curl}^{K}(\phi_h,\phi_h) &\simeq h^{2}_{K}\mscr{S}_{D}^{K}(\phi_h,\phi_h) \gtrsim h^{2}_{K} |\phi_h|^2_{2,K} \quad \forall \phi_h \in Z^k_h(K).
	\end{split}
\end{equation}
Furthermore, for any $\phi_h \in Z_h^{k}(K) \cap {\rm Ker}(\Pi_{K}^{k,\nabla^{\perp}})$, the first condition in the definition of the projection operator $\Pi_{K}^{k,\nabla^{\perp}}$ (cf.~\eqref{revised1}), yields
\begin{equation*}
\int_{\E} \partial_x \phi_h = 0 \qquad \text{and} \quad  \int_{\E} \partial_y \phi_h = 0.
\end{equation*}
Consequently, applying the Poincaré inequality~\eqref{Poincare1:H1} componentwise to $\nabla\phi_h$, we obtain
\begin{equation}\label{Snabla2}
	h^{-2}_{K} |\phi_h|^2_{1,K} \lesssim  |\phi_h|^2_{2,K}.    
\end{equation} 
Finally, combining estimates~\eqref{Snabla1} and~\eqref{Snabla2} leads to
\begin{equation*}
	\mscr{S}_{\curl}^{K}(\phi_h,\phi_h) \gtrsim |\phi_h|^2_{1,K},  
\end{equation*} 
which completes the proof.
\end{proof}}

	For all $\psi_h,\phi_h \in Z^k_h$,we define the following global discrete bilinear forms:
	\begin{subequations}\label{5.3.2.3}
		\begin{align}
			\label{5.3.2.3a}\mscr{A}^{h}(\psi_h,\phi_h) &:= \sum_{K \in \mscr{T}_{h}} \mscr{A}^{h,K}(\psi_h,\phi_h), &&
			\mscr{B}^{h}(\psi_h,\phi_h) := \sum_{K \in \mscr{T}_{h}} {\mscr{B}}^{h,K}(\psi_h,\phi_h), &&
			\mscr{C}^{h}(\psi_h,\phi_h) := \sum_{K \in \mscr{T}_{h}} \mscr{C}^{h,K}(\psi_h,\phi_h),
		\end{align}
	\end{subequations}
where $\mscr{B}^{h,K}(\psi_h,\phi_h) := \frac{1}{2} \big[\widetilde{\mscr{B}}^{h,K}(\psi_h,\phi_h)-\widetilde{\mscr{B}}^{h,K}(\phi_h,\psi_h)\big]$. The bilinear form $\mscr{B}^h(\cdot, \cdot)$ can be naturally extended to the continuous space $Z$. 

{The following result establishes the typical consistency, stability, and
	boundedness properties of the discrete local forms
	$\mscr{A}^{h,K}(\cdot,\cdot)$ and $\mscr{C}^{h,K}(\cdot,\cdot)$ (the latter being
	non-consistent for variable coefficients). Both properties follow from
	standard VEM arguments. In particular, we refer to~\cite[Proposition~3.5]{MORABRINKMAN} for a
	proof of the first bound. The second estimate is obtained by analogous
	techniques, replacing the tensor $\mathbb{K}^{-1}$ in Ref.~\refcite{MORABRINKMAN} with
	the quantity $\sigma - \frac{1}{2}\nabla \cdot \boldsymbol{\beta}$.
}
	\begin{theorem}\label{Theorem 5.3.2.1}
		For all $K \in \mscr{T}_{h}$, the local bilinear forms $\mscr{A}^{h,K}(\cdot,\cdot)$ and $\mscr{C}^{h,K}(\cdot,\cdot)$ defined in \eqref{5.3.2.2a} and \eqref{5.3.2.2c}, respectively, satisfy the following:
		\begin{enumerate}
			\item \textbf{Consistency:} For all elements $K \in \mscr{T}_{h}$ and $h>0$, we have
			\begin{align*}
				\mscr{A}^{h,K}(q,\phi_h) = \mscr{A}^{K} (q,\phi_h) \qquad \forall q \in \mathbb{P}_{k}(K) \quad \text{and} \quad \forall \phi_h \in Z_{h}^{k}(K).
			\end{align*}	
			\item \textbf{Stability:} For all elements $K \in \mscr{T}_{h}$ and $h>0$, we have
				\begin{align*}	
 \mscr{A}^{K}(\phi_h,\phi_h) \simeq \ &\mscr{A}^{h,K}(\phi_h,\phi_h)  \quad \forall \phi_h \in Z_{h}^{k}(K), &&
\mscr{C}^{K}(\phi_h,\phi_h) \simeq \ \mscr{C}^{h,K}(\phi_h,\phi_h)  \quad \forall \phi_h \in Z_{h}^{k}(K),
				\end{align*}
			where the hidden constants are independent of $h_K$.
		\end{enumerate}
	\end{theorem}
	\indent The proof of this theorem follows by standard arguments\cite{VEMCN, VEM1, VEMGSOC}. The next step is to construct a computable approximation of the linear functional for the right hand side terms.
	For this purpose, we define the following computable terms:
		\begin{align*}
            \mscr{L}_{\f,\u}^{h,K}(\phi_h) &:= \Big( \boldsymbol{\Pi}_{K}^{k-1} (\f+\uh),\boldsymbol{\Pi}_{K}^{k-1} \curl(\phi_h)\Big)_{K}, \\
            \mscr{L}_{\psi_d,\psi}^{h,K}(\tau_h) &:= \Big( \boldsymbol{\Pi}_{K}^{k-1} (\curl(\psi_h)-\curl(\psi_d)), \boldsymbol{\Pi}_{K}^{k-1} \curl(\tau_h)\Big)_{K}.
		\end{align*}
	\indent We consider the following approximations of the two linear functionals:
		\begin{align}
		\label{5.3.2.6} \mscr{L}_{\f,\u}^{h} (\phi_h) = \sum_{K \in \mscr{T}_{h}} \mscr{L}_{\f,\u}^{h,K} (\phi_h),\
		 \mscr{L}_{\psi_d,\psi}^{h}(\tau_h) = \sum_{K \in \mscr{T}_{h}} \mscr{L}_{\psi_d,\psi}^{h,K}(\tau_h), \forall \phi_h, \tau_h \in Z_h^k.
		\end{align}
	\indent The following lemma outlines key properties needed for the discrete bilinear forms introduced in \eqref{5.3.2.3} and \eqref{5.3.2.6}.
	\begin{lemma}\label{Lemma 5.3.2.2}
		There are positive constants $C_{\mscr{A}}, C_{\mscr{B}}, C_2, C_{\mscr{L}_{\f,\u}}$ and $C_{\mscr{L}_{\psi_d}}$ independent of $h$, such that the following conditions are satisfied for all $\psi_h, \phi_h$ in $Z_h^{k}$:
		\begin{subequations}\label{5.3.2.7n}
			\begin{align}
				\label{5.3.2.7anew}	\big|\mscr{A}^h(\psi_h,\phi_h)\big| &\le C_{\mscr{A}} \|\psi_h\|_{Z} \|\phi_h\|_{Z}, && \big|\mscr{C}^h(\psi_h,\phi_h)\big| \le C_{2} \|\psi_h\|_{Z} \|\phi_h\|_{Z}, \\
				\label{5.3.2.7c}	\mscr{A}^h(\phi_h,\phi_h) &\ge \alpha_0 \|\phi_h\|_{Z}^{2},\\ 
					\label{5.3.2.7cc}  \mscr{C}^h(\phi_h,\phi_h) &\ge \alpha_2 \kappa \|\phi_h\|_{1}^{2},\\
				\label{5.3.2.7d} \big|\mscr{L}_{\f,\u}^h(\phi_h)\big| &\le C_{\mscr{L}_{\f,\u}} (\|\f\|_{0} +\|\uh\|_{0}) \|\phi_h\|_{Z}, &&
				\big|\mscr{L}_{\psi_d,\psi}^{h}(\tau_h)\big| \le C_{\mscr{L}_{\psi_d,\psi}}( \|\mathrm{\bf curl}(\psi_h)\|_{0} + \|\mathrm{\bf curl}(\psi_d)\|_{0}) \|\tau_h\|_{Z},\\
				\label{5.3.2.7b}	\big|\mscr{B}^h(\psi_h,\phi_h)\big| &\le C_{\mscr{B}} \|\psi_h\|_{Z} \|\phi_h\|_{Z},\\
					\label{5.3.2.7g} \mscr{O}_{s}^h(\phi_h,\phi_h) &\ge \alpha_0 \|\phi_h\|_{Z}^{2}.
			\end{align}
		\end{subequations}
	\end{lemma}
	\begin{proof}
		The proof of (\ref{5.3.2.7c}-\ref{5.3.2.7b}) follow by arguments present in Refs.~\refcite{VEMCN}, \refcite{VEM1}, \refcite{VEMGSOC}. To prove \eqref{5.3.2.7b}, we use the Cauchy-Schwarz inequality and the estimate
		\begin{align*}
			|\phi_h|_{1} \le C_p \|\phi_h\|_{Z},
		\end{align*}
	where $C_p >0$ is a constant independent of $h$ and $K$. The proof of \eqref{5.3.2.7g} follows from \eqref{5.3.2.7c}, \eqref{5.3.2.7cc}
	and the fact that $\mscr{B}^h(\phi_h,\phi_h)=0$.
	\end{proof} 
	\subsection{VEM  Formulation}\label{VEM Discrete Formulation}
	We now present the discrete formulation for the continuous optimality system \eqref{5.2.1.7a}-\eqref{5.2.1.7c} using the discrete forms developed in Section~\ref{Virtual Element Methods}. The discrete first order optimality system is derived using the approach discussed in Ref.~\cite[Section~3.2]{CVEM1}. The VEM discretization of the continuous optimality system reads as follows: find $(\psi_h, \vartheta_h, \uh) \in Z_h^k \times Z_h^k \times \mscr{U}_{\boldsymbol{ad}}^{\boldsymbol{h}}$ such that
	\begin{subequations}\label{5.3.3.1}
		\begin{align}
			\label{5.3.3.1a} \mscr{O}_{s}^h(\psi_h,\phi_h) &:= \mscr{A}^h(\psi_h,\phi_h) + \mscr{B}^h(\psi_h,\phi_h) + \mscr{C}^h(\psi_h,\phi_h) = \mscr{L}_{\f,\u}^{h}(\phi_h),
			\\
			\label{5.3.3.1b} \mscr{O}_{a}^{h}(\tau_h,\vartheta_h)&:= \mscr{A}^h(\vartheta_h,\tau_h) - \mscr{B}^h(\vartheta_h,\tau_h) + \mscr{C}^h(\vartheta_h,\tau_h) \ \ = \mscr{L}_{\psi_d,\psi}^{h}(\tau_h),
			\\
			\label{5.3.3.1c} & \ \ \ \sum_{K \in \mscr{T}_{h}} \big(\boldsymbol{\Pi}_{K}^{k-1} \curl(\vartheta_h) + \lambda \uh,\tilde{\u}_{\boldsymbol{h}}-\uh\big)_{K} \ge \boldsymbol{0},
		\end{align}
	\end{subequations}
	{for all $(\phi_h, \tau_h,\tilde{\u}_{\boldsymbol{h}}) \in Z_h^{k} \times Z_h^{k} \times \mscr{U}_{\boldsymbol{ad}}^{\boldsymbol{h}}$}. The \textit{variational inequality} \eqref{5.3.3.1c} is equivalent to the following discrete projection formula:
	\begin{align}\label{dpf}
		\uh|_{K} := \mscr{P}_{[\boldsymbol{u_a},\boldsymbol{u_b}]} \left(-\frac{1}{\lambda}\boldsymbol{\Pi}_{K}^{{{k-1}}} \big(\curl(\vartheta_h)|_{K}\big)\right) \qquad \forall K \in \mscr{T}_h.
	\end{align}
	By a consequence of Lemma~\ref{Lemma 5.3.2.2} and Lax-Milgram theorem, we conclude the well-posedness of the discrete system \eqref{5.3.3.1} for a fixed control $\uh \in \mscr{U}_{\boldsymbol{ad}}^{\boldsymbol{h}}$.
	\section{Error Analysis} \label{Error Analysis}\setcounter{equation}{0}
	This section is dedicated to the development of a priori error analysis of the VEM scheme outlined in Sec.~\ref{Virtual Element Methods} for the optimal control problem.
	
	We define the following well defined broken Sobolev spaces $H^{\delta}(\Omega,\mscr{T}_{h})$ and broken $H^{\delta}$-seminorm for  $\delta >0$ as:
		\begin{align*}
			H^{\delta}(\Omega,\mscr{T}_{h}) &:= \big\{{\v \in L^{2}(\Omega): v|_{K} \in H^{\delta}(K)} \ \forall K \in \mscr{T}_{h} \big\},&&
			|\phi|_{\delta,h}^{2} := \sum_{K \in \mscr{T}_{h}} |\phi|_{\delta,K}^{2} \qquad \forall \phi \in H^{\delta}(\Omega,\mscr{T}_{h}).
	\end{align*}
An analogous construction holds for the vector- and matrix-valued spaces.
\subsection{Preliminary Results}
This subsection is devoted to a set of preliminary results that are essential for the forthcoming error analysis. We summarize approximation properties of polynomial and virtual element spaces, together with the corresponding projection operators, and we derive some useful identities that will be employed throughout the analysis.

We start by considering the following polynomial approximation result (cf. Ref.~\refcite{BSC}):
	\begin{lemma}\label{Lemma 5.4.2.2.2}
    If Assumption $\boldsymbol{\mathcal{A}2}$ is satisfied, then for each $v \in H^{\delta}(K)$, there exists $v_{\pi} \in \mathbb{P}_{m} (K), m \geq 0$ such that 
\begin{equation*}
\vert v - v_{\pi} \vert_{\ell,K} \lesssim h_{K}^{\delta - \ell}|v|_{\delta,K} \quad 0\leq \delta \leq m+1, \quad \ell=0,\ldots, [\delta],
\end{equation*}
where $[\delta]$ is the largest integer less than or equal to $\delta.$
	\end{lemma}
	
As a direct consequence of Lemma~\ref{Lemma 5.4.2.2.2}, by a component-wise application of the scalar approximation estimate, we obtain the following result.
\begin{lemma}\label{lemma:approx_vector_matrix}
Let Assumption~$\boldsymbol{\mathcal{A}2}$ be satisfied. Then the following
	approximation estimates hold for all $\ell = 0,\ldots,[\delta]$:
	\begin{equation*}
		\begin{aligned}
			\|\boldsymbol{v} - \boldsymbol{\Pi}^{k-1}_K \boldsymbol{v}\|_{\ell,K}
			&\lesssim h_K^{\delta-\ell} |\boldsymbol{v}|_{\delta,K},
			\quad \forall\, \boldsymbol{v} \in \boldsymbol{H}^{\delta}(K),
			\quad 0 \leq \delta \leq k,\\ 
			\|\boldsymbol{M} - \boldsymbol{\Pi}^{k-2}_{K,\mathrm{mat}} \boldsymbol{M}\|_{\ell,K}
			&\lesssim h_K^{\delta-\ell} |\boldsymbol{M}|_{\delta,K},
			\quad \forall\, \boldsymbol{M} \in [H^{\delta}(K)]^{2 \times 2},
			\quad 0 \leq \delta \leq k-1,
		\end{aligned}
	\end{equation*}
\end{lemma}

{Next, we recall the estimate for the interpolant $\phi_I \in Z_h^k$ of $\phi$ (see for instance~\cite{VEMCN,VEMPB}).
	\blue{
\begin{proposition}\label{VEM-approx}
	Assume that $\boldsymbol{\mathcal{A}1}$ and $\boldsymbol{\mathcal{A}2}$ are satisfied. Then, {for each $\phi \in H^{\xi}(\Omega)$}, there exists
	$\phi_I \in Z_h^k$, such that
	\[
	\|\phi - \phi_I\|_{m,\Omega}  \lesssim  \, h^{\xi-m} |\phi|_{\xi,\Omega},
	\qquad m = 0,1,2, \qquad 
	2 \le \xi \le k+1, \qquad k \ge 2.
	\]
\end{proposition}}

The following lemmas are instrumental in deriving the error bounds.
    \begin{lemma}\label{lemma:B1}
 For all $\psi_h, \phi_h \in Z_h^k$ holds
 \footnotesize{
 \begin{equation*}
    \begin{split}
  \widetilde{\mscr{B}}^{K}(\psi_h,\phi_h) - \widetilde{\mscr{B}}^{h,K}(\psi_h,\phi_h) 
  &=  \big(\big(\textbf{I}-  \boldsymbol{\Pi}_{K,\text{mat}}^{k-2}\big) ( \nabla \mathrm{\bf curl}(\psi_h)), \big(\textbf{I}-  \boldsymbol{\Pi}_{K,\text{mat}}^{k-2}  \big)(\mathrm{\bf curl}(\phi_h) \bbeta^T)  \big)_{K} \\
&  \quad + 
  \big(\big( \boldsymbol{\Pi}_{K,\text{mat}}^{k-2}-\textbf{I}\big)  (\nabla \mathrm{\bf curl}(\psi_h))  \bbeta , \big(\textbf{I}-  \boldsymbol{\Pi}_{K}^{k-1}  \big) \mathrm{\bf curl}(\phi_h) \big)_{K} \\
 & \quad + \big(\big(\textbf{I} - \boldsymbol{\Pi}_{K}^{k-2}\big)  (\nabla  \mathrm{\bf curl}(\psi_h) \bbeta),  \big( \textbf{I}-  \boldsymbol{\Pi}_{K}^{k-1}  \big) \mathrm{\bf curl}(\phi_h)\big)_{K}.   
    \end{split} 
 \end{equation*}
}
    \end{lemma}
    \begin{proof}
   Adding and subtracting adequate terms, we have  
   \begin{equation*}
 \begin{split}
 & \widetilde{\mscr{B}}^{K}(\psi_h,\phi_h) - \widetilde{\mscr{B}}^{h,K}(\psi_h,\phi_h) \\
&=  \big(\big(  \nabla \curl(\psi_h) \big) \bbeta , \curl(\phi_h)\big)_{K}-\big(\big( \boldsymbol{\Pi}_{K,\text{mat}}^{k-2}  \nabla \curl(\psi_h) \big) \bbeta, \boldsymbol{\Pi}_{K}^{k-1} \curl(\phi_h)\big)_{K}\\
 &=  \big(\big(  \nabla \curl(\psi_h) \big) \bbeta , \curl(\phi_h)\big)_{K} -\big(\big( \boldsymbol{\Pi}_{K,\text{mat}}^{k-2}  \nabla \curl(\psi_h) \big) \bbeta , \boldsymbol{\Pi}_{K}^{k-1} \curl(\phi_h)\big)_{K}\\
 & \quad + \big(\big( \boldsymbol{\Pi}_{K,\text{mat}}^{k-2}  \nabla \curl(\psi_h) \big) \bbeta, \curl(\phi_h)\big)_{K} - \big(\big( \boldsymbol{\Pi}_{K,\text{mat}}^{k-2}  \nabla \curl(\psi_h) \big) \bbeta,  \curl(\phi_h)\big)_{K}\\
 &=  \big(\big(  \nabla \curl(\psi_h) -  \boldsymbol{\Pi}_{K,\text{mat}}^{k-2}  \nabla \curl(\psi_h)\big) \bbeta , \curl(\phi_h)\big)_{K}\\
 & \quad + \big(\big( \boldsymbol{\Pi}_{K,\text{mat}}^{k-2}  \nabla \curl(\psi_h) \big) \bbeta, \curl(\phi_h)- \boldsymbol{\Pi}_{K}^{k-1} \curl(\phi_h)\big)_{K}\\
 &=  \big(   \big( \textbf{I}-  \boldsymbol{\Pi}_{K,\text{mat}}^{k-2}\big) ( \nabla \curl(\psi_h))  , \curl(\phi_h)  \bbeta^T\big)_{K}\\
 & \quad + \big( \big( \boldsymbol{\Pi}_{K,\text{mat}}^{k-2} - \textbf{I} \big) (\nabla \curl(\psi_h))  \bbeta, \big( \textbf{I}-  \boldsymbol{\Pi}_{K}^{k-1}  \big) \curl(\phi_h)\big)_{K}\\
 & \quad + \big(\nabla \curl(\psi_h) \bbeta,  \big( \textbf{I}-  \boldsymbol{\Pi}_{K}^{k-1}  \big) \curl(\phi_h)\big)_{K}.
 \end{split}   
   \end{equation*} 
 Thus, we obtain the desired result by using the orthogonality properties of the involved $L^2$-projections. 
    \end{proof}
      \begin{lemma}\label{lemma:B2}
      For all $\psi_h, \phi_h \in Z_h^k$ holds
 \begin{equation*}
    \begin{split}
  \widetilde{\mscr{B}}^{h,K}(\phi_h,\psi_h) - \widetilde{\mscr{B}}^{K}(\phi_h,\psi_h) 
  &=  \Big(  \big( \boldsymbol{\Pi}_{K,\text{mat}}^{k-2}  \nabla \mathrm{\bf curl}(\phi_h)\big) \bbeta ,   \big(   \boldsymbol{\Pi}_{K}^{k-1}- \textbf{I}  \big) \mathrm{\bf curl}(\psi_h) \Big)_{K}    \\
&  \hspace{-1cm} 
  + \Big(  \big( \boldsymbol{\Pi}_{K,\text{mat}}^{k-2}-\textbf{I}\big)  (\nabla \mathrm{\bf curl}(\phi_h)),  \big(   \boldsymbol{\Pi}_{K,\text{mat}}^{k-2} - \textbf{I} \big)(\mathrm{\bf curl}(\psi_h)  \bbeta^T) \Big)_{K}.   
   \end{split} 
 \end{equation*} 
    \end{lemma}
    \begin{proof}
    The proof follows by adding and subtracting adequate terms and employing the orthogonality properties of the involved $L^2$-projections. 
    \end{proof}
	\subsection{A Priori Error Estimates}\label{Subsection:Apriori:error}
	In this subsection, we develop a priori error estimates for the state, co-state
	and control variable. Firstly, we consider the so-called auxiliary equations. For arbitrary functions $\uone,\utwo \in \L^{2}(\Omega)$, let $\psi_h(\uone), \psi_h(\utwo), \vartheta_h(\psi_1)$ and $\vartheta_h(\psi_2)$ be the solutions of the following systems:
	\begin{align}
		\label{5.4.2.1}
		{\mscr{O}_{s}^h(\psi_h(\uone),\phi_h)} &= \mscr{L}_{\f,\uone}^{h}(\phi_h) &&\forall \phi_h \in Z_h^{k},\\ 
	\label{5.4.2.1.0}	{\mscr{O}_{s}^h(\psi_h(\utwo),\phi_h)} &= \mscr{L}_{\f,\utwo}^{h}(\phi_h) && \forall \phi_h \in Z_h^{k},
		\\
		\label{5.4.2.2} 
		{\mscr{O}_{a}^{h}(\tau_h,\vartheta_h(\psi_1))} &= \mscr{L}_{\psi_d,\psi_1}^{h}(\tau_h) &&\forall \tau_h \in Z_h^{k},\\
		\label{5.4.2.22} {\mscr{O}_{a}^{h}(\tau_h,\vartheta_h(\psi_2))} &= \mscr{L}_{\psi_d,\psi_2}^{h}(\tau_h) && \forall \tau_h \in Z_h^{k}.
	\end{align} 
	\indent The discrete linear functionals $\mscr{L}_{\f,\uone}^{h}, \mscr{L}_{\f,\utwo}^{h}, \mscr{L}_{\psi_d,\psi_1}^{h}$ and $\mscr{L}_{\psi_d,\psi_2}^{h}$ are defined as:
	\begin{align*}
		\mscr{L}_{\f,\uone}^{h} (\phi_h) &:= \sum_{K \in \mscr{T}_{h}} \mscr{L}_{\f,\uone}^{h,K} (\phi_h), && 
		\mscr{L}_{\f,\utwo}^{h} (\phi_h) := \sum_{K \in \mscr{T}_{h}} \mscr{L}_{\f,\utwo}^{h,K} (\phi_h), \\ 
		\mscr{L}_{\psi_d,\psi_1}^{h}(\tau_h) &:= \sum_{K \in \mscr{T}_{h}} \mscr{L}_{\psi_d,\psi_1}^{h,K}(\tau_h),
		&&\mscr{L}_{\psi_d,\psi_2}^{h}(\tau_h) := \sum_{K \in \mscr{T}_{h}} \mscr{L}_{\psi_d,\psi_2}^{h,K}(\tau_h),
	\end{align*}
where 
	\begin{align*}
	  \mscr{L}_{\f,\boldsymbol{u_{i}}}^{h,K} (\phi_h) &:= \Big( \boldsymbol{\Pi}_{K}^{k-1} (\f+\boldsymbol{u_{i}}),\boldsymbol{\Pi}_{K}^{k-1} \curl(\phi_h)\Big)_{K} &&\forall \phi_h \in Z_h^{k}, \\ 
	 \mscr{L}_{\psi_d,\psi_i}^{h,K}(\tau_h) &:= \Big( \boldsymbol{\Pi}_{K}^{k-1} (\curl(\psi_i)-\curl(\psi_d)), \boldsymbol{\Pi}_{K}^{k-1} \curl(\tau_h)\Big)_{K} &&\forall \tau_h \in Z_h^{k},
\end{align*}
	for $i=1,2$. Building on the auxiliary problems discussed above, we now establish the following lemma, which plays a crucial role in the subsequent error analysis.
	\begin{lemma}\label{Lemma 5.4.2.1}
		Suppose $(\psi_h(\uone), \vartheta_h(\psi_1))$ and $(\psi_h(\utwo), \vartheta_h(\psi_2))$ are solutions of the auxiliary problems (\ref{5.4.2.1}-\ref{5.4.2.22}). Then, we have the following estimates:
		\begin{align}\label{5.4.2.3}
			\|\psi_h(\uone) - \psi_h(\utwo)\|_{Z} &\le C_{1}\|\uone-\utwo\|_{0}, &&
			 \|\vartheta_h(\psi_1) - \vartheta_h(\psi_2)\|_{Z} \le C_{2}\|\psi_1-\psi_2\|_{0},
		\end{align}
		where $C_1 := \alpha_{0}^{-1}C_{\mscr{L}_{\f,\u}}$ and $C_2:= \alpha_{0}^{-1} C_{\mscr{L}_{\psi_d,\psi}}$ are positive constants.
	\end{lemma}
	\begin{proof}
		On subtracting \eqref{5.4.2.1.0} from \eqref{5.4.2.1}, we get
		\begin{align*}
			{\mscr{O}_{s}^h(\psi_h(\uone)-\psi_h(\utwo),\phi_h)} = \mscr{L}_{\f,\uone}^{h}(\phi_h) - \mscr{L}_{\f,\utwo}^{h}(\phi_h).
		\end{align*}
		\indent Now, we substitute $\phi_h = \psi_h(\uone)-\psi_h(\utwo)$ and use Lemma~\ref{Lemma 5.3.2.2} to obtain
		\begin{align*}
			\alpha_{0} \|\psi_h(\uone)-\psi_h(\utwo)\|_{Z}^{2} \le  \mscr{L}_{\f,\uone}^{h}(\psi_h(\uone)-\psi_h(\utwo)) - \mscr{L}_{\f,\utwo}^{h}(\psi_h(\uone)-\psi_h(\utwo)).
		\end{align*}
		\indent By using the definition of functionals $\mscr{L}_{\f,\uone}$ and $\mscr{L}_{\f,\utwo}$ and the Cauchy-Schwarz inequality, we have 
		\begin{align*}
			\alpha_{0} \|\psi_h(\uone)-\psi_h(\utwo)\|_{Z}^{2} &\le C_{\mscr{L}_{\f,\u}} \|\uone-\utwo\|_{0} \|\psi_h(\uone)-\psi_h(\utwo)\|_{Z}, \\
			\|\psi_h(\uone)-\psi_h(\utwo)\|_{Z} &\le C_{1} \|\uone-\utwo\|_{0}.
		\end{align*}
		\indent The second estimate can be derived following a similar sequence of steps.
	\end{proof}
	{
	\begin{lemma}\label{Lemma 5.4.2.1.0}
		Let $\psi$ and $\vartheta$ be the solutions of state and co-state systems \eqref{5.2.1.7a} and \eqref{5.2.1.7b}, respectively. Then, for $k\ge2$, we have the following estimates for $\star:=\min\{k-2,1\}$:
			\begin{align}
				\label{5.4.2.3.0a}	h^{\tilde{s}} \|\psi - \psi_h(\u)\|_{Z} + \|\psi - \psi_h(\u)\|_{1} &\lesssim h^{\min\{s,k-1\}+\tilde{s}}\big(\|\psi\|_{2+s} + {|\f+\u|_{\star,h} \big)},\\
			\label{5.4.2.3.0b} h^{\tilde{s}} \|\vartheta - \vartheta_h(\psi)\|_{Z} + \|\vartheta - \vartheta_h(\psi)\|_{1} &\lesssim h^{\min\{s,k-1\}+\tilde{s}}\big(\|\vartheta\|_{2+s} + |\mathrm{\bf curl}(\psi-\psi_d)|_{\star,h}\big),
			\end{align}
		where  $\psi_h(\u)$ and $\vartheta_h(\psi)$ are the auxiliary variables satisfying \eqref{5.4.2.1} (for $\uone = \u$) and \eqref{5.4.2.2} (for $\psi_{1} = \psi$), respectively,
		with $s\in(1/2, 2]$ as in Lemma~\ref{addReg} and Remark~\ref{remark:aditional:regu2}, and $\tilde{s}\in(1/2, 1]$
		satisfies \eqref{ineq:reg:add:aux}.
	\end{lemma}}
	\begin{proof}
   In what follows, we establish the error estimates \eqref{5.4.2.3.0a} for the variable $\psi$.  The second estimate \eqref{5.4.2.3.0b} for the co-state follows by using similar arguments.

We divide the proof into two steps. The first consists of deriving error bounds in the $Z$-norm, while the second is devoted to establishing the $H^1$-estimate.
   
\paragraph{Estimate in $Z$-norm.} 	Let $\psi_I \in Z_h^k$ denote the interpolant of $\psi \in H^{2+s}(\Omega)$ satisfying Proposition~\ref{VEM-approx}, and define $\delta_h := \psi_h(\u) - \psi_I$. Then, we have
				\begin{align}\label{ineq:trian}
					\|\psi-\psi_{h}(\u)\|_{Z} \le \|\psi-\psi_{I}\|_{Z} + \|\delta_h\|_{Z}.
				\end{align}
				\indent By an application of \eqref{5.3.2.7g}, continuous and discrete systems \eqref{5.2.1.7a} and \eqref{5.4.2.1}, respectively, for $\psi_{\pi} \in \mathbb{P}_{k}(K)$ we obtain
				\begin{align*}
					\alpha_{0} \|\delta_h\|_{Z}^{2} &\le {\mscr{O}_{s}^h(\delta_h,\delta_h)} = {\mscr{O}_{s}^h(\psi_h(\u),\delta_h)} - {\mscr{O}_{s}^h(\psi_I,\delta_h)} 
					= \mscr{L}_{\f,\u}^{h}(\delta_h) - {\mscr{O}_{s}^h(\psi_I,\delta_h)} \\
					&= \mscr{L}_{\f,\u}^{h}(\delta_h) - \mscr{L}_{\f,\u}(\delta_h) + {\mscr{O}_{s}(\psi,\delta_h)} - {\mscr{O}_{s}^h(\psi_I,\delta_h)}\\
					&= \mscr{L}_{\f,\u}^{h}(\delta_h) - \mscr{L}_{\f,\u}(\delta_h) + \sum_{K \in \mscr{T}_{h}} \big[\mscr{A}^{K}(\psi-\psi_{\pi},\delta_h) - \mscr{A}^{h,K}(\psi_I-\psi_{\pi},\delta_h) \big] \\
					& \ \ \ + \sum_{K \in \mscr{T}_{h}} \big[\mscr{B}^{K}(\psi,\delta_h) - \mscr{B}^{h,K}(\psi_I,\delta_h)\big] + \sum_{K \in \mscr{T}_{h}} \big[\mscr{C}^{K}(\psi,\delta_h) - \mscr{C}^{h,K}(\psi_I,\delta_h)\big]\\
					&:= T_{\mscr{L}_{\f,\u}} + \sum_{K \in \mscr{T}_{h}} \big(T_{\mscr{A}} + T_{\mscr{B}} + T_{\mscr{C}}\big).
				\end{align*}
	\indent By following Proposition~8.4 and Lemma~8.10 from Ref.~\refcite{MORAOSEEN}, $T_{\mscr{L}_{\f,\u}}$ and $T_{\mscr{C}}$ are bounded as
				\begin{align*}
					T_{\mscr{L}_{\f,\u}} & \lesssim \big\|\mscr{L}_{\f,\u}-\mscr{L}_{\f,\u}^{h}\big\| \|\delta_h\|_{Z},\\
					T_{\mscr{C}} &\lesssim|\psi-\psi_I|_{2,K}+|\psi-\psi_{\pi}|_{1,K}+h_{K}^{\min\{s,k-1\}+1}\|\psi\|_{2+s,K}.
				\end{align*}
				\indent For $T_{\mscr{B}}$, we add subtract $\mscr{B}^{K}(\psi_I,\delta_h)$ to obtain
\begin{equation}\label{term:TB}
              \begin{split}
			T_{\mscr{B}} &= \mscr{B}^{K}(\psi,\delta_h) - \mscr{B}^{h,K}(\psi_I,\delta_h) \\
			   &= \mscr{B}^{K}(\psi - \psi_I,\delta_h) + \big(\mscr{B}^{K}(\psi_I,\delta_h) - \mscr{B}^{h,K}(\psi_I,\delta_h)\big)\\
             &= \mscr{B}^{K}(\psi - \psi_I,\delta_h) + \frac{1}{2} \Big(\widetilde{\mscr{B}}^{K}(\psi_I,\delta_h) - \widetilde{\mscr{B}}^{h,K}(\psi_I,\delta_h)\Big)  + \frac{1}{2} \Big( \widetilde{\mscr{B}}^{h,K}(\delta_h,\psi_I) - \widetilde{\mscr{B}}^{K}(\delta_h,\psi_I)\Big)\\
             & \leq C_{\bbeta}\|\psi - \psi_I\|_{2,K} \|\delta_h\|_{2,K}  + \frac{1}{2}T_{\mscr{B}_1}+  \frac{1}{2} T_{\mscr{B}_2}.
                \end{split}
				\end{equation} 
Now by taking $\psi_h = \psi_I $ and $\phi_h = \delta_h$ in Lemma~\ref{lemma:B1}  and applying the Cauchy-Schwarz inequality, we have 
\begin{equation*}
\begin{split}
T_{\mscr{B}_1} &\leq
  \big \|  \big( \textbf{I}-  \boldsymbol{\Pi}_{K,\text{mat}}^{k-2}\big) ( \nabla   \curl(\psi_I)) \big \|_{K} \big \|   \big( \textbf{I}-  \boldsymbol{\Pi}_{K,\text{mat}}^{k-2}  \big)(\curl (\delta_h) \bbeta^T)  \big \|_{K}    \\
&  \quad + 
  \big \|\big(\big( \boldsymbol{\Pi}_{K,\text{mat}}^{k-2}-\textbf{I}\big)  (\nabla \curl(\psi_I)) \big)  \bbeta \big \|_{K} \big \|  \big( \textbf{I}-  \boldsymbol{\Pi}_{K}^{k-1}  \big)\curl(\delta_h) \big \|_{K}    \\
 & \quad + \big \| \big(\textbf{I} - \boldsymbol{\Pi}_{K}^{k-2}\big)  (\nabla  \curl (\psi_I) \bbeta)\big \|_{K} \big \|  \big( \textbf{I}-  \boldsymbol{\Pi}_{K}^{k-1}  \big)\curl (\delta_h) \big \|_{K},
\end{split}
\end{equation*}

First, from Lemma~\ref{lemma:approx_vector_matrix} we observe that 
\begin{equation*}
\begin{split}
& \big \|  \big( \textbf{I}-  \boldsymbol{\Pi}_{K,\text{mat}}^{k-2}\big) ( \nabla   \curl(\psi_I)) \big \|_{K}\leq  
\big \| \nabla \curl(\psi_I-   \psi) \big \|_{K} \\
&+ \big \|  \big( \textbf{I}-  \boldsymbol{\Pi}_{K,\text{mat}}^{k-2}\big) ( \nabla   \curl(\psi)) \big \|_{K} + \big \| \boldsymbol{\Pi}_{K,\text{mat}}^{k-2} ( \nabla   \curl(\psi-\psi_I)) \big \|_{K}\\
 & \lesssim
 \big \|  \big( \textbf{I}-  \boldsymbol{\Pi}_{K,\text{mat}}^{k-2}\big) ( \nabla   \curl(\psi)) \big \|_{K} +
\|\psi-\psi_I  \|_{2,K}\\
 & \lesssim h_{K}^{\min\{s,k-1\}} \|\psi\|_{2+s,K}.
\end{split}
\end{equation*}

Analogously, we obtain
\begin{equation*}
\begin{split}
& \big \| \big( \textbf{I}-  \boldsymbol{\Pi}_{K}^{k-2}\big) ( (\nabla \curl(\psi_I))\bbeta) \big \|_{K}
\lesssim \| \nabla \curl(\psi_I-   \psi) \bbeta \big \|_{K} \\
&\quad +  \big \|  \big( \textbf{I}-  \boldsymbol{\Pi}_{K,\text{mat}}^{k-2}\big) (\nabla   \curl(\psi))\bbeta \big \|_{K} +
\big \| \boldsymbol{\Pi}_{K}^{k-2} ( \nabla \curl(\psi-\psi_I) \bbeta) \big \|_{K}\\
& \lesssim\big \|  \big( \textbf{I}-  \boldsymbol{\Pi}_{K,\text{mat}}^{k-2}\big) (\nabla   \curl(\psi))\bbeta \big \|_{K} + \|\psi-\psi_I  \|_{2,K}\\
 & \lesssim h_{K}^{\min\{s,k-1\}} \|\psi\|_{2+s,K},
\end{split}
\end{equation*}
and 
\begin{equation*}
\begin{split}
    \big \|\big(\big( \boldsymbol{\Pi}_{K,\text{mat}}^{k-2}-\textbf{I}\big)  (\nabla \curl(\psi_I)) \big)  \bbeta \big \|_{K} &\lesssim \big \|  \big( \textbf{I}-  \boldsymbol{\Pi}_{K}^{k-2}\big) ( \nabla \curl(\psi_I)) \big \|_{K} \lesssim h_{K}^{\min\{s,k-1\}} \|\psi\|_{2+s,K}. 
\end{split}
\end{equation*}

Moreover, 
\begin{equation*}
\begin{split}
\big \|   \big( \textbf{I}-  \boldsymbol{\Pi}_{K,\text{mat}}^{k-2}  \big)(\curl (\delta_h) \bbeta^T)  \big \|_{2,K}   
&\lesssim h_{K} \|\curl(\delta_h)\|_{1,K},\\
\big \|\big( \textbf{I}- \boldsymbol{\Pi}_{K}^{k-2}  \big)(\curl (\delta_h))  \big \|_{K}   
&\lesssim h_{K} \|\curl(\delta_h)\|_{1,K}.\\
\end{split}
\end{equation*}
Thus, 
\begin{equation}\label{bound:TB1}
T_{\mscr{B}_1} \lesssim h_{K}^{\min\{s,k-1\}+1} \|\psi\|_{2+s,K}\|\delta_h\|_{2,K}.
\end{equation}

On the other hand, taking $\psi_h = \psi_I$ and $\phi_h = \delta_h$ in Lemma~\ref{lemma:B2},   applying the Cauchy-Schwarz inequality, and approximation properties, we obtain  
\begin{equation*}
\begin{split}
T_{\mscr{B}_2} &\leq \big\|\boldsymbol{\Pi}_{K,\text{mat}}^{k-2}  \nabla   \curl(\delta_h)\big\|_{K}  \big\|  \big(   \boldsymbol{\Pi}_{K}^{k-1}- \textbf{I}  \big)(\curl (\psi_I))  \big\|_{K} \\
&  \: + 
  \big\|  \big( \boldsymbol{\Pi}_{K,\text{mat}}^{k-2}-\textbf{I}\big)(\nabla \curl(\delta_h))\big\|_{K} \big\|  \big(   \boldsymbol{\Pi}_{K,\text{mat}}^{k-2} - \textbf{I} \big)(\curl (\psi_I)  \bbeta^T) \big \|_{K}.
\end{split}
\end{equation*}

Similarly, by using the approximation properties of projection $\boldsymbol{\Pi}_{K}^{k-1}$, we have 
\begin{equation*}
\begin{split}
& \big\|  \big(   \boldsymbol{\Pi}_{K}^{k-1}- \textbf{I}  \big)\curl (\psi_I)  \big\|_{K}\leq  
\big \| \curl(\psi_I-   \psi) \big \|_{K} \\
&+ \big \|  \big( \textbf{I}-  \boldsymbol{\Pi}_{K}^{k-1}\big)   \curl(\psi) \big \|_{K} + \big \| \boldsymbol{\Pi}_{K}^{k-1}   \curl(\psi-\psi_I) \big \|_{K}\\
 &  \lesssim \big \|  \big( \textbf{I}-  \boldsymbol{\Pi}_{K}^{k-1}\big)   \curl(\psi)\big \|_{K} +
|\psi-\psi_I  |_{1,K} \leq  h_{K}^{\min\{s,k-1\}+1} \|\psi\|_{2+s,K}.
\end{split}
\end{equation*}
and 
\begin{equation*}
\big\| \big(\boldsymbol{\Pi}_{K,\text{mat}}^{k-2} - \textbf{I} \big)(\curl \psi_I  \bbeta^T) \big \|_{K}   \lesssim \Big(h_{K}^{\min\{s,k-1\}+1}+ h_{K}^{\min\{s,k-1\}}\Big)  \|\psi\|_{2+s,K}.
\end{equation*}
Moreover, 
\begin{equation*}
\begin{split}
\big\|\boldsymbol{\Pi}_{K,\text{mat}}^{k-2} \nabla \curl(\delta_h)\big\|_{K}   
\lesssim \|\delta_h\|_{2,K},\qquad \big \|   \big(  \boldsymbol{\Pi}_{K,\text{mat}}^{k-2} -\textbf{I}\big)(\nabla \curl (\delta_h))  \big \|_{K}   \lesssim  \|\delta_h\|_{2,K}.
\end{split}
\end{equation*}

Therefore 
\begin{equation}\label{bound:TB2}
\begin{split}
T_{\mscr{B}_2} (\delta_h,\psi_I)  \lesssim (h_{K}^{\min\{s,k-1\}}+ h_{K}^{\min\{s,k-1\}+1}) \|\psi\|_{2+s,K}\|\delta_h\|_{2,K}. 
\end{split}
\end{equation}

Finally, by combining the bounds \eqref{term:TB}, \eqref{bound:TB1}, and \eqref{bound:TB2}, we have  
\begin{equation*}
\begin{split}
\sum_{K \in \mscr{T}_h}T_{\mscr{B}} &\leq C_{\mscr{B}} \|\psi- \psi_I\|_{Z} \|\delta_h\|_{Z} + C h^{\min\{s,k-1\}+1} \|\psi\|_{2+s}\|\delta_h\|_{Z} \\
& \quad + C \Big(h^{\min\{s,k-1\}}+ h^{\min\{s,k-1\}+1}\Big) \|\psi\|_{2+s}\|\delta_h\|_{Z}\\
& \lesssim \Big(h^{\min\{s,k-1\}}+ h^{\min\{s,k-1\}+1} \Big) \|\psi\|_{2+s}\|\delta_h\|_{Z}.
\end{split}
\end{equation*}

\indent Now, using continuity of the bilinear forms $\mscr{A}^{K}$ and $\mscr{A}^{h,K}$, we obtain
\begin{align*}
T_{\mscr{A}} &= \mscr{A}^{K}(\psi-\psi_{\pi},\delta_h) - \mscr{A}^{h,K}(\psi_I-\psi_{\pi},\delta_h) 
 \lesssim \big(|\psi-\psi_{\pi}|_{2,K} + |\psi_I-\psi_{\pi}|_{2,K} \big)|\delta_h|_{2,K}.
				\end{align*}
				\indent Combining the estimates for all the terms, we have
				\begin{align*}
					\hat{\alpha} \|\delta_h\|_{Z} &\lesssim \big\|\mscr{L}_{\f,\u}-\mscr{L}_{\f,\u}^{h}\big\| + \|\psi-\psi_{I}\|_{Z} + |\psi-\psi_{\pi}|_{1,h} \\
					& \qquad \quad + |\psi-\psi_{\pi}|_{2,h} + h^{\min\{s,k-1\}}\|\psi\|_{2+s}.
				\end{align*}
  The desired estimate follows from the above bound, \eqref{ineq:trian}, Proposition~\ref{VEM-approx} and Lemma~\ref{Lemma 5.4.2.2.2}.               
  
\paragraph{Estimate in $H^{1}$-norm.} Now, let $\phi \in Z$ be the solution of the following auxiliary variational problem: find $\phi \in Z$ such that
\begin{align}\label{5.4.2.3.02}
\mscr{O}_{s}(v,\phi) = \big(\nabla (\psi - \psi_h(\u)), \nabla v\big) \qquad \forall v \in Z.
\end{align}
By an application of Theorem~\ref{Theorem 5.2.1.3}, there exists $\tilde{s} \in (1/2,1]$ and a positive constant $C$, such that $\phi \in H^{2+\tilde{s}}(\Omega)$ and
\begin{align}\label{ineq:reg:add:aux}
\|\phi\|_{2+\tilde{s}} \lesssim |\psi - \psi_h(\u)|_{1}.
\end{align}

By taking $v := (\psi - \psi_h(\u)) \in Z$ as a test function in \eqref{5.4.2.3.02}, we obtain
\begin{equation}\label{5.4.2.3.04}
\begin{split}
|\psi - \psi_h(\u)|_{1}^{2} &\le \mscr{O}_{s}(\psi - \psi_h(\u),\phi) \\
 &= \mscr{O}_{s}(\psi - \psi_h(\u),\phi - \phi_{I})  + \mscr{O}_{s}(\psi - \psi_h(\u),\phi_{I}) \\
 &= \mscr{O}_{s}(\psi - \psi_h(\u),\phi - \phi_{I})  + \mscr{O}_{s}(\psi,\phi_{I}) -  \mscr{O}_{s}(\psi_h(\u),\phi_{I}) \\
&= \mscr{O}_{s}(\psi - \psi_h(\u),\phi - \phi_{I})  + \big[\mscr{L}_{\f,\u}(\phi_{I}) - \mscr{L}_{\f,\u}^{h}(\phi_{I})\big] \\
&\quad + \big[\mscr{O}_{s}^h(\psi_h(\u),\phi_{I}) - \mscr{O}_{s}(\psi_h(\u),\phi_{I})\big] := T_1 + T_2 +T_3.    
\end{split}
\end{equation}
\indent Firstly, we bound the term $T_1$ by using the Cauchy-Schwarz inequality, the $H^2$ estimate and \eqref{ineq:reg:add:aux}, as follows:
\begin{align*}
T_{1} 
&\lesssim \|\psi - \psi_h(\u)\|_{Z} \|\phi - \phi_{I}\|_{Z} \\
&\lesssim h^{\min\{s,k-1\} + \tilde{s}} \big[\|\psi\|_{2+s} + |\f+\u|_{\star,h} \big] |\psi - \psi_h(\u)|_{1}.
\end{align*}

 Next, we bound the second term $T_2$ using Cauchy–Schwarz inequality as follows:
\begin{align*}
T_{2} 
\lesssim h^{\min\{s,k-1\} + \tilde{s}} \big[\|\psi\|_{2+s} + |\f+\u|_{\star,h} \big] |\psi - \psi_h(\u)|_{1}.
\end{align*}

Now, consider the remaining term $T_3$:
\begin{equation}\label{termT3:H1}	
\begin{split}
    T_3 &= \sum_{K \in \mscr{T}_{h}} \big[\mscr{A}^{h,K}(\psi_h(\u),\phi_{I}) - \mscr{A}^{K}(\psi_h(\u),\phi_{I})\big] + \sum_{K \in \mscr{T}_{h}} \big[\mscr{B}^{h,K}(\psi_h(\u),\phi_{I}) - \mscr{B}^{K}(\psi_h(\u),\phi_{I})\big] \\
		 & \quad \ + \sum_{K \in \mscr{T}_{h}} \big[\mscr{C}^{h,K}(\psi_h(\u),\phi_{I}) - \mscr{C}^{K}(\psi_h(\u),\phi_{I})\big] \\
&: = T_{31} + T_{32} + T_{33}.
\end{split}
\end{equation}
\indent  We will analyze in detail the term $T_{32}$. Indeed, by using the definition of the bilinear forms $\mscr{B}^{h,K}(\cdot,\cdot)$ and $\mscr{B}^{K}(\cdot,\cdot)$,  have that
\begin{align*}
T_{32} 
&= \sum_{K \in \mscr{T}_h} \frac{1}{2}\big(\widetilde{\mscr{B}}^{h,K}(\psi_h(\u),\phi_I) - \widetilde{\mscr{B}}^{K}(\psi_h(\u),\phi_I)\big)\\
&\quad + \sum_{K \in \mscr{T}_h} \frac{1}{2}\big( \widetilde{\mscr{B}}^{h,K}(\phi_I,\psi_h(\u)) - \widetilde{\mscr{B}}^{K}(\phi_I,\psi_h(\u)).
\end{align*}
For the first term, we take $\varphi_h = \psi_h(\u)$ and $\phi_h = \phi_I$ in Lemma~\ref{lemma:B1}, and by using the same arguments as those used to obtain \eqref{bound:TB1}, we can derive
\begin{equation}\label{term:TB1:H1}
\begin{split}
\sum_{K \in \mscr{T}_h}& \big(\widetilde{\mscr{B}}^{h,K}(\psi_h(\u),\phi_I) - \widetilde{\mscr{B}}^{K}(\psi_h(\u),\phi_I)\big)
 \leq C_{\bbeta} h^{\min\{s,k-1\}+\tilde{s}}(\|\psi\|_{2+s} +|\f+\u|_{\star,h})\|\phi_I\|_{Z},
\end{split}
\end{equation}
where we have also applied the $H^2$-estimate.

Analogously, for the second term, we take $\varphi_h = \phi_I$ and $\phi_h= \psi_h(\u)$ in Lemma~\ref{lemma:B2}. Then,  by employing  the same arguments as those used to derive \eqref{bound:TB2}, and observing that
\begin{equation*}
\begin{split}
\big \|   \big(  \boldsymbol{\Pi}_{K,\text{mat}}^{k-2} -\textbf{I}\big)(\nabla \curl (\phi_I))  \big \|_{K}  &\leq   \big \| \boldsymbol{\Pi}_{K,\text{mat}}^{k-2} ( \nabla   \curl(\phi_I-\phi)) \big \|_{K}
 \\
&\quad  + \big \|  \big( \textbf{I}-  \boldsymbol{\Pi}_{K,\text{mat}}^{k-2}\big) ( \nabla   \curl(\phi)) \big \|_{K} + \big \| \nabla \curl(\phi -  \phi_I) \big \|_{K}\\
 & \lesssim \big \|  \big( \textbf{I}-  \boldsymbol{\Pi}_{K,\text{mat}}^{k-2}\big) ( \nabla   \curl(\phi)) \big \|_{K} +
\|\phi-\phi_I  \|_{2,K}\\
 & \lesssim h_{K}^{\tilde{s}} \|\phi\|_{2+\tilde{s},K}.
\end{split}
\end{equation*}
We can obtain the extra power for $h$  as follows
\begin{equation}\label{term:TB2:H1}
\begin{split}
\sum_{K \in \mscr{T}_h} &\big(\widetilde{\mscr{B}}^{h,K}(\phi_I,\psi_h(\u)) - \widetilde{\mscr{B}}^{K}(\phi_I,\psi_h(\u))\big) 
 \leq C_{\bbeta} h^{\min\{s,k-1\}+\tilde{s}}(\|\psi\|_{2+s,\Omega} + |\f +\u|_{\star,h}) \|\phi\|_{2+\tilde{s}}.
\end{split}
\end{equation}
By combining the bounds \eqref{termT3:H1}, \eqref{term:TB1:H1}, \eqref{term:TB2:H1} and using  standard arguments (cf.~\cite{MORABRINKMAN,BM}) along with \eqref{ineq:reg:add:aux}, we have
\begin{align*}
T_{3}  \leq  C_{\bbeta} h^{\min\{s,k-1\}+\tilde{s}}(\|\psi\|_{2+s,\Omega} + |\f +\u|_{\star,h})|\psi_h(\u)-\psi|_{1}.
\end{align*}

The desired result is obtained by inserting the above bound, and the estimates for $T_1$, $T_2$  in \eqref{5.4.2.3.04}. 
	\end{proof}
	\indent Now, we derive the error estimates for the optimal control and the state and co-state variables.
	\begin{theorem}\label{Theorem 5.4.2.2}
		{Let ${k\ge2}$ and the mesh regularity assumptions $\boldsymbol{\mathcal{A}1}$ and $\boldsymbol{\mathcal{A}2}$ hold}. Let $(\psi, \vartheta, \u) \in Z \times Z \times \mscr{U}_{\boldsymbol{ad}}$ be a solution of the optimality system \eqref{5.2.1.7}. {Let $(\psi_h, \vartheta_h, \uh) \in Z_h^{{k}} \times Z_h^{{k}} \times \mscr{U}_{\boldsymbol{ad}}^{\boldsymbol{h}}$ be the discrete solution of system \eqref{5.3.3.1} under a \textit{variational discretization} of control}. Then, the following bound holds:
		{
		\begin{equation*}
			\|\u-\uh\|_{0} \lesssim h^{{\min\{s,k-1\}+\tilde{s}}},
		\end{equation*}
		where $s\in(1/2, 2]$ as in Lemma~\ref{addReg} and Remark~\ref{remark:aditional:regu2}, and $\tilde{s}\in(1/2, 1]$
		satisfies \eqref{ineq:reg:add:aux}.}
	\end{theorem}
	\begin{proof}
		By substituting $\tilde{\u}=\uh$ in the continuous \textit{variational inequality} \eqref{5.2.1.7c} and $\tilde{\u}_{\boldsymbol{h}}=\u$ in the discrete \textit{variational inequality} \eqref{5.3.3.1c}, we obtain:
		\begin{align*}
			\big(\curl(\vartheta) + \lambda \u,\uh-\u\big) = - \big(\curl(\vartheta) + \lambda \u,\u - \uh\big) &\ge \boldsymbol{0},\\
			\sum_{K \in \mscr{T}_{h}} \big(\boldsymbol{\Pi}_{K}^{{{k-1}}} \curl(\vartheta_h) + \lambda \uh,\u-\uh\big)_{K} &\ge \boldsymbol{0}.
	\end{align*}
	Combining the above two inequalities, we write
		\begin{align}\label{neweqq}
			\sum_{K \in \mscr{T}_{h}} \big(\boldsymbol{\Pi}_{K}^{{{k-1}}} \curl(\vartheta_h) + \lambda \uh,\u-\uh\big)_{K} \ge \boldsymbol{0} \ge \big(\curl(\vartheta) + \lambda \u,\u-\uh\big).
		\end{align}
		
		{Now, using the coercivity condition \eqref{5.2.1.9} for $\u - \uh \in \mscr{U}_{\boldsymbol{ad}}$, applying the mean value theorem and using \eqref{neweqq}, we obtain}
		\begin{align*}
			\lambda \|\u-\uh\|_{0}^{2} & \le \mathcal{J}_r'(\u)(\u-\uh) - \mathcal{J}_r'(\uh)(\u-\uh)\\
			&= (\curl(\vartheta) + \lambda \u, \u - \uh) - (\curl(\vartheta(\psi(\uh))) + \lambda \uh, \u - \uh)\\
			&\le \sum_{K \in \mscr{T}_{h}} \big(\boldsymbol{\Pi}_{K}^{{{k-1}}} \curl(\vartheta_h) - \curl(\vartheta(\psi(\uh))),\u-\uh\big)_{K},
		\end{align*}
		where $\vartheta(\psi(\uh))$ is a solution of the co-state problem \eqref{5.2.1.6} when $\psi$ is replaced by $\psi(\uh)$ with $\psi(\uh)$ being a solution of \eqref{5.2.1.2} for $\u=\uh$. Adding and subtracting $\boldsymbol{\Pi}_{K}^{{{k-1}}} \curl(\vartheta(\psi(\uh)))$ in the first component, we obtain
		\begin{equation}
			\begin{split}
		&\lambda \|\u-\uh\|_{0}^{2} \le \sum_{K \in \mscr{T}_{h}} \big(\boldsymbol{\Pi}_{K}^{{{k-1}}}\big( \curl(\vartheta_h) - \curl(\vartheta(\psi(\uh)))\big),\u-\uh\big)_{K} \\ 
		 & \quad  + \sum_{K \in \mscr{T}_{h}} \big(\boldsymbol{\Pi}_{K}^{{{k-1}}} \curl(\vartheta(\psi(\uh))) - \curl(\vartheta(\psi(\uh))),\u-\uh\big)_{K}  \label{5.4.2.6}
			:= T_1 + T_2. 	
			\end{split}
		\end{equation}
		\indent The term $T_1$ is bounded using the Cauchy-Schwarz inequality, stability property of 
        $\boldsymbol{\Pi}_{K}^{{{k-1}}}$ with respect to the $L^2$-norm and an estimate similar to Lemma~\ref{Lemma 5.4.2.1.0}, as follows:
			\begin{align*}
				T_1 &\le \sum_{K \in \mscr{T}_{h}} \big\| \curl(\vartheta_h) - \curl(\vartheta(\psi(\uh)))\big\|_{0,K} \|\u-\uh\|_{0,K}
				\lesssim h^{{{\min\{s,k-1\}+\tilde{s}}}}\|\u-\uh\|_{0}.
			\end{align*}
			\indent Using Cauchy–Schwarz inequality and approximation properties of the projection operator $\boldsymbol{\Pi}_{K}^{{{k-1}}}$ in Lemma~\ref{lemma:approx_vector_matrix}, we bound the term $T_2$ as follows:
			\begin{align*}
				T_2 &\le \sum_{K \in \mscr{T}_{h}} \big\|\boldsymbol{\Pi}_{K}^{{{k-1}}} \curl(\vartheta(\psi(\uh))) - \curl(\vartheta(\psi(\uh)))\big\|_{0,K} \|\u-\uh\|_{0,K}\\
				&\le \sum_{K \in \mscr{T}_{h}} Ch_{K}^{{{\min\{s,k-1\}+\tilde{s}}}} \|\u-\uh\|_{0,K} \|\curl(\vartheta(\psi(\uh)))\|_{1+s,K}   \\
				&\lesssim h^{{{\min\{s,k-1\}+\tilde{s}}}}\|\u-\uh\|_{0}.
			\end{align*}
			We obtain the desired estimate by applying the bounds for $T_1$ and $T_2$ in \eqref{5.4.2.6}.
		\end{proof}
		\begin{theorem}\label{Theorem 5.4.2.3}
			Let $\u$ be the optimal control for the system \eqref{5.2.1.7} with the state and co-state stream-function $\psi$ and $\vartheta$, respectively. {Let $(\psi_h, \vartheta_h, \uh)$ be the discrete solution of system \eqref{5.3.3.1}. Then, the following estimates hold:
			\begin{align*}
    h^{\tilde{s}} \|\psi-\psi_h\|_{Z} + \|\psi-\psi_h\|_{1}  &\lesssim h^{\min\{s,k-1\}+\tilde{s}},\\
    h^{\tilde{s}} \|\vartheta-\vartheta_h\|_{Z} + \|\vartheta-\vartheta_h\|_{1} &\lesssim h^{\min\{s,k-1\}+\tilde{s}},
			\end{align*}
			where $s\in(1/2, 2]$ as in Lemma~\ref{addReg} and Remark~\ref{remark:aditional:regu2}, and $\tilde{s}\in(1/2, 1]$
		satisfies \eqref{ineq:reg:add:aux}.}
		\end{theorem}
		\begin{proof}
			Both the estimates for the state and co-state stream-function can be derived by an application of triangle inequality, Lemmas~\ref{Lemma 5.4.2.1}, \ref{Lemma 5.4.2.1.0} and Theorem~\ref{Theorem 5.4.2.2}.
		\end{proof}
		\subsection{Recovering the Velocity, Vorticity and Pressure Fields} \label{Recovering the Velocity, Vorticity and Pressure Fields}
		The solution of the VEM scheme proposed in Sec.~\ref{VEM Discrete Formulation} approximates both the stream-function and the control variable. One key advantage of employing the stream-function method for fluid flow problems is that it allows us to obtain other important quantities, such as the velocity field $\y$, the fluid pressure $p$, and the vorticity $\omega$. In this section, we will outline the techniques used to compute these three variables.
		\subsubsection{\textbf{Computation of Velocity and Vorticity}}\label{Computation of Velocity Field and Fluid Vorticity}
		Firstly, we compute the discrete velocity and vorticity which are obtained as a straightforward post-processing step from the computed stream-function, utilizing appropriate projection techniques.
		
		Since the stream-functions $\psi, \vartheta \in Z$ are unique solutions of the system \eqref{5.2.1.7}, we have
		\begin{align}\label{5.4.2.1.1}
			\y = \curl(\psi), \qquad \qquad \w = \curl(\vartheta),
		\end{align}
		where $\w$ denotes the co-state velocity field. At the discrete level, we compute the discrete velocities $\yh$ and $\wh$ as a post-processing of the computed stream-functions $\psi_h, \vartheta_h \in Z_{h}^{{{k}}}$ as follows: if $\psi_h, \vartheta_h$ are the unique solutions of system \eqref{5.3.3.1}, then 
		\begin{align}\label{5.4.2.1.2}
			\yh = \boldsymbol{\Pi}^{{{k-1}}} \curl(\psi_h), \qquad \qquad \qquad \wh = \boldsymbol{\Pi}^{{{k-1}}} \curl(\vartheta_h),
		\end{align}
		are the velocity approximations, where $\boldsymbol{\Pi}^{{{k-1}}}$ is defined in $\L^{2}(\Omega)$ as:
		\begin{align*}
			\big(\boldsymbol{\Pi}^{{{k-1}}} \v\big)|_{K} = \boldsymbol{\Pi}_{K}^{{{k-1}}}(\v|_{K}) \qquad \forall K \in \mscr{T}_{h}.
		\end{align*}
		\begin{theorem}\label{Theorem 5.4.2.1.1}
			Suppose that the hypothesis of Lemma~\ref{Lemma 5.4.2.1.0} holds true. Then, for $s \in (1/2,2]$ we have
				\begin{align*}
                    |\y-\yh|_{1,h} &\ \lesssim h^{\min\{s,k-1\}}\big[\|\psi\|_{2+s} + {|\f+\u|_{\star,h}} \big],\\
                    |\w-\wh|_{1,h} & \lesssim h^{\min\{s,k-1\}}\big[ \|\vartheta\|_{2+s} + |\mathrm{\bf curl}(\psi)- {\mathrm{\bf curl}(\psi_d)}|_{\star,h} 
					\big].
				\end{align*}
		\end{theorem}
		\begin{proof}
			The first estimate is obtained using Theorem~8.4 in Ref.~\refcite{MORAOSEEN}, through an application of the triangle inequality, estimate \eqref{3.5.1.5}, and Theorem~\ref{Theorem 5.4.2.3}. The second estimate can be derived in a comparable manner by following analogous steps.
		\end{proof}
			\begin{theorem}\label{Theorem 5.4.2.1.1.0}
				Suppose that the hypothesis of Lemma~\ref{Lemma 5.4.2.1.0} holds true. Then, for ${k\ge2}$, $\tilde{s} \in (1/2,1]$ and $s \in (1/2,2]$, we have the following bounds:
				\begin{align*}
					\|\y-\yh\|_{0} &\lesssim h^{\min\{s,k-1\}+\tilde{s}},\qquad \qquad 
					\|\w-\wh\|_{0}  \lesssim C_{2}h^{\min\{s,k-1\}+\tilde{s}}.
				\end{align*}
		\end{theorem}
		\begin{proof}
				Using \eqref{5.4.2.1.1}, \eqref{5.4.2.1.2}, the triangle inequality, Lemma~\ref{lemma:approx_vector_matrix}, and the $H^1$-norm estimate derived in Theorem~\ref{Theorem 5.4.2.3}, we obtain
				\begin{align*}
					&\|\y-\yh\|_{0} 
					\le \sum_{K \in \mscr{T}_h} \|\curl(\psi) - \boldsymbol{\Pi}_{K}^{{{k-1}}} \curl(\psi_h)\|_{0,K}\\
					&\le \sum_{K \in \mscr{T}_h} \Big(\|\curl(\psi) - \boldsymbol{\Pi}_{K}^{{{k-1}}} \curl(\psi)\|_{0,K} + \|\boldsymbol{\Pi}_{K}^{{{k-1}}} (\curl(\psi-\psi_h))\|_{0,K}\Big)\\
					&\lesssim \sum_{K \in \mscr{T}_h} \Big(h_{K}^{\min\{s,k-1\}+\tilde{s}} + \|\psi-\psi_h\|_{1,K}\Big)\\
					&\lesssim h^{\min\{s,k-1\}+\tilde{s}}.
				\end{align*}
			The second estimate for the co-state velocity follows by a similar approach.
		\end{proof}
		Now, we will present a strategy to recover the state and co-state fluid vorticity $\omega$ and $\theta$, which is key in several important applications in fluid mechanics (cf. Refs.~\refcite{VANAYA}, \refcite{MORAB}, \refcite{AERRN}, \refcite{HSAK2}). A use of the definition of vorticity and the relation \eqref{5.4.2.1.1} gives
		\begin{align*}
            \omega = \curl(\y) = - \Delta \psi, \qquad \qquad \theta = \curl(\w) = - \Delta \vartheta.
		\end{align*}
		\indent Next, we introduce an orthogonal $L^2$-projection $\Pi_{K}^{{{k-2}}}$ onto $\mathbb{P}_{{{k-2}}}(K)$, such that the following approximation result from Refs.~\refcite{VEMGSOC} and \refcite{BSC} is satisfied:
		\begin{align}
			\label{5.4.2.1.5b} \|v-\Pi_{K}^{{{k-2}}}v\|_{m,K} &\le \hat{C} h_{K}^{\delta-m}|v|_{\delta,K} \qquad &&\forall v\in H^{\delta}(K), \quad 0\le m \le \delta \le k-1,
		\end{align}
		where $\hat{C}$ is a positive constant independent of the mesh parameter $h_K$.
		
		\indent Now, we compute the discrete state and co-state vorticity $ \mscr{T}_{h}$ and $\theta_h$ as follows: if $(\psi_h, \vartheta_h, \uh)$ is a unique solution of system \eqref{5.3.3.1}, then
		\begin{align}\label{5.4.2.1.6}
			\omega_h := - \Pi^{{{k-2}}}(\Delta \psi_h), \qquad \qquad \theta_h := - \Pi^{{{k-2}}}(\Delta \vartheta_h),
		\end{align}
		is the approximation of the state and co-state vorticity, where
		\begin{align*}
			\big(\Pi^{{{k-2}}}v\big)\big|_{K} = \Pi_{K}^{{{k-2}}}(v|_{K}) \qquad \forall v \in L^{2}(\Omega) \quad \text{and} \quad \forall K \in \mscr{T}_{h}.
		\end{align*}
		\begin{remark}
			For all $\phi_h$ in the local virtual space $Z_{h}^{{{k}}}(K)$, by using the degrees of freedom $\boldsymbol{D_1}$-$\boldsymbol{D_5}$, the polynomial function $\Pi_{K}^{{k-2}}\big(\Delta \phi_h\big) \in \mathbb{P}_{{k-2}}(K)$ can be computed as elaborated in Remark~8.4\cite{MORAOSEEN}.
		\end{remark}
		\noindent Now, we present the convergence result for the state and co-state discrete vorticity.
		\begin{theorem}\label{Theorem 5.4.2.1.2}
			Suppose that the hypothesis of Lemma~\ref{Lemma 5.4.2.1.0} holds true. Then, for $s \in (1/2,2]$, we have
				\begin{align*}
            \|\omega-\omega_h\|_{0} &\lesssim h^{\min\{s,k-1\}}\big[\|\psi\|_{2+s} + {|\f+\u|_{\star,h}} \big],\\
    \|\theta-\theta_h\|_{0} &\lesssim h^{\min\{s,k-1\}}\big[ \|\vartheta\|_{2+s} + |\mathrm{\bf curl}(\psi)- {\mathrm{\bf curl}(\psi_d)}|_{\star,h} 
					\big].
				\end{align*}
		\end{theorem}
		\begin{proof}
			The first estimate follows from the analysis presented in Theorem~8.5,\cite{MORAOSEEN} using \eqref{5.4.2.1.5b},  \eqref{5.4.2.1.6} and Theorem~\ref{5.4.2.3}, while the second estimate for the co-state vorticity error follows by similar steps.
		\end{proof}
		\subsubsection{\textbf{Recovering the Fluid Pressure}}
		In this subsection, we present a scheme similar to Section~8.5.3 of Ref.~\refcite{MORAOSEEN} to recover the state fluid pressure $p$ (same approach will follow for the co-state pressure $r$). From the momentum equation in \eqref{5.2.0.2}, we obtain
		\begin{align}
		\nonumber	\nabla p &= \f + \u - \nu \, \curl\big(\curl(\y)\big) - (\boldsymbol{\beta} \cdot \nabla) \y - \sigma \y  \\
		\nonumber	&= \f + \u - \nu \, \curl\big(\curl \big(\textbf{\textit{curl}}(\psi)\big)\big) - (\nabla \curl(\psi)) \boldsymbol{\beta} - \sigma \, \curl(\psi) \\
		\label{5.4.2.2.01}	&= \f + \u + \nu \, \curl \big(\Delta \psi\big) - (\nabla \curl (\psi)) \boldsymbol{\beta} - \sigma \, \curl(\psi).
		\end{align}
		\indent Now, to calculate the fluid pressure, we test \eqref{5.4.2.2.01} with $\nabla q$, where $q \in \mscr{H}$ (as defined in \eqref{Hspace}) and obtain the following variational problem : find $p \in \mscr{H}$ such that
		\begin{align*}
			\mscr{R}_{\nabla}(p,z) &= \mscr{G}^{\psi}(z)  && \forall z \in \mscr{H},
		\end{align*}
		where for all $p,z \in \mscr{H}$, $\mscr{R}_{\nabla}: \mscr{H} \times \mscr{H} \rightarrow \R$ and $\mscr{G}^{\psi}: \mscr{H} \rightarrow \R$ are defined as:
		\begin{align}
			\mscr{R}_{\nabla} (p,z) &:= (\nabla p, \nabla z), \qquad
			\label{5.4.2.2.3} 
			\mscr{G}^{\psi}(z) := \left(\f + \u + \nu \ \curl(\Delta \psi) - (\nabla \curl(\psi)) \boldsymbol{\beta} - \sigma \, \curl(\psi),\nabla z\right).
		\end{align}
		\indent From this point forward, we assume that $\Omega$ is a convex domain. This gives us extra regularity for the unique solution $\psi$ of the state problem \eqref{5.2.1.2}. Specifically, we have $\psi \in H^3(\Omega)$, and there is a positive constant $C$ such that $\|\psi\|_3 \le C \|\f + \u\|_0$.\\
		
		 We now build a lowest-order scheme to estimate the pressure on the mesh $\mscr{T}_h$. Firstly, we decompose the bilinear form $\mscr{R}_{\nabla} (\cdot, \cdot)$ into elemental contributions as:
		\begin{align*}
			\mscr{R}_{\nabla} (p,z) &:= \sum_{K \in \mscr{T}_{h}} \mscr{R}_{\nabla} ^{K}(p,z) = \sum_{K \in \mscr{T}_{h}} (\nabla p, \nabla z)_{K} && \forall p,z \in \mscr{H}.
		\end{align*}
		\indent For all $K \in \mscr{T}_{h}$, consider the finite-dimensional space $\widetilde{\mscr{P}}_{h}(K)$, which is defined as:
		\begin{align*}
			\widetilde{\mscr{P}}_{h}(K) := \big\{z_h \in H^{1}(K) \cap C^{0}(\partial K): z_h|_{E} \in \mathbb{P}_{1}(E) \quad \forall E \subset \partial K, \quad \Delta z_h \in \mathbb{P}_{0}(K) \big\}.
		\end{align*}
		\indent The set $\boldsymbol{P_1}$, which contains linear operators which compute values of $z_h$ at the vertices of $K$ is defined for all $z_h \in \widetilde{\mscr{P}}_{h}(K)$. Next, for all $q_h \in \widetilde{\mscr{P}}_{h}(K)$, we define the projector $\Pi_{K}^{\nabla}:\widetilde{\mscr{P}}_{h}(K) \rightarrow \mathbb{P}_{1}(K) \subseteq \widetilde{\mscr{P}}_{h}(K)$ as the solution of
		\begin{align*}
			\mscr{R}_{\nabla} ^{K}\left(\Pi_{K}^{\nabla}z_h, p_1\right) &= \mscr{R}_{\nabla} ^{K}(z_h,p_1) \quad \forall p_1 \in \mathbb{P}_{1}(K), && \hat{\Pi_{K}^{\nabla}z_h} = \hat{z_h},
		\end{align*}
		where $\hat{(\cdot)}$ is defined as in \eqref{5.3.1.2.0}. The operator $\Pi_{K}^{\nabla}$ can be directly computed using the set $\boldsymbol{P_1}$\cite{VEMGSOC}. By combining this projection with the definition of $\widetilde{\mscr{P}}_{h}(K)$, we define the following local virtual space:
		\begin{align*}
			\mscr{P}_{h}(K) := \big\{z_h \in \widetilde{\mscr{P}}_{h}(K): (z_h - \Pi_{K}^{\nabla}z_h,1)_{K} = 0 \big\}.
		\end{align*}
		\indent Clearly, $\mathbb{P}_{1}(K) \subseteq \mscr{P}_{h}(K) \subseteq \widetilde{\mscr{P}}_{h}(K)$. Additionally, the set $\boldsymbol{P_1}$ serves as the degrees of freedom for $\mscr{P}_{h}(K)$, and the operator $\Pi_{K}^{\nabla}$ can be computed using the set $\boldsymbol{P_1}$ \cite{VEMGSOC}. We will now form the global virtual space to approximate the discrete pressure as:
		\begin{align*}
			\mscr{P}_{h} := \big\{q_h \in \mscr{H}: q_h|_{K} \in \mscr{P}_{h}(K) \quad \forall K \in \mscr{T}_h\big\},
		\end{align*}
		where $\mscr{H}$ is as defined in \eqref{Hspace}. We have that for all $z_h \in \mscr{P}_{h}(K)$, $\boldsymbol{\Pi}_{K}^{0} \nabla z_h$ is fully computable using the degrees of freedom $\boldsymbol{P_1}$, where  $\boldsymbol{\Pi}_{K}^{0}:\L^{2}(K) \rightarrow [\mathbb{P}_{0}(K)]^{2}$ is the classical $L^{2}$-projection.
        
       Next, we consider $\mscr{S}_{\nabla}^{K}(\cdot,\cdot)$ as any symmetric positive definite bilinear form which satisfies the following:
		\begin{align}\label{5.4.2.2.5} 
			\alpha_{4}\mscr{R}_{\nabla}^{K}(z_h,z_h)\le \mscr{S}_{\nabla}^{K}(z_h,z_h) \le \alpha_{5}\mscr{R}_{\nabla}^{K}(z_h,z_h),
		\end{align}
		for all $z_h \in \mscr{P}_{h}(K), \ \text{with} \ \boldsymbol{\Pi}_{K}^{0} \nabla q_h = 0$ and positive constants $\alpha_{4}$ and $\alpha_{5}$ independent of $K$.
		
		 Now, for all $p_h,z_h \in \mscr{P}_{h}(K)$, we define
			\begin{align*}
        \mscr{R}_{\nabla}^{h}(p_h,z_h)&:= \sum_{K \in \mscr{T}_{h}} \mscr{R}_{\nabla}^{h,K}(p_h,z_h);\\
        \mscr{R}_{\nabla}^{h,K}(p_h,z_h) &:= \big(\boldsymbol{\Pi}_{K}^{0} \nabla p_h, \boldsymbol{\Pi}_{K}^{0} \nabla z_h\big)_{K} + \mscr{S}_{\nabla}^{K}\big(p_h - \Pi_{K}^{\nabla}p_h,z_h - \Pi_{K}^{\nabla}z_h\big).
			\end{align*}
			
			\blue{By following Refs.~\refcite{VEM1} and \refcite{STEKLOV}, the bilinear form $\mscr{S}_{\nabla}^{K}(\cdot,\cdot)$ satisfying \eqref{5.4.2.2.5} is chosen as:
				$$
				\mscr{S}_{\nabla}^{K}(p_h,z_h) := \sum_{j=1}^{N_K} p_h(\v_{i}) z_h(\v_{i}) \qquad \forall p_h,z_h \in \mscr{P}_{h}(K).
				$$}
		\indent The local bilinear forms $\mscr{R}_{\nabla}^{K}(\cdot,\cdot)$ and $\mscr{R}_{\nabla}^{h,K}(\cdot,\cdot)$ satisfy the classical consistency and stability properties (cf. Ref.~\refcite{VEMGSOC}).
		Next, we will create an approximation for the right-hand side of equation \eqref{5.4.2.2.3}.
		We will assume that the discrete problem \eqref{5.3.3.1} has been solved for \(k=3\). Consequently, we now have \(\psi_h \in Z_h^{3}\), which meets the error estimate stated in Theorem~\ref{Theorem 5.4.2.3}.
		For every element \( K \in \mscr{T}_{h} \) and for every function \( z_h \) in the space \( \mscr{P}_{h}(K) \), we define the discrete linear functional as follows:
		\begin{align*}
			\mscr{G}^{\psi_h,K}(z_h) &:= \left(\f+\uh, \boldsymbol{\Pi}_{K}^{0}\nabla z_h\right)_{K} + \nu \big(\curl \big(\Pi_{K}^{1}(\Delta \psi_h)\big), \boldsymbol{\Pi}_{K}^{0}\nabla z_h\big)_{K} \\
			& \ \ \ \ - \big((\nabla \boldsymbol{\Pi}_{K}^{2} \curl(\psi_h)) \bbeta, \boldsymbol{\Pi}_{K}^{0}\nabla z_h\big)_{K} - \big(\sigma \boldsymbol{\Pi}_{K}^{2} \curl(\psi_h), \boldsymbol{\Pi}_{K}^{0}\nabla z_h\big)_{K},
		\end{align*}
		where $\boldsymbol{\Pi}_{K}^{2}$ and $\Pi_{K}^{1}$ are the projections introduced in Sec.~\ref{Virtual Spaces and Polynomial Projections Operator} and Sec.~\ref{Computation of Velocity Field and Fluid Vorticity}, respectively. By a use of the set of degrees of freedom $\boldsymbol{P_1}$ , $\mscr{G}^{\psi_h,K}(\cdot)$ is fully computable. We define the global linear functional $\mscr{G}^{\psi_h}(\cdot)$ as:
		\begin{align}\label{5.4.2.2.9} 
			\mscr{G}^{\psi_h}(z_h)&:= \sum_{K \in \mscr{T}_{h}} \mscr{G}^{\psi_h,K}(z_h) && \forall z_h \in \mscr{P}_{h}.
		\end{align}
		We propose a lowest-order virtual element to estimate the fluid pressure. Given $\psi_h \in Z_{h}^{3}$, we want to determine $p_h \in \mscr{P}_{h}$, such that
		\begin{align}\label{5.4.2.2.10} 
			\mscr{R}_{\nabla}^{h}(p_h,z_h) &= \mscr{G}^{\psi_h}(z_h) && \forall z_h \in \mscr{P}_{h}.
		\end{align} 
		By an application of Proposition~8.6\cite{MORAOSEEN},
		$\mscr{R}_{\nabla}^{h}(\cdot,\cdot)$ satisfies the bound
		\begin{align*}
			\mscr{R}_{\nabla}^{h}(z_h,z_h) &\ge \alpha_{p} \|z_h\|_{1}^{2} && \forall z_h \in \mscr{P}_{h},
		\end{align*}
		where $\alpha_{p} >0$ is a constant independent of $h$. Our next aim is to show that the discrete linear functional $\mscr{G}^{\psi_h}(\cdot)$ defined in \eqref{5.4.2.2.9} is bounded. Firstly, we consider the following interpolation result established in Refs.~\refcite{CANG} and \refcite{STEKLOV}. 
		\begin{lemma}\label{Lemma 5.4.2.2.3}
			There exists a constant \(C>0\) independent of \( h \) such that for all \( v \in H^{2}(\Omega) \), if assumptions $\boldsymbol{\mathcal{A}1}$ and $\boldsymbol{\mathcal{A}2}$ hold, then there exists \( v_{I} \in \mscr{P}_h \) which satisfies
			\begin{align*}
				\|v-v_{I}\|_{0} + h|v-v_{I}|_{1} \le Ch^{2}|v|_{2}.
			\end{align*}
		\end{lemma}
\indent  To prove the boundedness of $\mscr{G}^{\psi_h}(\cdot)$ and derive an error estimate for the pressure scheme (cf. Lemma~\ref{Lemma 5.4.2.2.4} and Theorem~\ref{Theorem 5.4.2.2.5}), we assume that the polygonal mesh family $\mscr{T}_h$ is quasi-uniform. More precisely, the following additional mesh assumption is required.\\
			$\boldsymbol{\mathcal{A}3}:$ For all $h>0$ and for all $K \in \mscr{T}_h$, there exists a constant $\hat{c}>0$ independent of $h$, such that $h_{K} \ge \hat{c}h$. 
		\begin{lemma}\label{Lemma 5.4.2.2.4}
			Let \(\psi\) be the unique solution of the state problem \eqref{5.2.1.2} in \(H^{3}(\Omega)\), and \(\psi_h\) be the unique solution of the discrete system in \(Z_{h}^{3}\). If conditions $\boldsymbol{\mathcal{A}1}$-$\boldsymbol{\mathcal{A}3}$ hold, then the discrete linear functional \(\mscr{G}^{\psi_h}(\cdot)\) is well-defined and bounded.
		\end{lemma}
		\begin{proof}
			The proof follows from the proof of Proposition~4.20\cite{MORABRINKMAN}.
		\end{proof}
		By using the ellipticity of the bilinear form \(\mscr{R}_{\nabla}^{h}(\cdot,\cdot) \) and Lemma~\ref{Lemma 5.4.2.2.4}, the discrete virtual element scheme \eqref{5.4.2.2.10} has a unique solution \( p_h \in \mscr{P}_h \). Additionally, there exists a constant \( C > 0 \) that does not depend on \( h \) such that
		\begin{align*}
			\|p_h\|_{1} \le C\big(\|\f\|_{0} + \|\u\|_{0}+|\f|_{1,h}+|\u|_{1,h}\big).
		\end{align*}
		\indent Next, we will derive the convergence estimate for the discrete scheme \eqref{5.4.2.2.10}, considering the conditions $\boldsymbol{\mathcal{A}1}$-$\boldsymbol{\mathcal{A}3}$ along with the assumptions that \( p \in H^2(\Omega) \) and \( \psi \in H^4(\Omega) \). This additional regularity for the stream-function can be attained, for instance, if $\f \in \L^{2}(\Omega)$ and $\Omega$ is a rectangular domain (cf. Ref.~\refcite{C0IP}).
		\begin{theorem}\label{Theorem 5.4.2.2.5}
			Let {$\f, \u \in \boldsymbol{H}^{1}(\Omega,\mscr{T}_{h})$} and $(\psi,p)$ and $(\psi_h, p_h)$ be the stream-function and pressure solutions for continuous and discrete state problem, respectively. If the assumptions $\boldsymbol{\mathcal{A}1}$-$\boldsymbol{\mathcal{A}3}$ are satisfied, $p \in H^2(\Omega)$ and $\psi \in H^4(\Omega)$.
			Then the following error estimate hold:
			\begin{align*}
				\|p-p_h\|_{1} &\lesssim h\big(\|\psi\|_{4} + |\f|_{1,h} + |\u|_{1,h}\big).
			\end{align*}
		\end{theorem}
		\begin{proof}
			The proof of this estimate is obtained by a combination of the Lemmas~\ref{Lemma 5.4.2.2.2}, \ref{Lemma 5.4.2.2.3} and Proposition~8.9, 8.10\cite{MORAOSEEN}.
		\end{proof}

		\section{Numerical Experiments}\label{Numerical Experiments.}\setcounter{equation}{0}
		In this section, we provide a set of numerical tests to demonstrate the practical effectiveness of the proposed VEM schemes and verify the theoretical estimates. We will apply the proposed scheme on four different polygonal meshes shown in Fig.~\ref{MESH} using Algorithm~1.
		\begin{figure}
			\centering
			\subfloat[$\mscr{T}_{h}^{1}$ \label{mesh1}]{\includegraphics[width=0.45\textwidth]{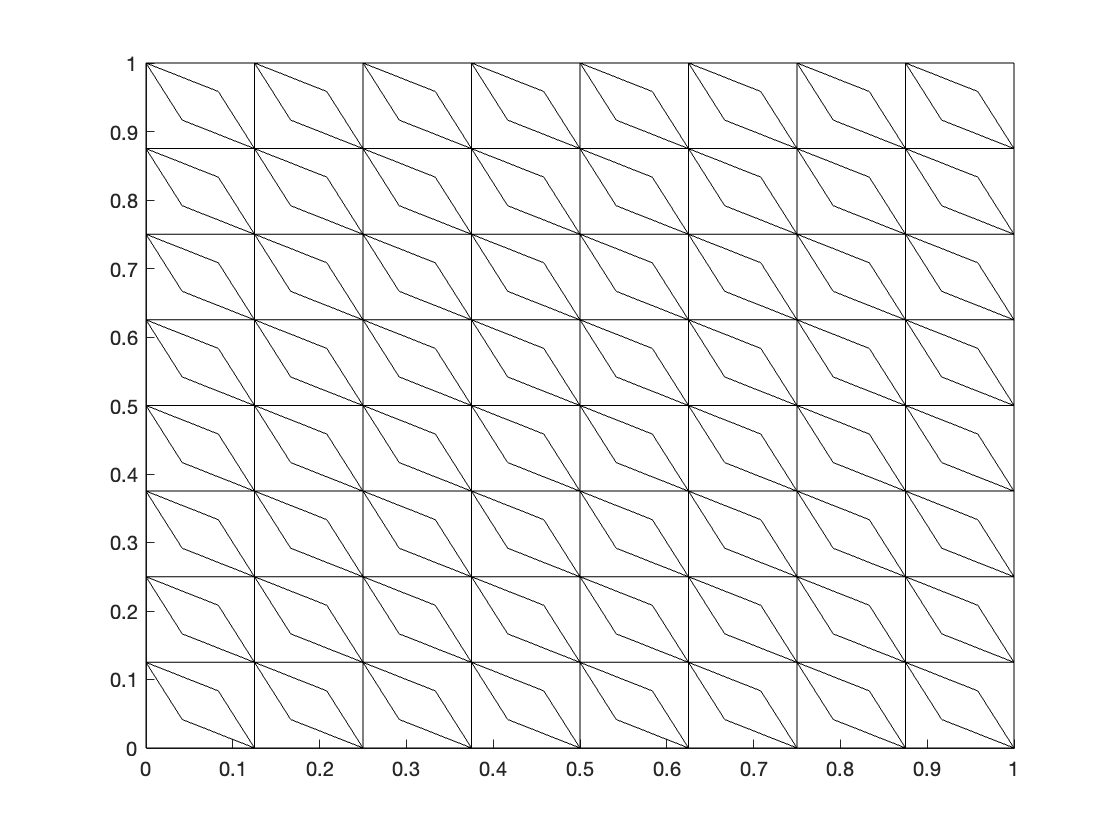}} \hfill
			\subfloat[$\mscr{T}_{h}^{2}$ \label{mesh2}]{\includegraphics[width=0.45\textwidth]{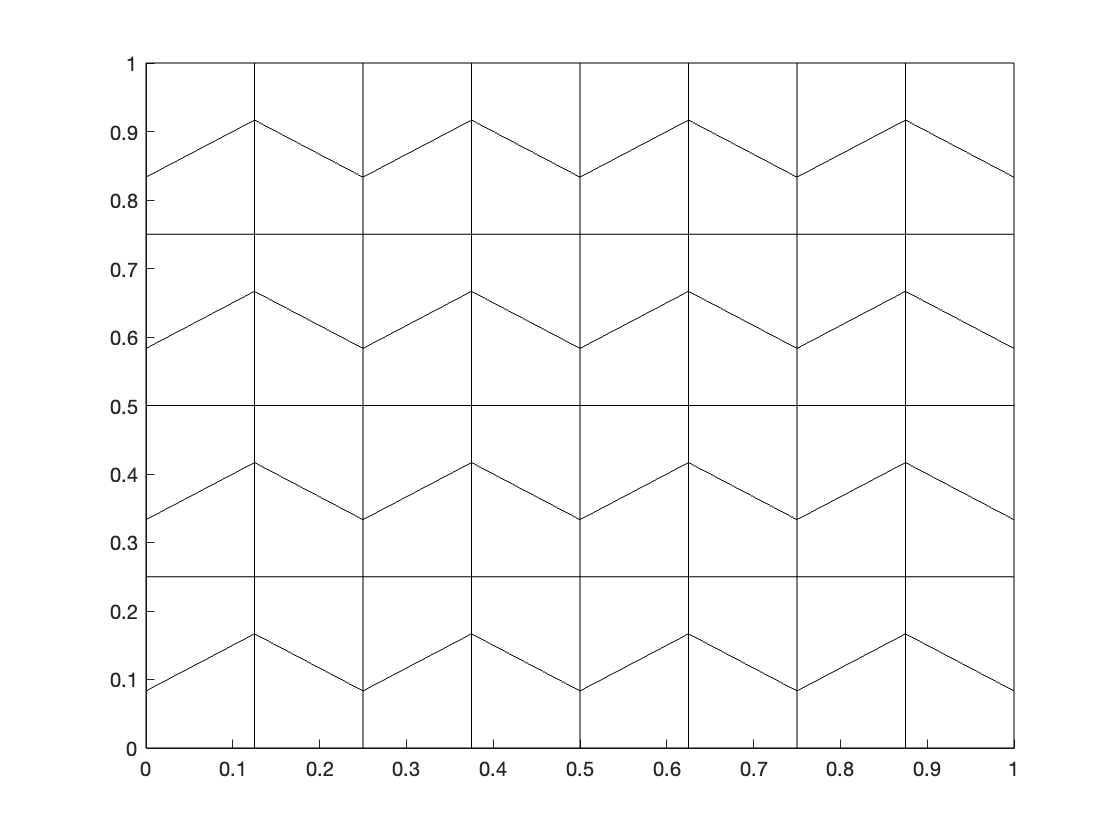}}
			\\
			\subfloat[$\mscr{T}_{h}^{3}$ \label{mesh3}]{\includegraphics[width=0.45\textwidth]{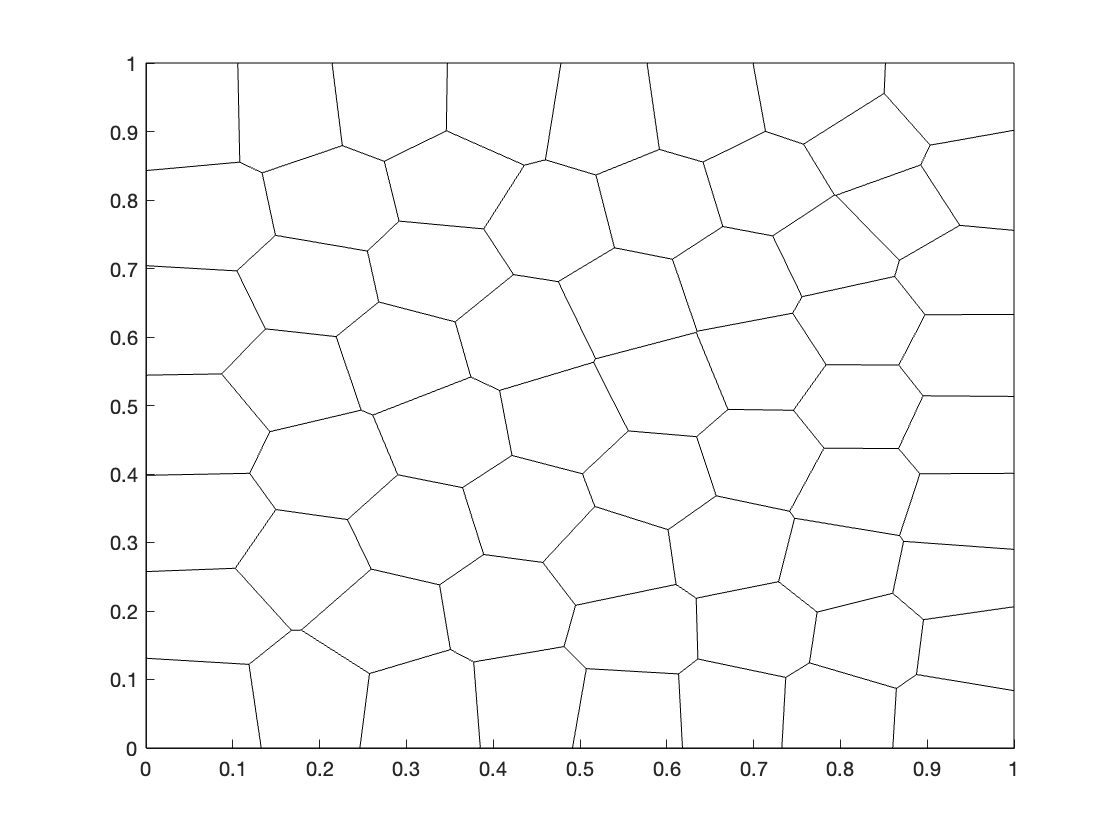}} \hfill
			\subfloat[$\mscr{T}_{h}^{4}$ \label{mesh4}]{\includegraphics[width=0.45\textwidth]{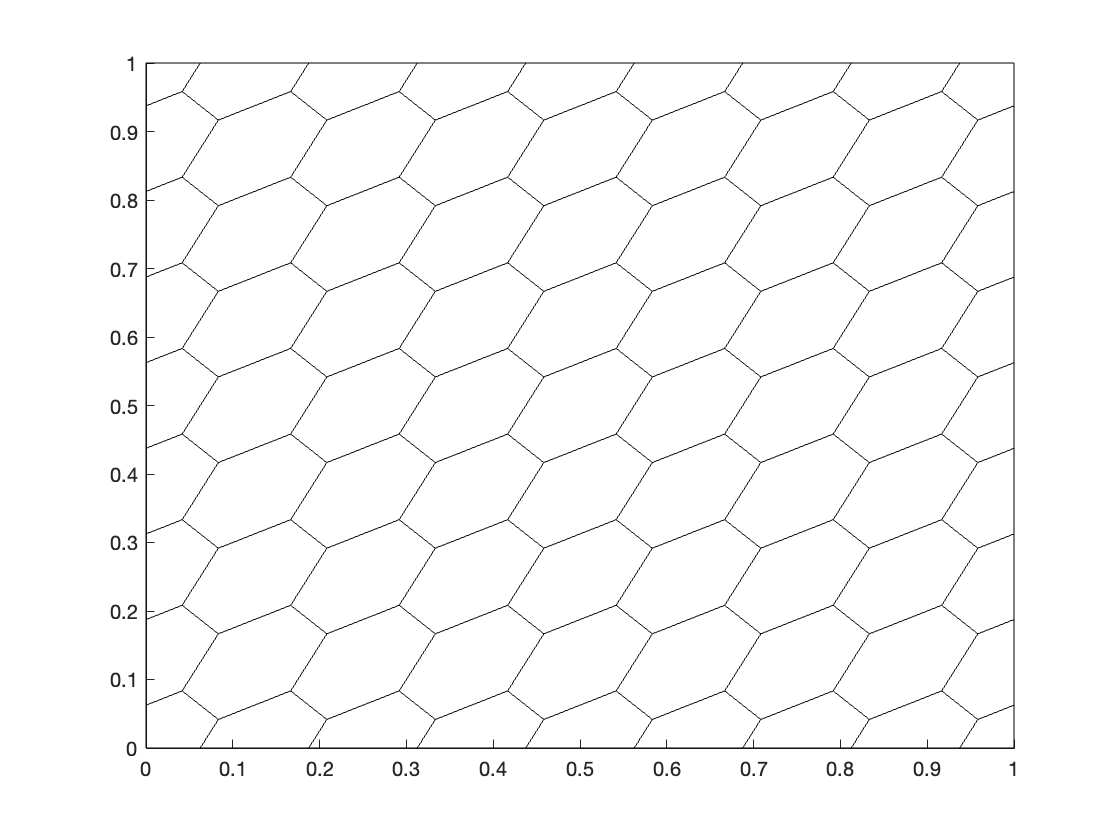}} 
			\caption{\eqref{mesh1} $\mscr{T}_{h}^{1}$:Distorted Concave Rhombic Quadrilaterals;
				\eqref{mesh2} $\mscr{T}_{h}^{2}$: Trapezoidal meshes;
				\eqref{mesh3} $\mscr{T}_{h}^{3}$: Sequence of Centroidal Voronoi Tessellation; and \eqref{mesh4} $\mscr{T}_{h}^{4}$: Hexagonal meshes.} \label{MESH}
		\end{figure}
		\begin{algorithm}
			\caption{Projected Gradient Algorithm.}
			\textbf{Requirement:} Regularization parameter $\lambda$ and tolerance error $\eta$.
			\begin{algorithmic}[1]\label{Algorithm1}
				\State Choose an arbitrary initial guess $\uh$ and set $error= 1$.
				\While{$error > \eta$,}
				\State Solve the discrete system (\ref{5.3.3.1a}-\ref{5.3.3.1b}) to obtain $\psi_h$ and $\vartheta_h$.
				\State Using the discrete projection formula \eqref{dpf}, find the new control as:
				\begin{align*}
					\boldsymbol{u_{h,n}}|_{K} = \mscr{P}_{[\boldsymbol{u_a},\boldsymbol{u_b}]}\left(-\frac{1}{\lambda}\boldsymbol{\Pi}_{K}^{{{k-1}}} \big(\curl(\vartheta_h)|_{K}\big) \right) \qquad \forall K \in \mscr{T}_{h}.
				\end{align*}
				\State Calculate the error between $\uh$ and $\boldsymbol{u_{h,n}}$ as: $error = \|\uh-\boldsymbol{u_{h,n}}\|_{0}$.
				\State Update the control variable: $\uh = \boldsymbol{u_{h,n}}$.
				\EndWhile
			\end{algorithmic}
		\end{algorithm}
		For computing the VEM errors and convergence rates, we consider the following computable error quantities and the convergence rate formula:\\
		\begin{align*}
			e_{i}(\psi) &:= \left(\sum_{K \in \mscr{T}_{h}}|\psi-\Pi_{K}^{{{k}},D}\psi_h|_{i,K}^{2}\right)^{\frac{1}{2}}, &&i=1,2; \\
			 e_{0}(\omega) &:= \left(\sum_{K \in \mscr{T}_{h}}\|\omega-\Pi_{K}^{{{k-2}}} \Delta \psi_h\|_{0,K}^{2}\right)^{\frac{1}{2}}, && e_{0}(\u) := \left(\sum_{K \in \mscr{T}_{h}}\|\u-\uh\|_{0,K}^{2}\right)^{\frac{1}{2}}, \\\\
			e_{j}(p) &:= \left(\sum_{K \in \mscr{T}_{h}}|p-\Pi_{K}^{\nabla}p_h |_{j,K}^{2}\right)^{\frac{1}{2}}, && j =0,1; \\ 
			r_{l,*} &:= \frac{log(e_{l}(*)/e_{l}'(*))}{log(h/h')}, && l = 0,1,2,
		\end{align*}
		where $h$ and $h'$ are consecutive mesh sizes with corresponding errors $e_{l}$ and $e_{l}'$, respectively. The errors for the co-state variables are computed using the same formulas and the stream-function $\vartheta_h$.
	\begin{example}\label{Example 5.2.}
	\textbf{\textit{Viscosity variation:}}
	In this example, for the unit square domain $\Omega = (0,1)^{2}$, we consider a variable viscosity test. More precisely, we take
	two specifications for the viscosity: 
	\begin{itemize}
		\item \textbf{Case 1:} $\nu(\x) = \nu_{0} + (\nu_1 - \nu_{0})x_1 x_2$,
		\item \textbf{Case 2:} $\nu(\x) = \nu_{0}^{3}$,
	\end{itemize}
	 where $\nu_{0} = 0.001$ and $\nu_{1} = 1$. The parameters $\bbeta = (1,1)^{T}$, $\sigma = 1000$ and control bounds $\boldsymbol{u_a} = -(0.1,0.1)^{T}$ and $\boldsymbol{u_b} = (0.25,0.25)^{T}$. We choose $\f$ and $\boldsymbol{y_d}$ such that the analytical solutions are:
	\begin{align*}
		\psi(\x) &= \vartheta(\x) = \sin^{2}(\pi x_1) \sin^{2}(\pi x_2)/(2\pi), &&
		\y(\x) = \w(\x) = \begin{pmatrix}
			\sin^{2}(\pi x_1) \sin(\pi x_2) \cos(\pi x_2) \\ - \sin^{2}(\pi x_2) \sin(\pi x_1) \cos(\pi x_1)
		\end{pmatrix},\\
		 \omega(\x) &= \theta(\x) = - \Delta \psi,&&
		p(\x) =  r(\x) = \sin(2\pi x_1) \sin(2\pi x_2).
	\end{align*}
	\end{example}
	
We have slightly modified the numerical scheme to incorporate the
variable viscosity and the results are presented with $k=3$ on hexagonal meshes $\mscr{T}_{h}^{4}$ for both cases and achieve optimal convergence rates for all state and co-state variables, as shown in Fig.~\ref{state211} for \textbf{Case 1} and Fig.~\ref{state2111} for \textbf{Case 2}. Fig.~\ref{FIGURE 8} visually depicts the stream-function $\psi_h$, velocity field $\boldsymbol{y_h}$, vorticity $\omega_h$, pressure field $p_h$, and the control components $\boldsymbol{u_{h1}}$ and $\boldsymbol{u_{h2}}$, providing a clear representation of the solution's behavior for \textbf{Case 1}. These results highlight the flexibility of the proposed approach in handling more general scenarios with spatially varying viscosity, confirming its suitability for addressing realistic and complex physical problems. In particular, although a complete theoretical foundation for the variable viscosity case is not developed in this work, the numerical results still exhibit the expected convergence rates for all variables. A rigorous extension of the theory to variable coefficient problems is an interesting and relevant direction for future research.
		\begin{figure}
				\centering
				\subfloat[\textbf{Case 1} \label{state211}]{\includegraphics[width=0.49\textwidth]{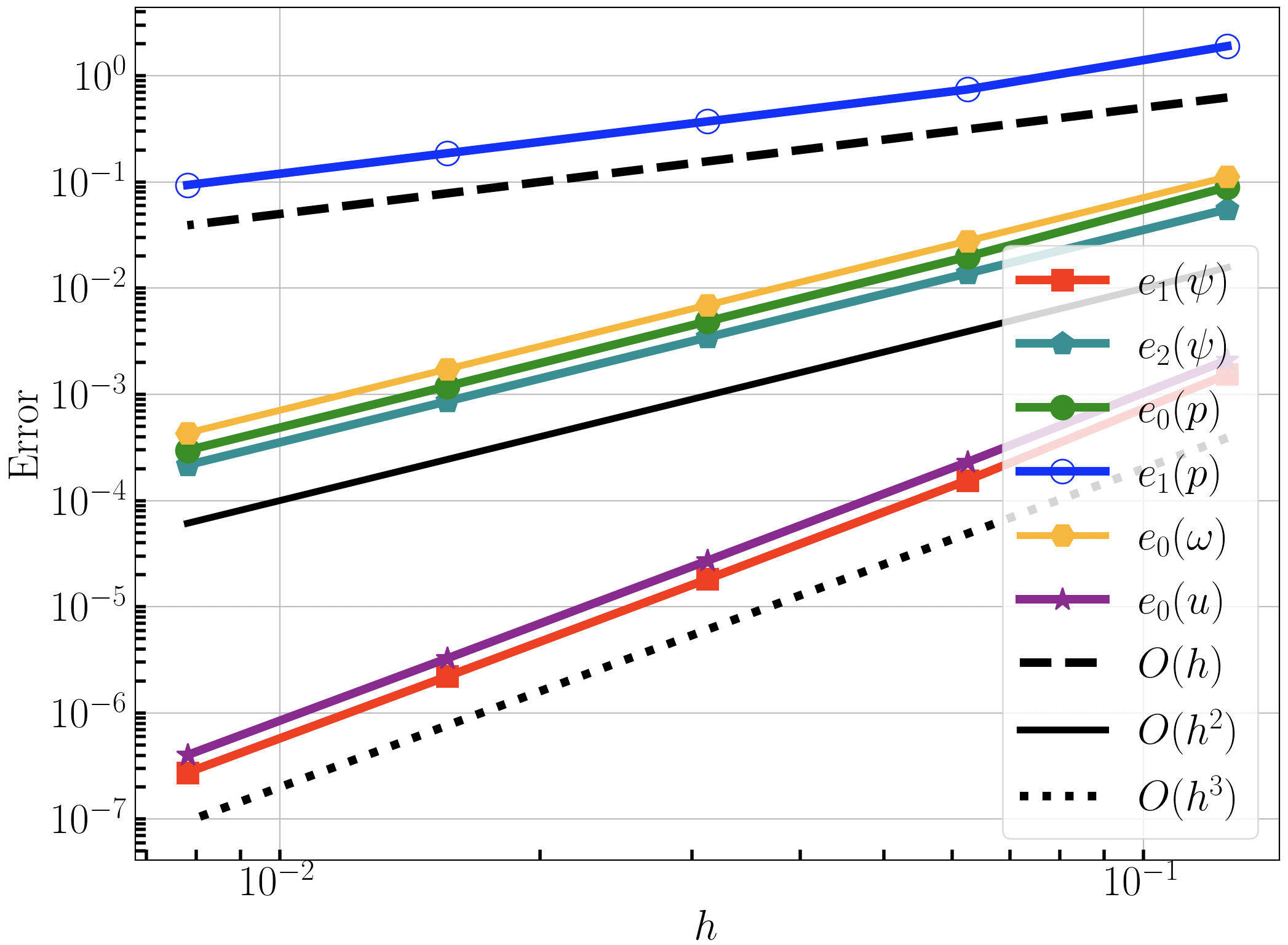} \hfill
				\includegraphics[width=0.49\textwidth]{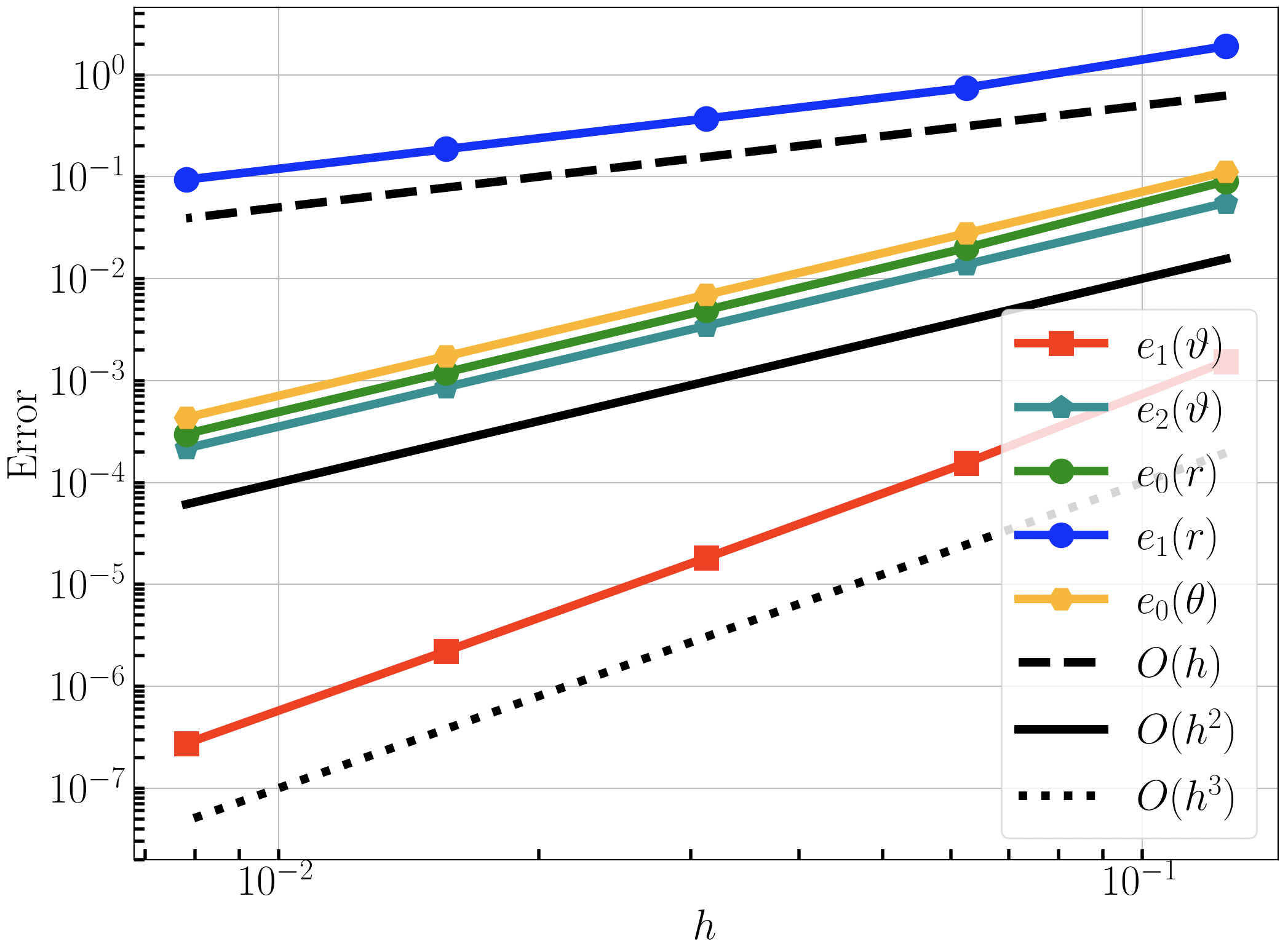}} 
					\\ 
				\subfloat[\textbf{Case 2} \label{state2111}]{\includegraphics[width=0.49\textwidth]{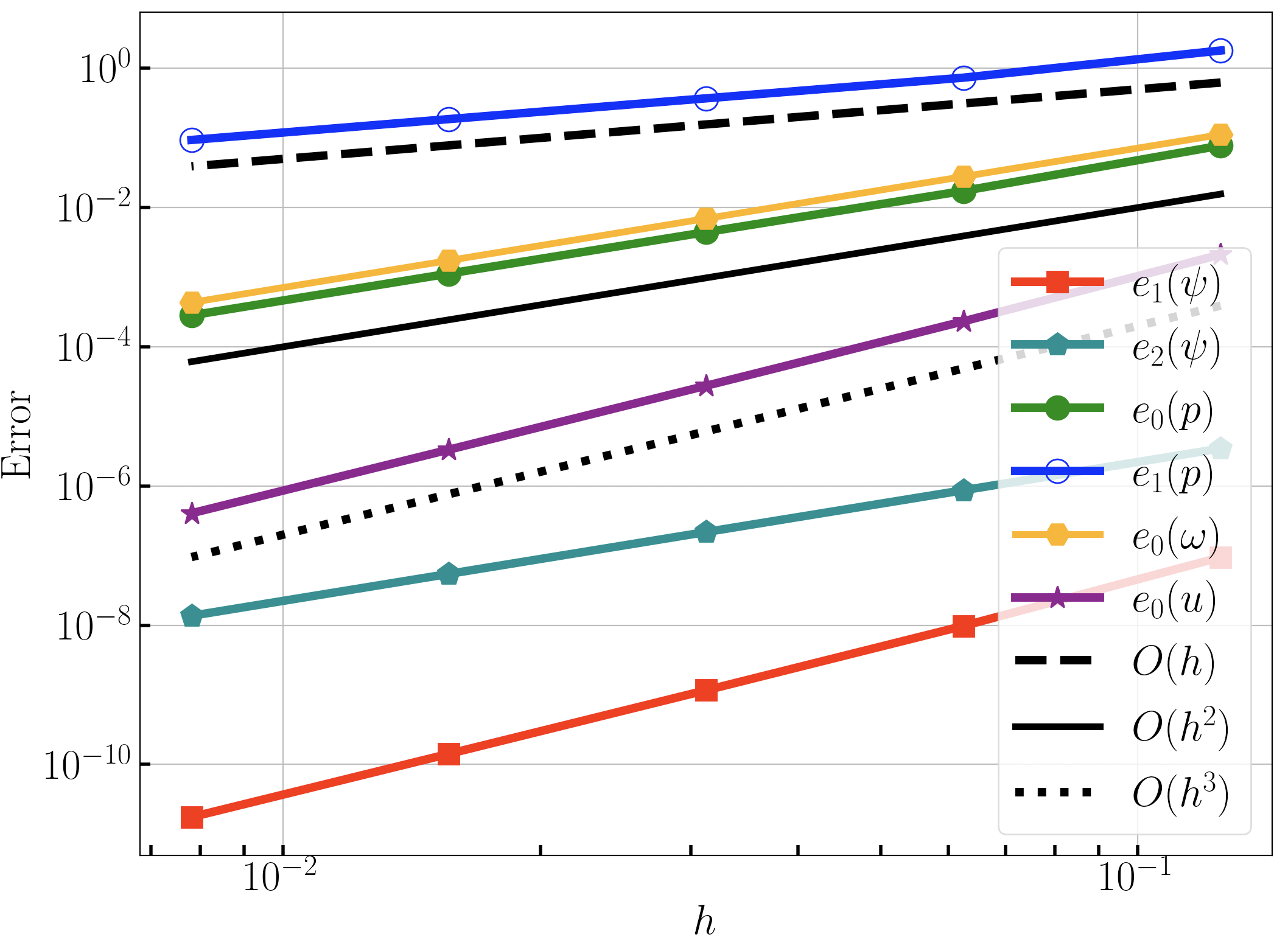} \hfill
				\includegraphics[width=0.49\textwidth]{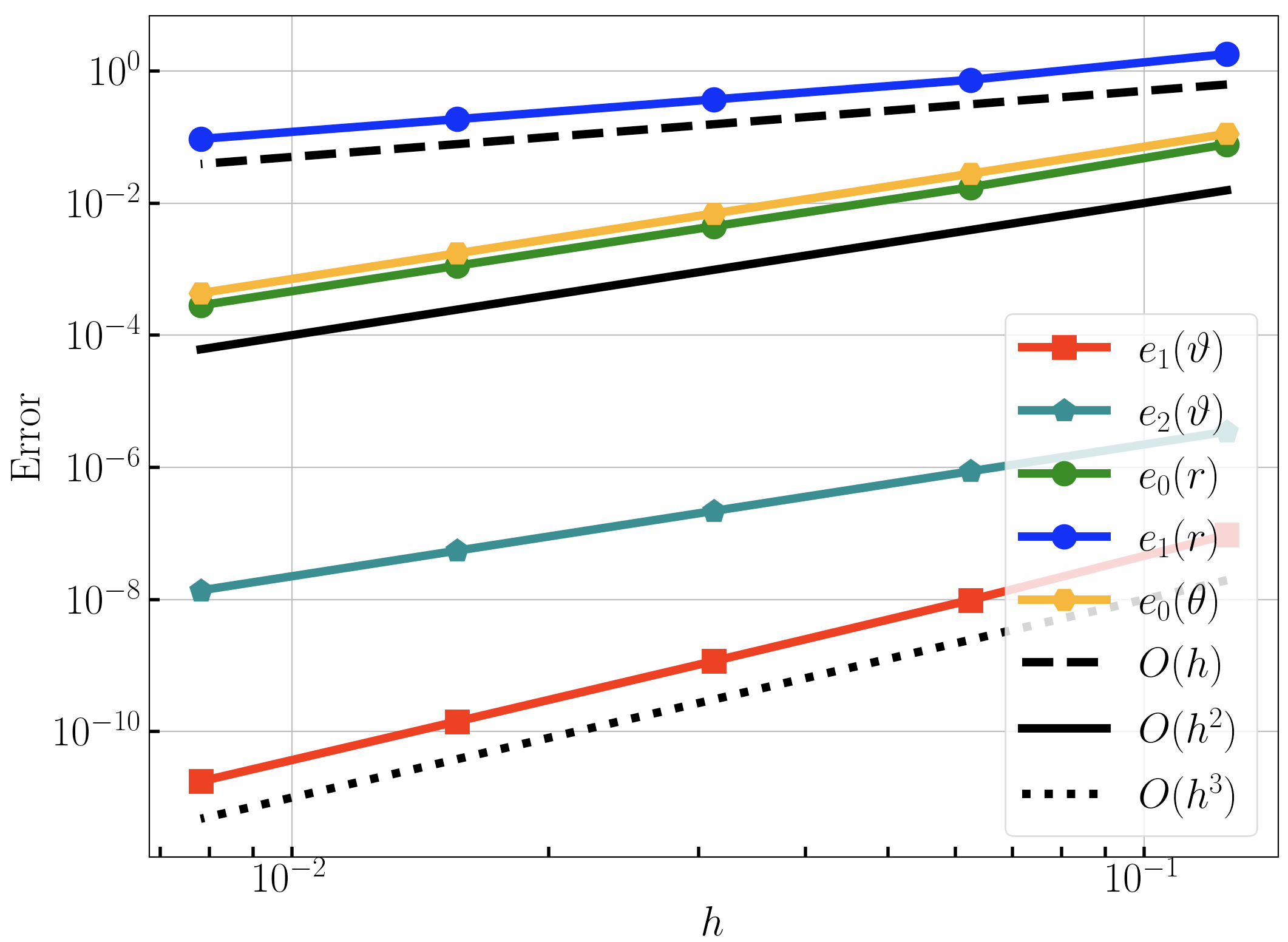}}  
				\caption{Convergence plots using the VEM scheme with $k=3$, meshes $\mscr{T}_{h}^{4}$ and two different viscosity functions \eqref{state211} \textbf{Case 1} and \eqref{state2111} \textbf{Case 2} for Example~\ref{Example 5.2.}.} \label{FIGURE 7}
					\end{figure}
					\begin{figure}
						\centering
						\subfloat[$\psi_h$ \label{psih1}]{\includegraphics[width=0.52\textwidth]{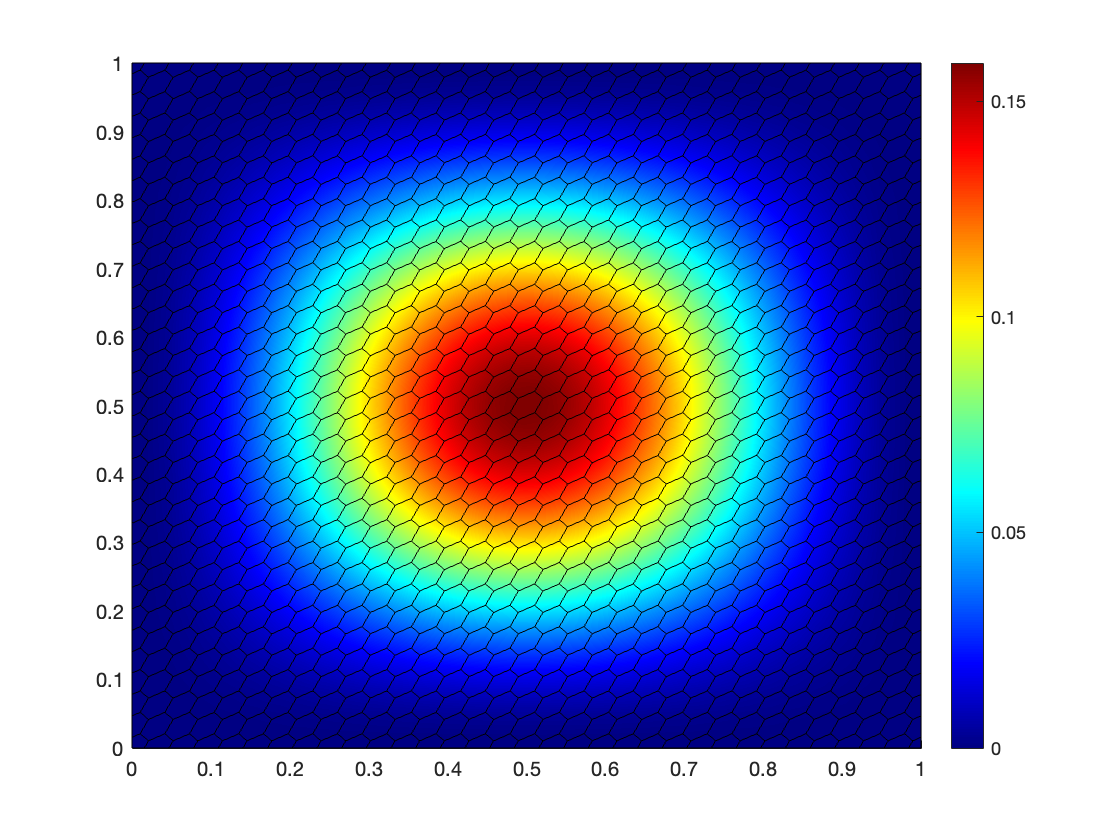}}
						\subfloat[$\boldsymbol{y_{h}}$ \label{yh1}]{\includegraphics[width=0.52\textwidth]{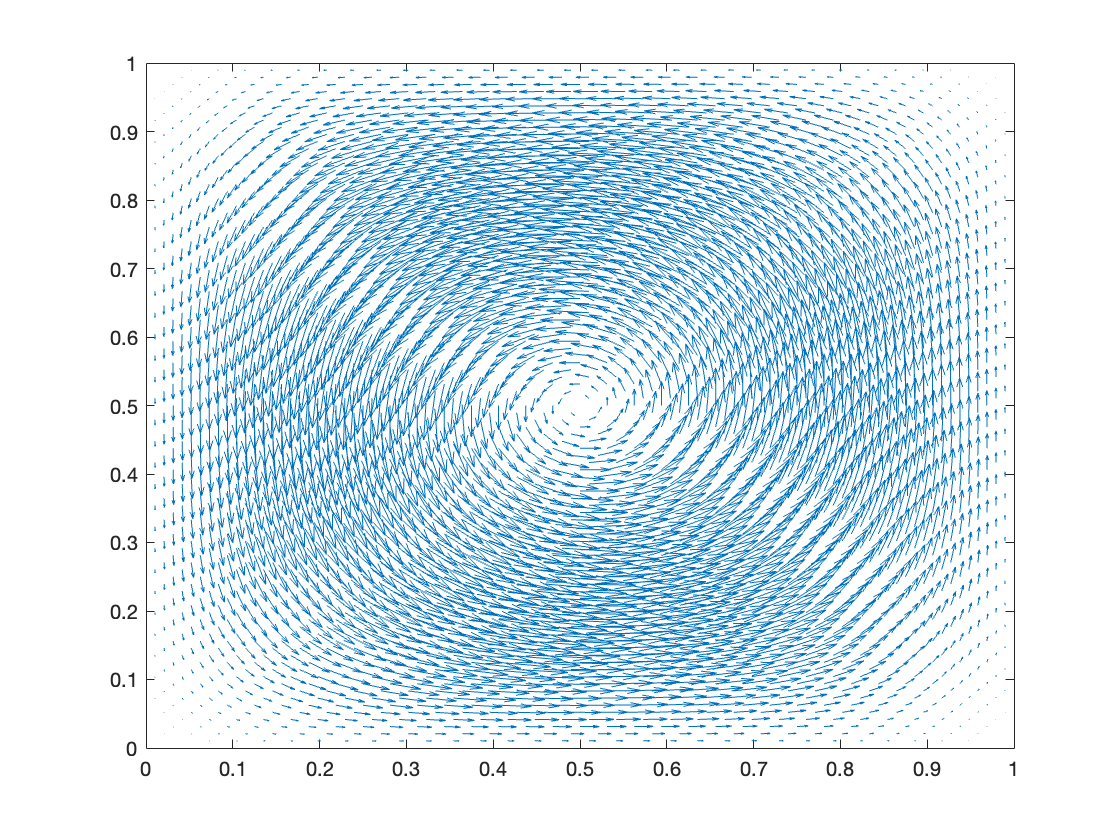}}\\
						\subfloat[$\omega_h$  \label{omega1}]{\includegraphics[width=0.52\textwidth]{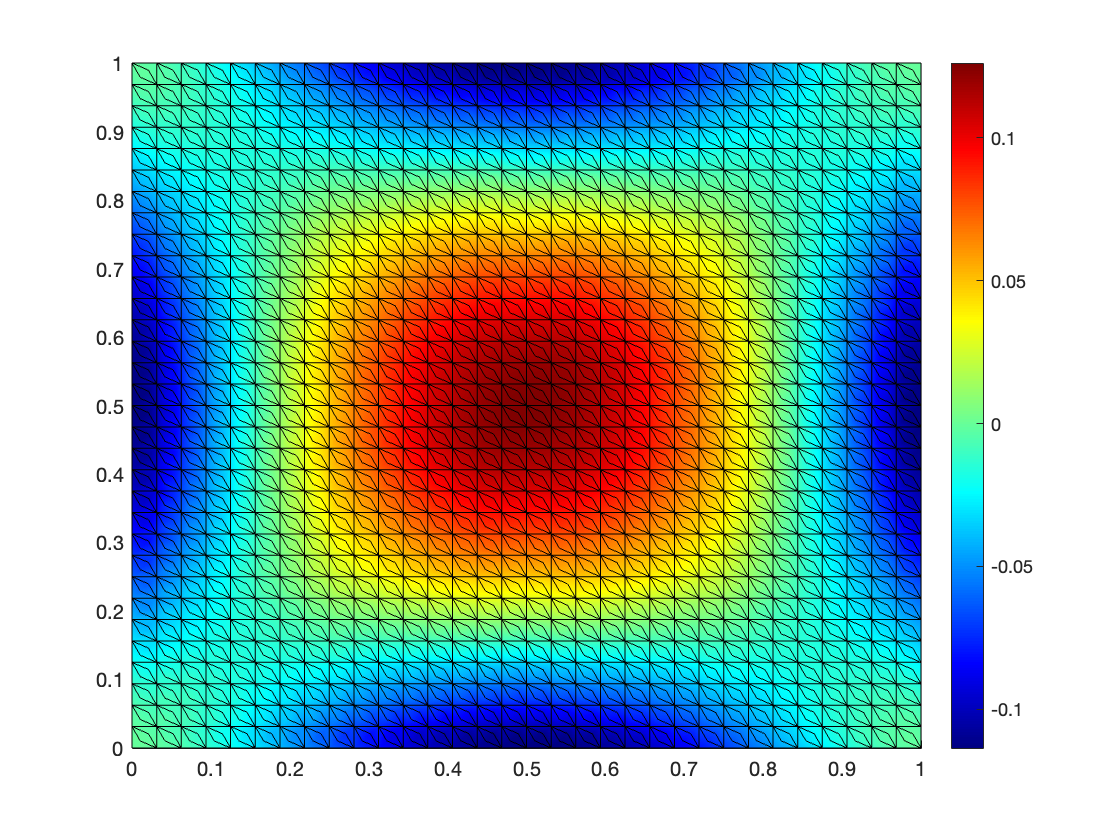}}
						\subfloat[$p_h$  \label{ph2}]{\includegraphics[width=0.52\textwidth]{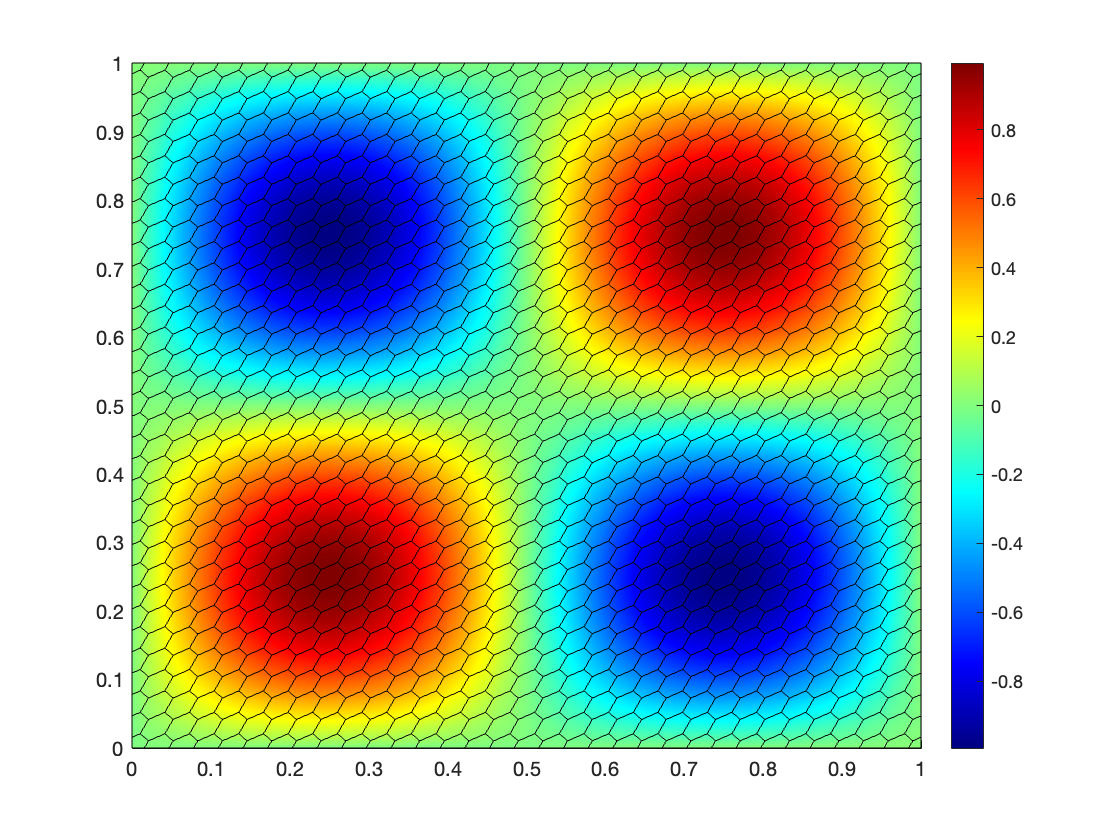}}	\\
						\subfloat[$\boldsymbol{u_{h1}}$  \label{uh111}]{\includegraphics[width=0.52\textwidth]{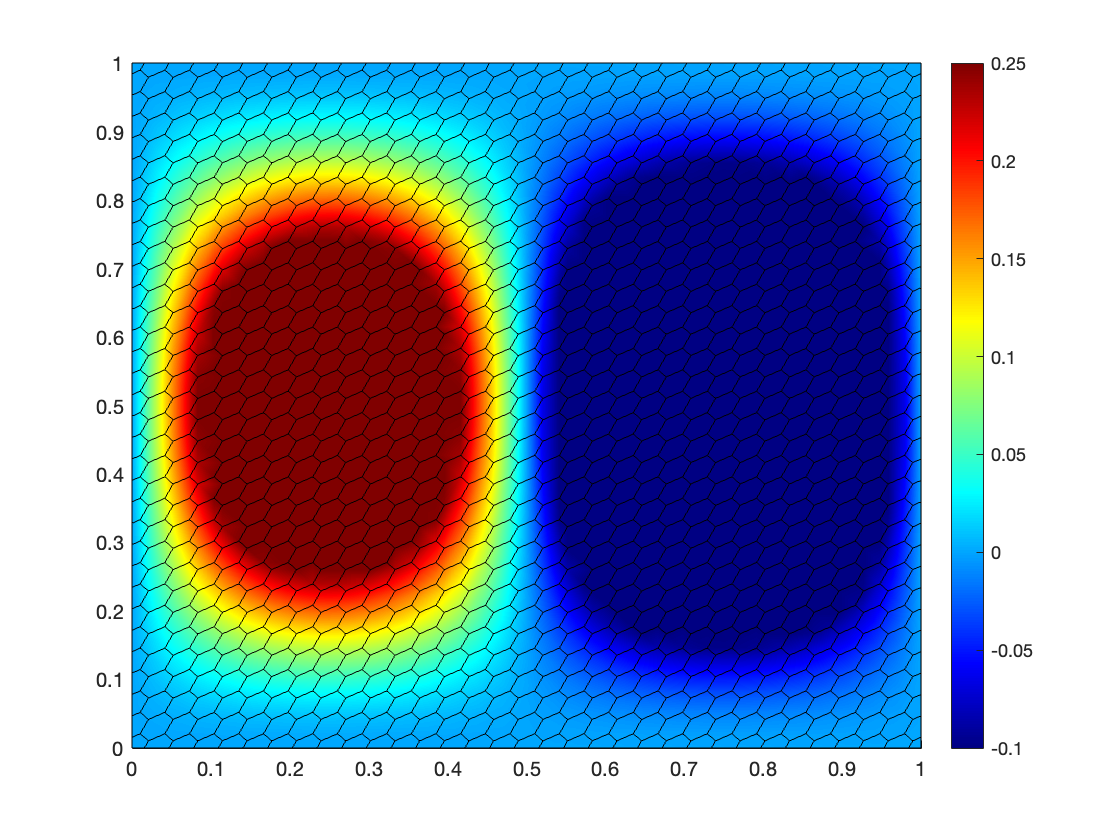}}
						\subfloat[$\boldsymbol{u_{h2}}$  \label{uh121}]{\includegraphics[width=0.52\textwidth]{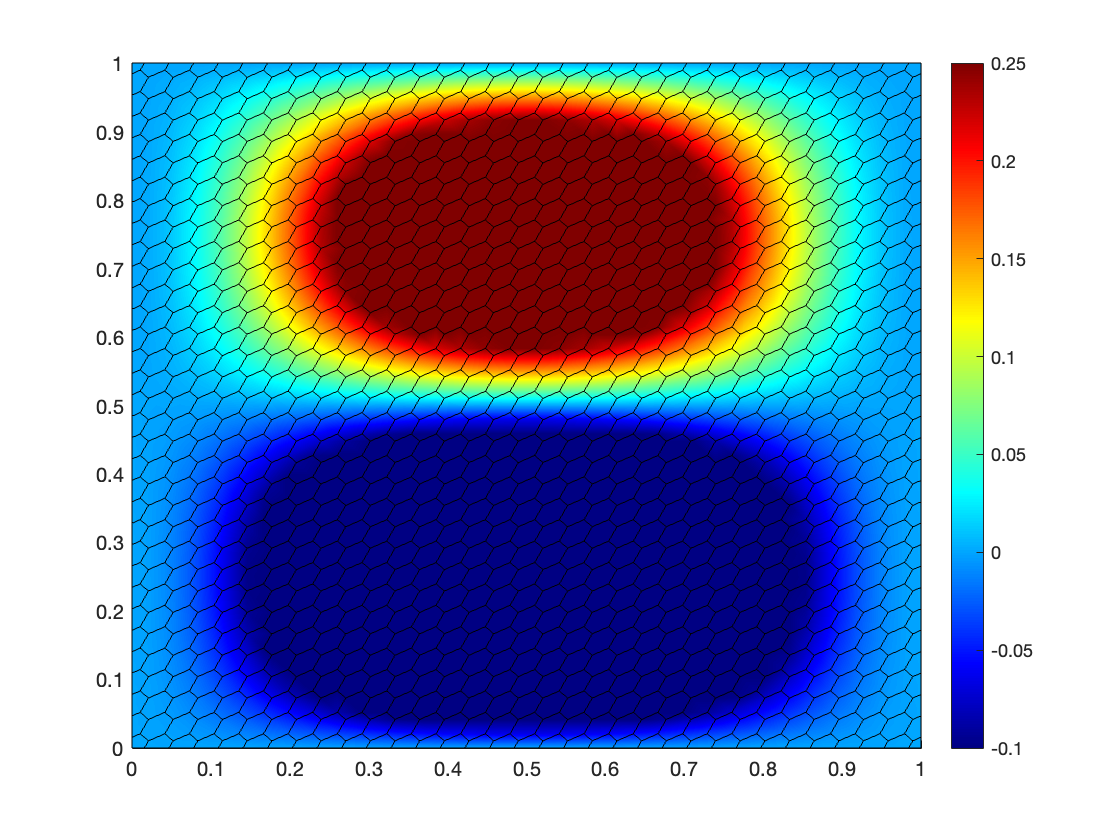}} 
						\caption{Plots of numerical approximations of \eqref{psih1} $\psi_h$ \eqref{yh1} $\boldsymbol{y_{h}}$ \eqref{omega1} $\omega_h$ \eqref{ph2} $p_h$ and control components \eqref{uh111} $\boldsymbol{u_{h1}}$ \eqref{uh121} $\boldsymbol{u_{h2}}$ with $k=3$, $h = 1/32$ and $\mscr{T}_{h}^{4}$ mesh for \textbf{Case 1} in Example~\ref{Example 5.2.}.} \label{FIGURE 8}
					\end{figure}
\begin{example}\label{Example 5.4.}
	\textbf{\textit{Lid-driven cavity test:}} 
		In our last example, we perform the well-known classical 2D lid-driven cavity test. This problem models the motion of a viscous, incompressible fluid inside a rectangular container, where the top lid moves at a constant velocity while the other walls remain fixed with no-slip conditions. The flow takes place in the unit square domain $\Omega = (0,1)^2$. The velocity along the upper lid is prescribed as $\mathbf{y} = (1,0)^T$, while it is set to $\mathbf{y} = (0,0)^T$ for the stationary walls. The boundary conditions expressed in terms of the stream-function are:
			\begin{align*}
				\psi_h = \partial_x \psi_h = 0, \qquad \qquad \text{and} \qquad \qquad \partial_y \psi_h = 1,
			\end{align*}
		 on the upper lid and
			\begin{align*}
				\psi_h = \partial_x \psi_h = \partial_y \psi_h = 0.
			\end{align*}
		on the static walls. The co-state velocity $\w$ is subject to homogeneous Dirichlet data and hence the co-state stream-function $\vartheta_h$. We use the following parameters: $\nu = 0.01$, $\bbeta$ is the velocity obtained from the Stokes problem, $\sigma = 0.01$, $\boldsymbol{u_a} = -(0.5,0.5)^T$, $\boldsymbol{u_b} = (0.5,0.5)^T$, source function $\f = (0,0)^T$ and a stationary target velocity field $\boldsymbol{y_d} = (0,0)^T$.
	\end{example}
 \indent The exact solutions for this problem are not known to us. To approximate the solution, we use the proposed VEM scheme with $k=3$ on uniform square meshes. The stream-function solution in Fig.~\ref{FIGURE 12} highlights the formation of a dominant central vortex along with smaller secondary vortices near the bottom corners. The velocity field shows strong horizontal motion at the top, gradually circulating downward along the sidewalls, creating a recirculating flow. Fig.~\ref{FIGURE 12} also demonstrates the behaviour of co-state discrete stream-function $\vartheta_h$. The vorticity plot in Fig.~\ref{FIGURE 13} provides additional insights into the rotational behaviour of the flow. High vorticity values are observed near the upper lid, particularly at the two corners where the moving lid meets the stationary sidewalls. These regions experience strong shear effects, leading to intense rotational motion. Similarly, the pressure plot shows that the higher values are concentrated in the upper lid corners due to the abrupt change in velocity direction. 
 
		\begin{figure}
	\centering
	\subfloat[$\psi_h$ \label{streamcav}]{\includegraphics[width=0.52\textwidth]{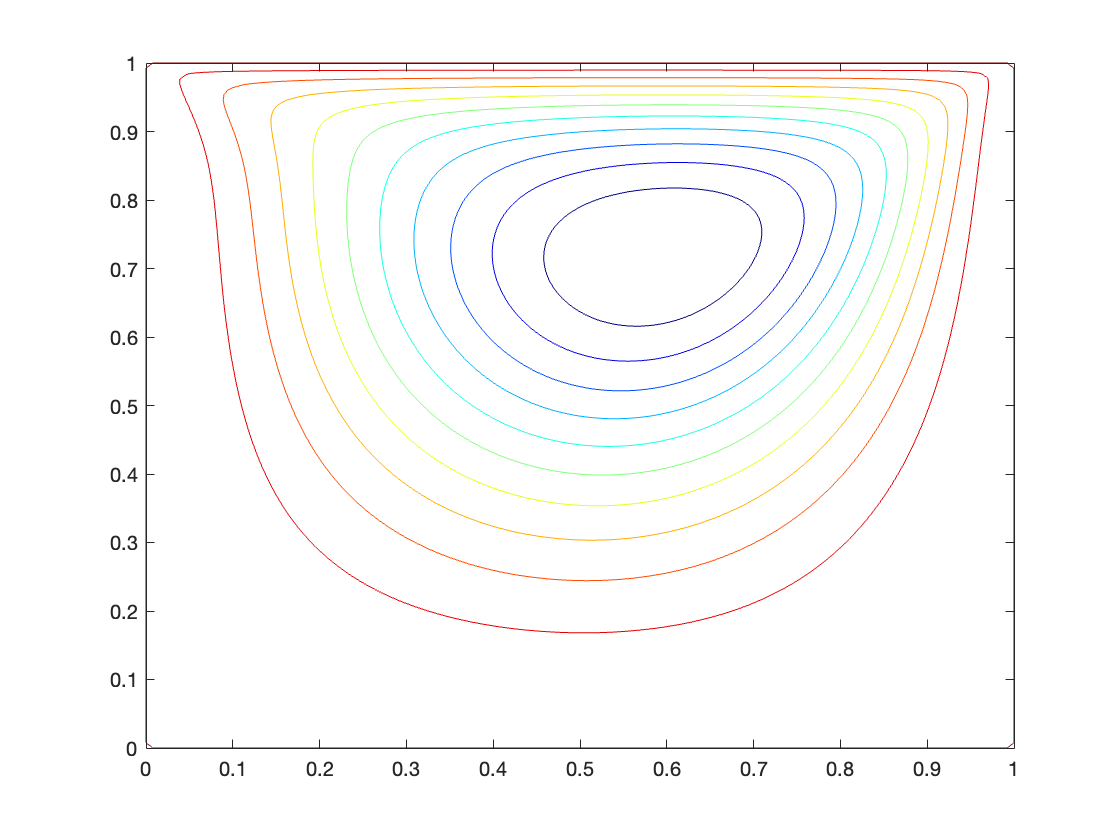}}
	\subfloat[$\vartheta_h$ \label{costreamcav}]{\includegraphics[width=0.52\textwidth]{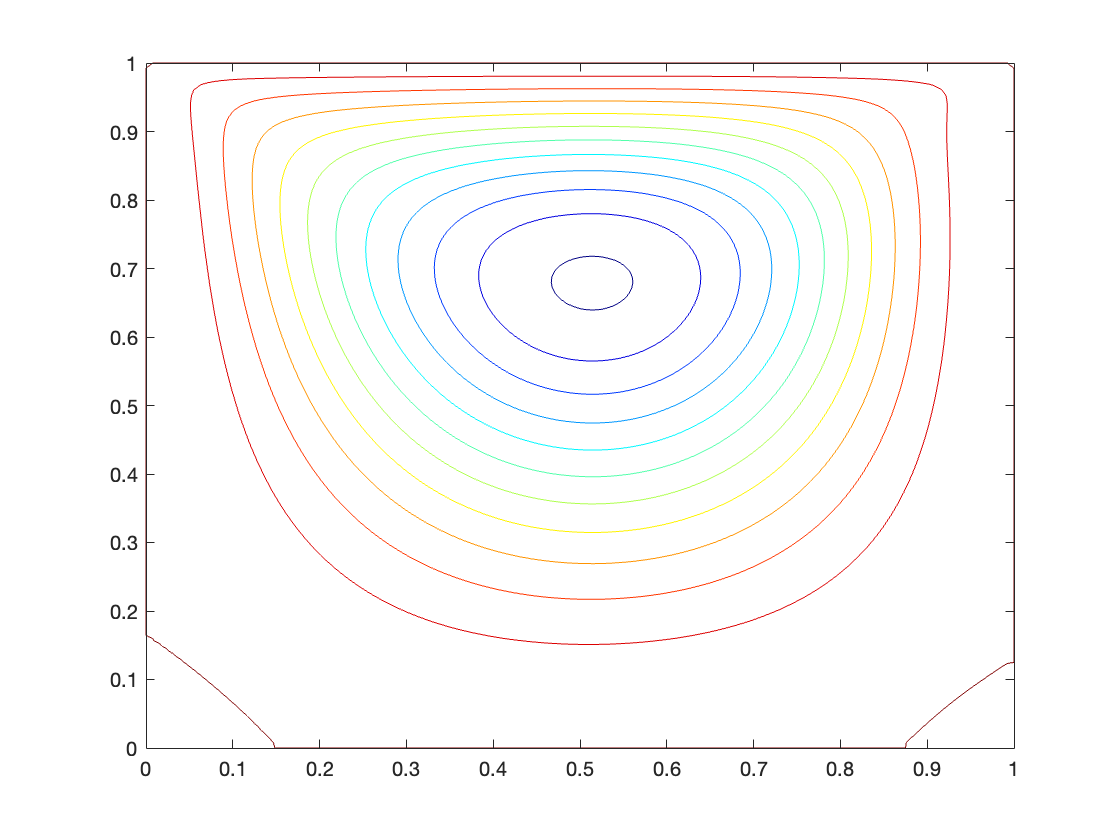}} \\
	\subfloat[$\boldsymbol{y_{h}}$ \label{velocav}]{\includegraphics[width=0.52\textwidth]{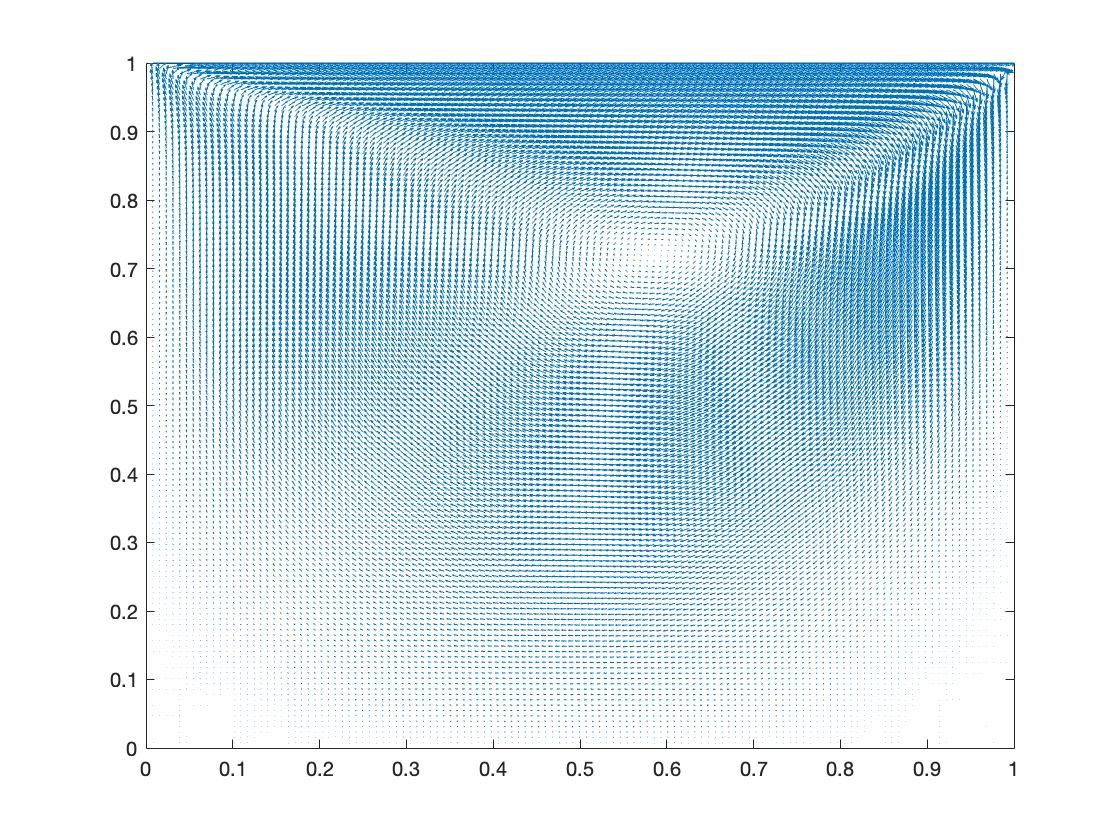}}
	\subfloat[$\boldsymbol{w_{h}}$  \label{veloccav}]{\includegraphics[width=0.52\textwidth]{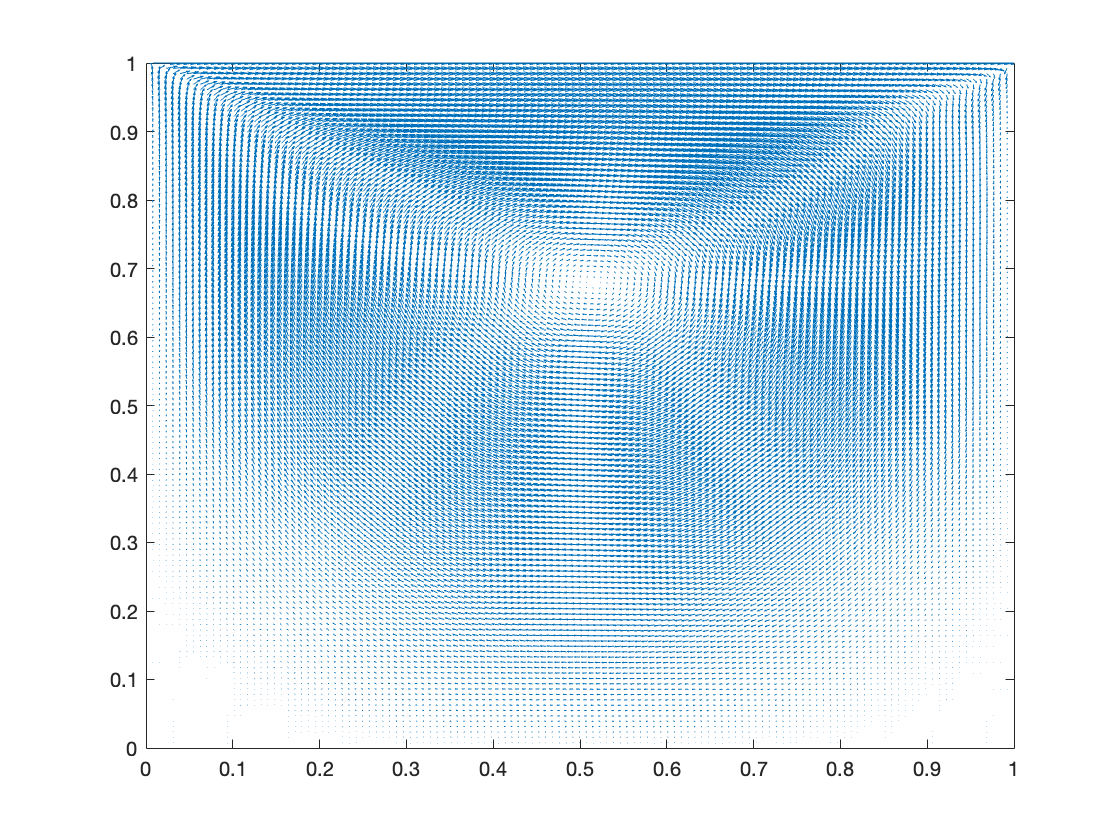}}
	\caption{Streamline plots for the numerical approximations of \eqref{streamcav} $\psi_h$ \eqref{costreamcav} $\vartheta_h$ \eqref{velocav} $\boldsymbol{y_{h}}$ \eqref{veloccav} $\boldsymbol{w_{h}}$ 
		with $h = 1/128$ for the lid-driven cavity Example~\ref{Example 5.4.}.} \label{FIGURE 12}
\end{figure}
		\begin{figure}
	\centering 
	\subfloat[$\omega_h$ \label{vortcavs}]{\includegraphics[width=0.5\textwidth]{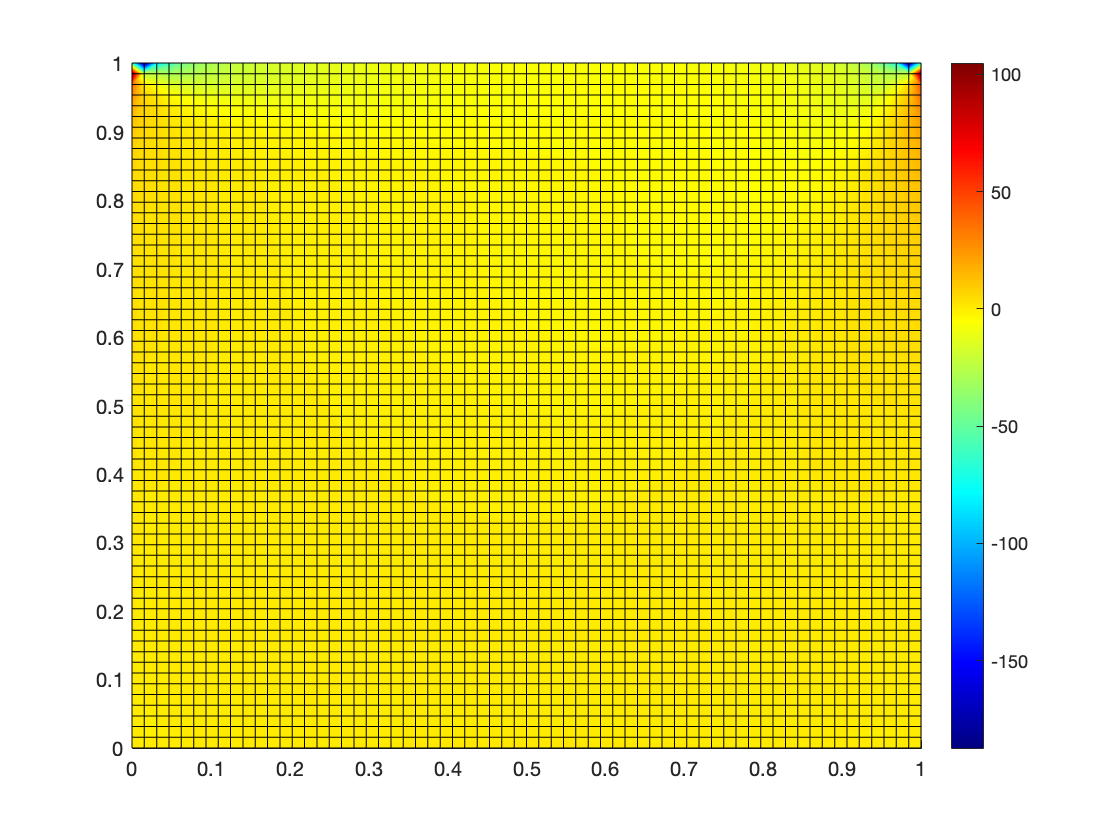}} \hfill
	\subfloat[$\theta_h$  \label{vortcavcs}]{\includegraphics[width=0.5\textwidth]{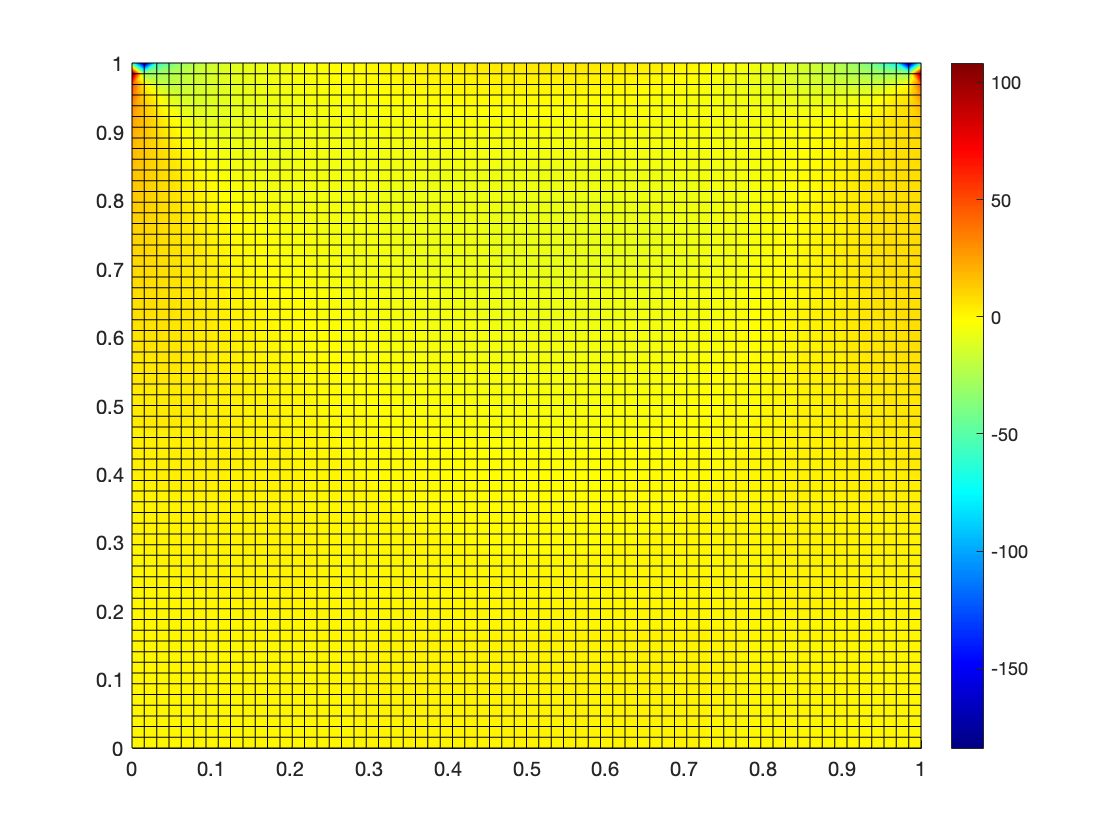}} 
	\\
	\subfloat[$p_h$  \label{phcav1}]{\includegraphics[width=0.5\textwidth]{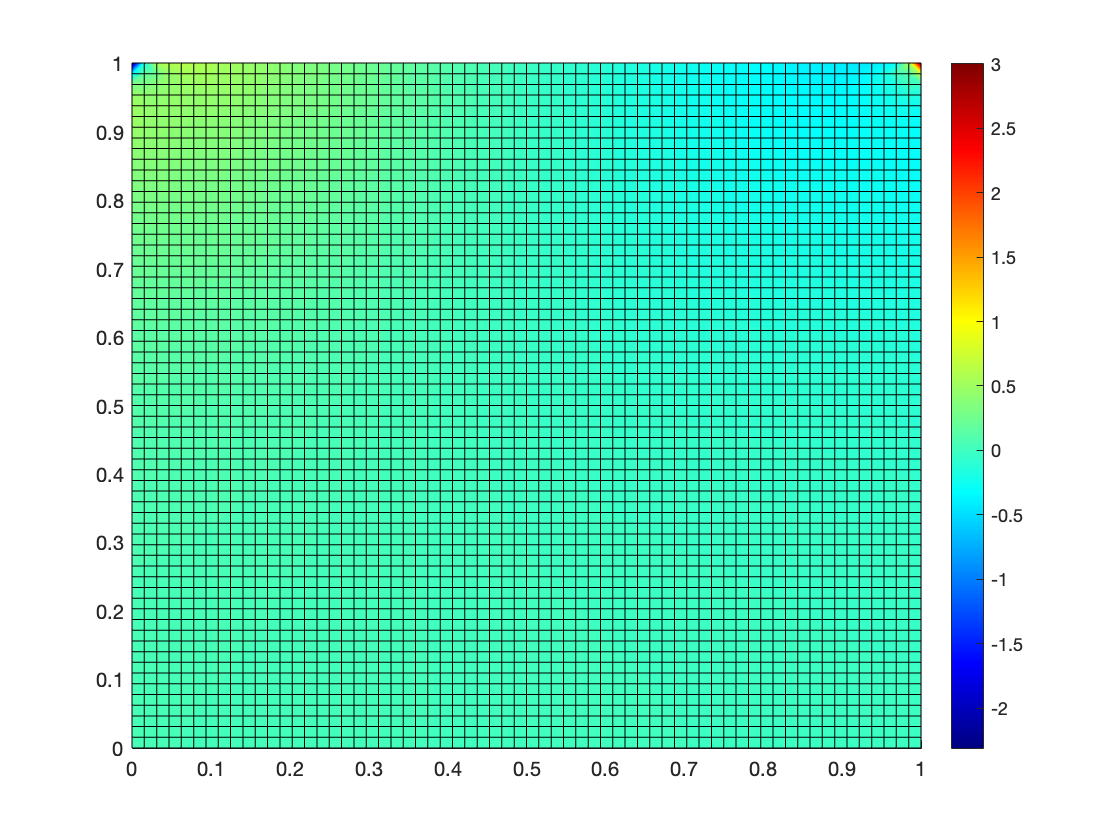}} \hfill
	\subfloat[$r_h$  \label{rhcav1}]{\includegraphics[width=0.5\textwidth]{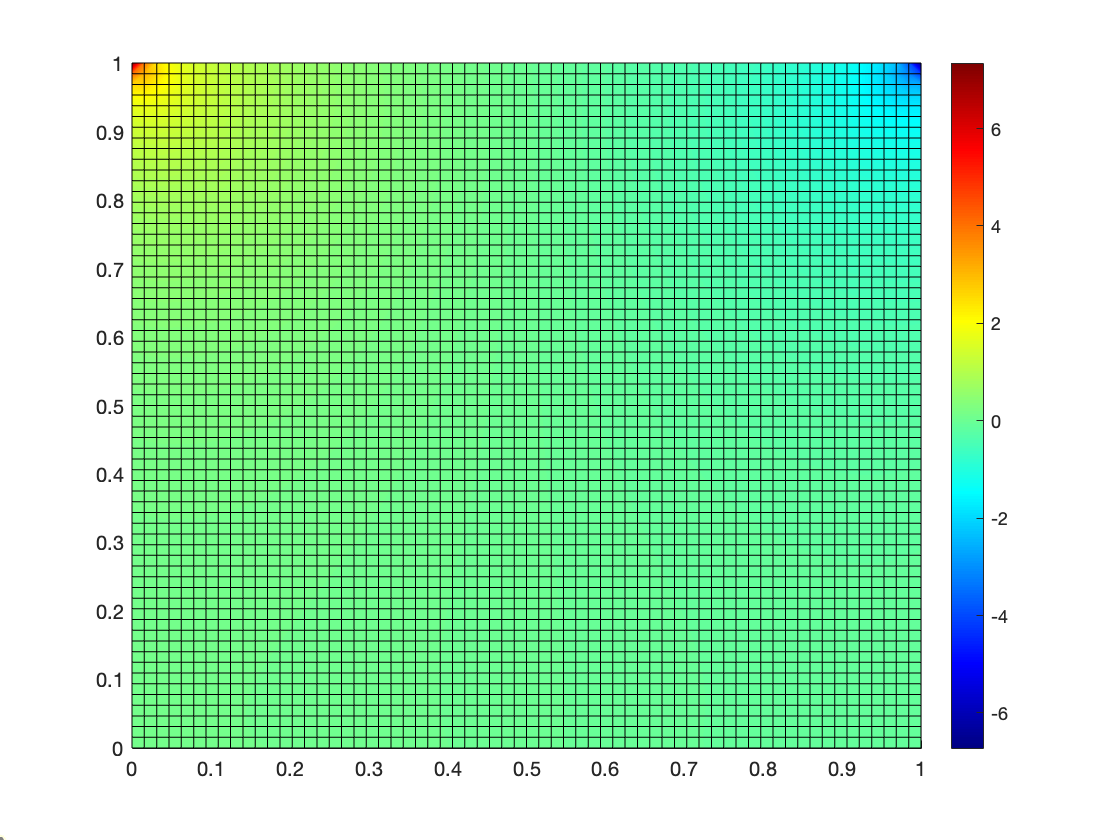}}
	\caption{Plots of state and co-state vorticity and pressure, respectively, with $h = 1/64$ for Example~\ref{Example 5.4.}.} \label{FIGURE 13}
\end{figure}
\section{Concluding Remarks}\label{SEC:CONCLUSIONS}
We have proposed a $C^1$-conforming virtual element scheme with a stream-function formulation for the numerical approximation of optimal control problems governed by generalized Oseen equations on simply connected polygonal domains. Optimal error estimates are derived for the stream-function, velocity, vorticity, and pressure, which are recovered using suitable polynomial projections. For $k=2$, we prove optimal convergence $O(h^2)$ in the $L^2$-norm for the control variable, while for $k=3$, the convergence improves to $O(h^3)$, provided that the stream-function $\psi, \vartheta \in H^4(\Omega)$. The derived results have been thoroughly validated on various polygonal meshes, demonstrating their certain robustness and consistency across a range of viscosity variations. We present various numerical experiments to validate theoretical results using different polygonal meshes. Additionally, we numerically explore the feasibility of incorporating variable viscosity. Our analysis also includes the optimal control problems governed by the Brinkman and Stokes equations. 
\section*{Acknowledgments}
The authors are sincerely grateful to the anonymous referees for their careful reading and insightful comments. Their constructive feedback has significantly improved the clarity, quality, and presentation of this work.
\bibliography{ref11}
\bibliographystyle{ws-m3as}

\end{document}